\documentclass[12pt]{amsart}
\usepackage{srcltx}
\usepackage{epsfig}

\usepackage{varioref}
\usepackage{hyperref}

\allowdisplaybreaks

\DeclareMathOperator{\End}{End}
\DeclareMathOperator{\Hom}{Hom}

\DeclareMathOperator{\ev}{ev}
\DeclareMathOperator{\coev}{coev}
\DeclareMathOperator{\tev}{\widetilde{ev}}
\DeclareMathOperator{\tcoev}{\widetilde{coev}}

\newcommand\be{\begin{equation}}
\newcommand\ee{\end{equation}}

\theoremstyle{plain}

\newtheorem{theorem}{Theorem}
\newtheorem{Thm}{Theorem}

\newtheorem{lemma}[theorem]{Lemma}
\newtheorem{proposition}[theorem]{Proposition}
\newtheorem{corollary}[theorem]{Corollary}

\theoremstyle{definition}

\newtheorem{remark}[theorem]{Remark}
\newtheorem{definition}[theorem]{Definition}

\numberwithin{equation}{section}
\numberwithin{theorem}{section}
	
\newcommand*{\refh}[2]{\hyperref[#2]{#1~\ref{#2}}} 

\newcounter{ourcount}
\advance\textheight\topskip
\usepackage{amsmath,amsfonts,amsthm,amssymb,amscd}
\usepackage{mathtools} 
\allowdisplaybreaks

 \usepackage{graphicx}
 \usepackage[curve,matrix,arrow,color]{xy}

 \usepackage[usenames,dvipsnames]{color}
 \xyoption{line}

\usepackage[mathscr]{eucal}
\usepackage[OT2,OT1]{fontenc}

\newcommand{\id}{\mathrm{id}}

\newcommand{\ok}{{\ensuremath{\Bbbk}}}

\newcommand{\cat}{\mathcal{C}}

\newcommand{\vect}{\mathrm{{\bf vect}}}

\newcommand{\rep}{\mathrm{\bf Rep}\,}

\newcommand{\tensor}{\otimes}

\newcommand{\one}{\boldsymbol{1}}%{1\kern-4pt 1}

\newcommand{\Urest}{\overline{U}_{\mbox{}\!\!q}\,\mathfrak{sl}_2}

\newcommand{\UA}{{A}}

\renewcommand{\t}{{\mathsf{t}}}

\newcommand{\tr}{\operatorname{tr}}
\newcommand{\Id}{\operatorname{id}}
\newcommand{\Proj}{\mathcal{P}}
\newcommand{\Amod}{A\text{\rm{-mod}}}

\newcommand{\HH}{\operatorname{HH}_0}
\newcommand{\Upmod}{\Urest\text{\rm{-pmod}}}
\newcommand{\Apmod}{A\text{\rm{-pmod}}}

\newcommand{\PP}{\mathcal{P}}

\newcommand{\kk}{\ok}
\newcommand{\UH}{{H}}

\DeclareMathOperator{\Hpmod}{{\it H}-pmod}
\DeclareMathOperator{\Hmod}{{\it H}-mod}
\DeclareMathOperator{\Ker}{Ker}

\newcommand{\ffrac}[2]{\mbox{\footnotesize$\displaystyle\frac{#1}{#2}$}}

\newcommand{\qbin}[2]{\mathchoice%
  {{\qbinm{#1}{#2}}}{\qbinmm{#1}{#2}}%
  {\qbinmm{#1}{#2}}{\qbinmm{#1}{#2}}}
\newcommand{\qbinm}[2]{\mbox{\footnotesize$\displaystyle
    \genfrac{[}{]}{0pt}{}{#1}{#2}$}}
\newcommand{\qbinmm}[2]{\genfrac{[}{]}{0pt}{}{#1}{#2}}

\newcommand{\Bnm}{\mathsf{B}}

\newcommand{\coint}{{\boldsymbol{c}}}
\newcommand{\rint}{{\boldsymbol{\mu}}}
\newcommand{\lint}{{\boldsymbol{\mu}}^l}

\newcommand{\pivot}{{\boldsymbol{g}}}
\newcommand{\rintg}{\rint_{\pivot}}
\newcommand{\lintg}{\rint^l_{\pivot^{-1}}}

\newcommand{\comod}{{\boldsymbol{a}}}
\newcommand{\ribbon}{{\boldsymbol{v}}}
\newcommand{\sqs}{{\boldsymbol{u}}}
\newcommand{\balance}{{\boldsymbol{g}}}
\newcommand{\wt}{\mathsf{wt}}

\newcommand{\gp}{\ .}
\newcommand{\gc}{\ ,}
\newcommand{\qp}{\quad .}
\newcommand{\qc}{\quad ,}

\newcommand{\ipic}[3][-0.5]{\raisebox{#1\height}{\scalebox{#3}{\includegraphics{#2.pdf}}}}
\newcommand{\ipicc}[3][-0.4]{\raisebox{#1\height}{\scalebox{#3}{\includegraphics{#2.pdf}}}}
\renewcommand{\arg}{\,-\,}

\newcommand{\eps}{\epsilon}

\newcommand{\rt}{p}
\newcommand{\UU}{\overline{U}_{\mbox{}\!\!q}\,\mathfrak{g}}

\newcommand{\op}{\mathrm{op}}
\newcommand{\Mat}{\mathrm{Mat}}

\newcommand{\catCY}{\mathcal{D}}

\newcommand{\proj}{\mathsf{Proj}}
\newcommand{\tP}{\hat{P}}
\newcommand{\B}{{\mathsf{B}}}
\newcommand{\A}{{\mathsf{A}}}

\begin{document}

\title{Modified trace is a symmetrised integral}
\author[A.\,Beliakova]{Anna Beliakova}
\address{University of Zurich, I-Math, Winterthurerstrasse 190, CH-8057 Zurich, Switzerland.} 
\email{anna@math.uzh.ch}
\author[C.\,Blanchet]{Christian Blanchet}
\address{ Universit\'e Paris Diderot, IMJ-PRG, UMR 7586 CNRS,  F-75013, Paris, France.}
\email{Christian.Blanchet@imj-prg.fr}
\author[A.M.\,Gainutdinov]{Azat M.\,Gainutdinov}
\address{Institut Denis Poisson, CNRS, Universit\'e de Tours, Universit\'e d'Orl\'eans, Parc de Grandmont, 37200 Tours, France. }
\email{azat.gainutdinov@lmpt.univ-tours.fr}

\begin{abstract} 
A modified trace for a finite  $\ok$-linear pivotal category is a family of linear forms on endomorphism spaces of projective objects
which has cyclicity  and   so-called partial trace properties. 
The modified trace provides a meaningful generalisation of the categorical trace to non-semisimple categories and allows to construct interesting topological invariants.
We show that  a non-degenerate modified trace defines  a compatible with  duality Calabi-Yau structure on the subcategory of projective objects.
We prove, that for any finite-dimensional unimodular pivotal Hopf algebra over a field $\ok$,
 a modified trace is determined by a  symmetric linear form  on the Hopf algebra constructed from an integral.
More precisely, we prove that  shifting with the pivotal element defines
an isomorphism between the space of right integrals, which  is known to be $1$-dimensional, and the space of modified traces.
This result allows us to compute
modified traces for all simply laced restricted quantum groups
at roots of unity.

\end{abstract}

\maketitle

\vskip 3mm

\mbox{}\quad modified traces, quantum groups.\\
\mbox{}\quad\textit{AMS codes:}  18D10, 16T05, 17B37.

\vskip 5mm
\section{Introduction}
This paper establishes a one-to-one correspondence  between two {\em a priori} very different notions in the  theory
of  finite-dimensional pivotal Hopf algebras.
One of them is the well-known
linear form on the Hopf algebra $H$, called
 {\em integral}, 
and the other  is a certain trace function
on the category of projective $H$-modules, called
 {\em modified trace}.
Our new  correspondence allows to transfer 
properties (such as existence and uniqueness) from the integral to the modified trace, 
and most importantly to compute the modified traces explicitly. The main application of  modified traces 
is in constructing topological invariants based on \textsl{non-semisimple} categories. 
In particular, our correspondence allows to use the integral theory  to construct fully featured non-semisimple TQFTs \cite{RGP, RGP1,DGGPR} extending the ones
of Kerler and Lyubashenko \cite{Kerler:2001}.
 The importance
of non-semisimple theories in physics was also stressed recently in
the work of Gukov, Pei, Putrov and Vafa \cite{GPPV}. In mathematics, such theories are expected to  provide the right
framework for
 categorification of 3-manifold invariants~\cite{GM}. 

Let us now introduce our main players.

\subsection*{Integral}
The integral or dually cointegral
 can be thought as  analogs of the Haar measure on a compact group and the invariant $\sum_{g\in G} g$ in the group algebra
 of a finite group, respectively. 
 If non-zero, they generate one-dimensional ideals in
 the algebra and its dual. 
 The integral has important topological applications.
It plays the role of a Kirby color in the 
 Hennings construction~\cite{Hennings} 
  of 3-manifold invariants
 generalizing those of Reshetikhin-Turaev.
 
Let $\UH=(\UH, m,\one, \Delta,\epsilon,S)$ be a Hopf algebra  over a field $\ok$.
A {\em right integral} on $\UH$ is a linear form $\rint\colon \UH\to \ok$ satisfying 
\begin{equation}\label{eq:rint-def}
  (\rint\tensor\id)\Delta(x)=\rint(x)\one\ \quad
  {\text {for any}}\quad x\in \UH .
\end{equation}
  Analogously, a {\em left integral}
$\lint\in \UH^*$ satisfies 
 \be \label{eq:lint-def}
 (\id\tensor\lint)\Delta(x)=\lint(x)\one\ 
 \quad
  {\text {for any}}\quad x\in \UH .
\ee
 If $H$ is finite-dimensional,
the space of solutions of these
equations  is known to be 1-dimensional. 
A  \textit{pivotal} Hopf algebra  is a pair $(\UH,\pivot)$, where
the pivot $\pivot\in \UH$
is a group-like element implementing $S^2$, i.e.
$S^2(x)=\pivot x\pivot^{-1}$ for any $x\in H$.
   
   A \textit{symmetrised right} integral $\rintg$ 
on $(\UH,\pivot)$ is defined by 
\be\label{eq:mug-gen}
\rintg(x):=\rint(\pivot x) \quad{\text {for any}} \quad
x\in \UH\  .
\ee
Analogously, a {\em symmetrised left} integral is
\be
\lintg(x):=\lint(\pivot^{-1} x)\quad{\text{for any}}  \quad x\in \UH\ .
\ee
We call  a pivotal Hopf algebra $(\UH,\pivot)$  \textit{unibalanced} if
its symmetrised right integral is also left.

Dually, 
a left (resp. right) 
\textit{cointegral} in $\UH$ is an element $\coint\in H$ such that
$x\coint=\epsilon(x) \coint$ (resp. $\coint x=\epsilon(x) \coint$) for all $x\in \UH$. 
Non-trivial right and left cointegrals 
are unique up to scalar~\cite{LarsonSweedler}.
We call a Hopf algebra {\em unimodular} if its right   cointegral is also left.

In the unimodular case, the  symmetrised
 integrals define {\em symmetric}  linear forms on $H$, i.e.
\be  \rintg(xy)=\rintg(yx) \quad{\text{and}}\quad
 \lintg(xy)=\lintg(yx), 
\ee
which are also non-degenerate
(compare with Proposition~\ref{lemma:mug-sym} below).

\subsection*{Modified trace}
Our second main player is the {\em modified trace} introduced in~\cite{GPV, GKP}.
Unlike the integral, it is defined
on the category of modules and  motivated by topology.
For braided pivotal categories,  the modified
trace allows a non-zero evaluation of the Reshetikhin-Turaev type invariants 
on links colored with  projective objects, 
even if the category is not semisimple.
We will work with pivotal categories without  braiding assumptions and refer to  Section~\ref{sec:pivot-str} for 
detailed definitions and graphical conventions.

Let $\cat$ be a $\ok$-linear 
 pivotal category.
Given $V,W\in \cat$ and $f\in \End_{\cat}(W\otimes V)$, let $\tr_W^l (f)$ and $\tr_V^r (f)$ be the left and right {\it partial} traces defined as follows
\begin{align}\label{E:partialL}
&\tr^l_W(f)=(\ev_W \otimes \Id_V)\circ(\Id_{W^*} \otimes f)\circ(\tcoev_W \otimes \Id_V) \;= \;
 \;\ipicc{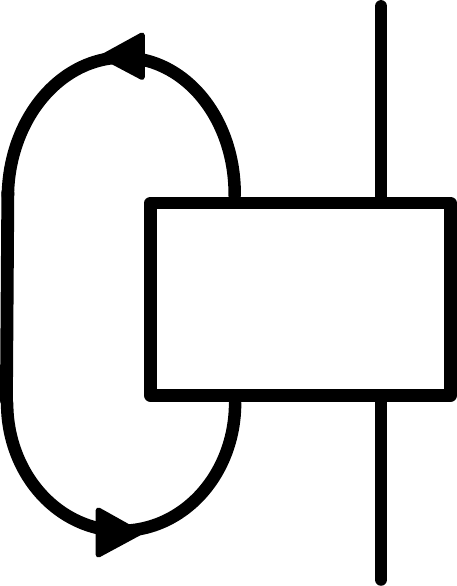}{.16} 
 \put(-19,21){{\tiny $W$}}  \put(-16,1){{\footnotesize $f$}}\put(-10,29){{\tiny $V$}}
 \;\in\; \End_\cat(V),  \\
&\tr^r_V(f)=(\Id_W \otimes \tev_V)\circ(f \otimes \Id_{V^*})\circ(\Id_W \otimes \coev_V) \; = \; \put(15,-18){{\tiny $V$}}
\;\ipicc{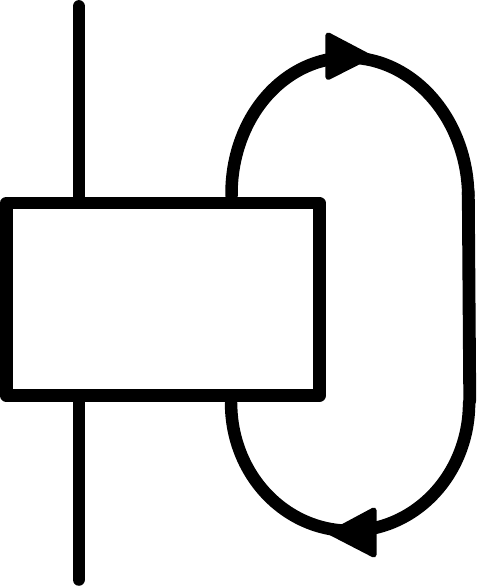}{.16} \put(-27,1){{\footnotesize $f$}}\put(-34,-25){{\tiny $W$}} \;\in \;\End_{\cat}(W)\gp \label{E:PartialLRtrace}
\end{align}
The main example of a pivotal category used in this paper is the category
$\Hmod$ of finite-dimensional left modules over 
a pivotal Hopf algebra $(H,\pivot)$.
In $\Hmod$
 the left  (co)evaluation morphisms are those for vector spaces while the right ones are defined using the pivot. 

Setting $W=\one$  in \eqref{E:PartialLRtrace} and assuming
$\End_\cat(\one) = \ok$, we get the 
definition of the (right) categorical trace 
\be\label{eq:cat-tr-def}
\tr^{\cat}_V(f):= \tev_V\circ(f\tensor\id)\circ\coev_V\; \in \; \ok .
\ee
Analogously, assuming $V=\one$ in
\eqref{E:partialL}, we get its left version $^{\cat}\!\tr_V(f)$.

We assume now that tensor product in $\cat$ is exact and
let $\proj(\cat)$ be the tensor ideal of projective objects in $\cat$.
A {\em right (left) modified trace} on $\proj(\cat)$ is  a family of linear functions 
\be
\{\t_P\colon \End_\cat(P)\rightarrow \ok \}_{P\in \proj(\cat)}
\ee
satisfying cyclicity and  right (left) partial
trace properties formulated below.
\begin{description}
\item[ \sc Cyclicity] If $P,P'\in \proj(\cat)$ then for any morphisms \mbox{$f\colon P\rightarrow P' $} and $g\colon P'\rightarrow P$
\begin{equation}\label{E:Tracefggf}
\t_P(g \circ f)=\t_{P'}(f  \circ g)\gp
\end{equation} 
\item[\sc Right partial trace property] If 
$P\in \proj(\cat)$
 and $V\in \cat$ then
\begin{equation}\label{rpartial}
\t_{P\otimes V}\left(f \right)
=\t_P \bigl( \tr^r_V(f)\bigr)
\end{equation}
for any \mbox{$f\in \End_\cat(P\otimes V)$}. 
\item[\sc Left partial trace property] If $P\in \proj(\cat)$
 and $V\in \cat$ then
\begin{equation}\label{lpartial}
\t_{V\otimes P}\left(f\right)
=\t_P\bigl( \tr^l_V(f)\bigr) \ 
\end{equation}
for any \mbox{$f\in \End_\cat(V\otimes P)$}. 
\end{description}
A left and right modified trace  will be called {\em modified trace}.

It is then clear from the definition that the right categorical trace is also  a right modified trace, and analogously for the left. The trace $\tr^{\cat}$ is  non-zero on $\proj(\cat)$ if and only if $\cat$ is semisimple. However, there are many examples of non-semisimple categories where a non-zero modified trace exists, and even non-degenerate, which we discuss below.

We call a right (left) modified trace $\t$ {\em non-degenerate} if
 the pairings
\be\label{eq:nondegeneracy} 
	\Hom_\cat(M,P) \times \Hom_\cat(P,M) \to \,\ok
	\quad , \quad
	(f,g) \mapsto \t_P(f \circ g) \ ,
\ee
 are non-degenerate
for all $P\in \proj(\cat)$ and  $M\in \cat$.\footnote{We note that $M$ is not necessarily projective and so the cyclicity property does not generally applies here.}

For our main example $\cat=\Hmod$,  $\proj(\cat)=\Hpmod$ is the full subcategory of projective $H$-modules.

Let us motivate the definition of the modified trace from a
different perspective. 

\subsection*{Modified trace and Calabi-Yau structure}
Let $\catCY$  be a $\ok$-linear  category 
 equipped with a family of {\it trace maps}, i.e.
$\ok$-linear maps 
\be\label{eq:tV-gen}
\{t_V\colon \End_{\catCY}(V)\rightarrow \ok \}_{V\in \catCY}
\ee
satisfying the trace relation (or cyclicity)
$$t_V(g\circ f)=t_W(f\circ g)$$
for any $f\colon V\to W$ and $g\colon W\to V$ in $\catCY$. 
We say that $\catCY$ is {\em Calabi-Yau} 
 if the following  pairings 
 \be\label{eq:CY} 
	\Hom_\catCY(V,W) \times \Hom_\catCY(W,V) \to \,\ok
	\quad , \quad
	(f,g) \mapsto t_W(f \circ g) 
\ee
are non-degenerate for all $V,W\in \catCY$.

In any $\ok$-linear pivotal 
category $\catCY$ we have the following duality isomorphisms:
\begin{equation} \label{E:duality_iso_r}
\begin{split}
\begin{array}{rcl}
d^{\cap}\colon\; &\Hom_\catCY(W,U\tensor V) \xrightarrow{\sim} \Hom_\catCY(W\tensor V^*, U)\\
&f\mapsto (\id_U\otimes \tev_V)\circ (f\otimes \id_{V^*})\ 
\end{array}\
, \qquad
 \ipic{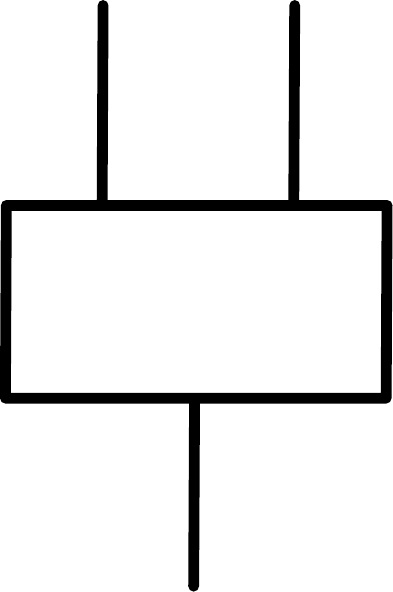}{.15} 
\put(-17,-3){\footnotesize $f$}
 \put(-25,24){\scriptsize $U$} \put(-11,24){\scriptsize $V$} 
\put(-18,-29){\scriptsize $W$} 
  \; \mapsto \; 
\ipic{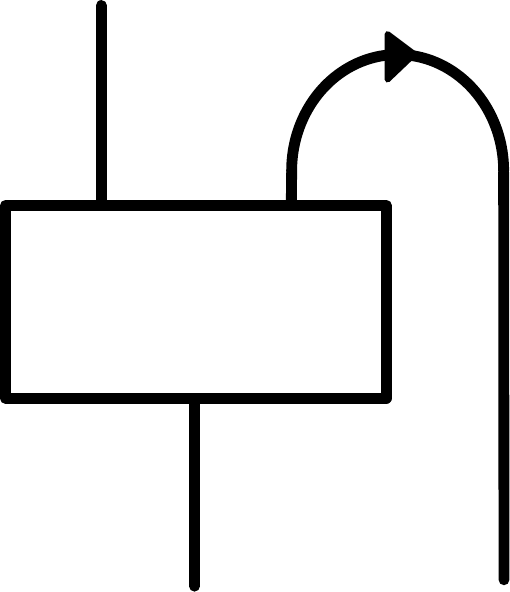}{.15} 
\put(-25,-3){\footnotesize $f$}
 \put(-33,24){\scriptsize $U$} 
\put(-28,-29){\scriptsize $W$}  \put(-4,-29){\scriptsize $V^*$} 
\put(-28,-35){\ }
\\
\begin{array}{rcl}
d_{\cup}\colon\; &\Hom_\catCY(U\tensor V, W) \xrightarrow{\sim} \Hom_\catCY(U,W\tensor V^*)\\
&f\mapsto(f\otimes \id_{V^*})\circ (\id_U\otimes \coev_V) \ 
\end{array}\
,\qquad
 \ipic{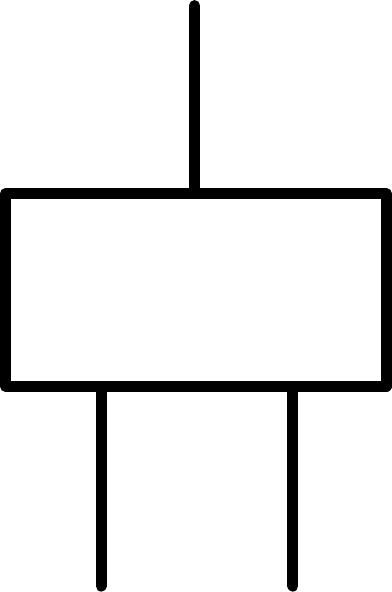}{.15} 
\put(-17,-2){\footnotesize $f$}
\put(-18,24){\scriptsize $W$} 
 \put(-25,-29){\scriptsize $U$} \put(-11,-29){\scriptsize $V$} 
  \; \mapsto \; 
\ipic{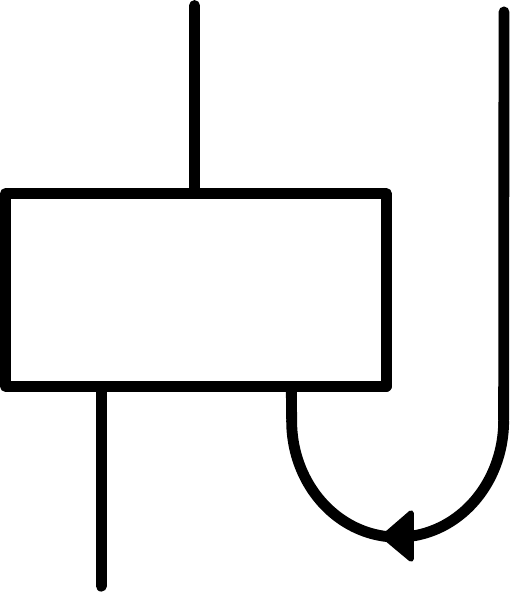}{.15} 
\put(-25,-2){\footnotesize $f$}
\put(-28,24){\scriptsize $W$}  \put(-4,24){\scriptsize $V^*$} 
 \put(-33,-29){\scriptsize $U$} 
 \end{split}
\quad \gp
\end{equation}

Let $\catCY$ be a  $\ok$-linear 
pivotal category.
We call a Calabi-Yau stucture on  $\catCY$ 
{\it compatible with  duality on the  right}
 if the following
diagram commutes, for all $U,V,W\in\catCY$,
\be\label{rcomp}
\xymatrix@R=20pt@C=40pt@W=10pt@M=10pt{
\Hom_{\catCY}(U\otimes V,W) \times \Hom_{\catCY}(W,U\otimes V) \ar[r]^{\qquad\qquad\quad \circ}\ar@<-45pt>[dd]^{d_{\cup}}\ar@<25pt>[dd]^{d^{\cap}}&
\End_{\catCY}(U\otimes V)\ar@<-0pt>[d]_{ t_{U\tensor V}}\\
& \ok\\
 \Hom_\catCY(U,W\tensor V^*) \times  \Hom_\catCY(W\tensor V^*, U) \ar[r]^{\qquad\qquad\qquad\circ}& \End_{\catCY}(U) \ar[u]^{ t_U}
}
\ee
We analogously define Calabi-Yau stucture on  $\catCY$ compatible with  duality on the  {\it left}, see more details in Section~\ref{sec:pivot-str}.
It is now easy to check that  the right  partial 
 trace condition~\eqref{rpartial} formulated for the family~\eqref{eq:tV-gen}  with $\catCY=\proj(\cat)$
implies commutativity of \eqref{rcomp},
and similarly for the left property. We give a proof that the inverse is also true, in Theorem~\ref{thm:CYcomp}.

\subsection*{Main results}
The previous discussion together with  Theorem~\ref{thm:CYcomp}
imply that a non\--de\-ge\-ne\-rate modified trace on $\proj(\cat)$
 is nothing else but a Calabi-Yau structure  on $\proj(\cat)$ 
compatible with  duality. 
For a 
finite-dimensional pivotal Hopf algebra $H$,
 such Calabi-Yau structure  
 on $\Hpmod$
is uniquely determined
by the non-degenerate symmetric linear form $\t_H\colon \End_H(H)\rightarrow \ok$ associated with the left regular representation. This is proven in
Proposition \ref{cor:ext} and Theorem~\ref{Thm:nondegeneracy} in a more general setting.

We are now ready to formulate our main result.
\begin{Thm}\label{thm:main}
Let $(\UH,\pivot)$ be a finite-dimentional unimodular pivotal Hopf algebra over a field~$\ok$. Then the space of right (left) modified traces on $\Hpmod$
is equal to the space of symmetrised right (left) integrals, and hence is 1-dimensional. Moreover, the right  modified trace on $\Hpmod$ is non-degenerate and determined by
 \begin{equation}\label{eq:tH-mu}
\t_\UH(f) = \rintg\bigl(f(\one)\bigr) \quad{\text{
for any}}\quad f\in\End_\UH(\UH)\ .
\end{equation}
Analogously, the left modified trace is non-degenerate and determined by
\begin{equation}
\t_{\UH}(f) = \lint_{\pivot^{-1}}\bigl(f(\one)\bigr)
\quad{\text{
for any}}\quad f\in\End_\UH(\UH)\ .
 \end{equation}
In particular, $\UH$ is unibalanced if and only if the right modified trace is also left. 
\end{Thm}

In the language of  Calabi-Yau categories, Theorem~\ref{thm:main} can be reformulated as follows.

\begin{corollary}\label{cor:1.1}
If $H$ is a finite-dimensional  unimodular pivotal Hopf algebra over a field $\ok$ then the space of  Calabi-Yau structures on $\Hpmod$ compatible with duality on  the right (left) is one dimensional.
\end{corollary}

To the best of our knowledge, Theorem
\ref{thm:main} is the first result relating
modified traces with  general concepts in the theory of Hopf algebras.
The power of this theorem is in the generality of its assumptions.
 The existence and uniqueness of the modified trace  was  proven previously in~\cite{Azat-Runkel}
for finite pivotal braided  categories with
a non-degenerate monodromy, called factorisable
(see also  \cite[Cor.\,3.2.1]{GKP} for a more technical statement for a larger class of categories).
The equality of the right and left modified traces  was
known in the ribbon case only.
However, Theorem \ref{thm:main} 
does not require braiding  and 
allows to compute the modified trace in all cases where the integral and pivot are known explicitly. We give few infinite
families  of unimodular Hopf algebras
with explicit formulas for the integral and pivots.

Theorem \ref{thm:main} was a starting point for further generalizations in  the following settings:
\begin{itemize}
\item
finite-dimensional pivotal {\it non-unimodular}
Hopf algebras in
\cite{FG} and \cite{GKPnew},
\item 
Hopf $G$-coalgebras in \cite{Ha}
 providing  examples of $G$-graded categories of~\cite{Turaev-G, Vi-G},
\item
quasi-Hopf algebras \cite{BGR,SS},
\item
module  categories in \cite{FG}.
\end{itemize}

As we already mentioned, our results were  crucially used in
\cite{RGP, RGP1} to construct a TQFT
with  stronger monoidality and functoriality properties than the one in~\cite{Kerler:2001}.
This was done for any finite-dimensional factorisable ribbon Hopf algebra over a field
 of characteristic zero.

In \cite{BBG}, combining the modified trace on the finite-dimensional restricted quantum $\mathfrak{sl}(2)$ at a root of unity
with the Hennings
construction, a {\em logarithmic Hennings} invariant was defined for any 3-manifold with a colored link inside.
An interesting feature of this construction is
that it works for a not necessarily quasi-triangular Hopf algebra.
The results of this paper suggest 
that the invariants of~\cite{BBG} can be extended to
   finite-dimensional 
Lusztig quantum groups at a root of unity which might not allow braiding.

To prove Theorem \ref{thm:main}, we  first show that 
the right partial trace property for the regular representation implies the general property  
in \eqref{rpartial}, and similarly for the left property. This is the context of the so-called Reduction Lemma
that is proven in  Section \ref{sec:pivot-str} in the general context
of finite pivotal $\ok$-linear categories. 

Then we study the centralizer algebras $\End_{\UH}(H\otimes W)$ for any  $W\in \Hmod$.
 An explicit algebra isomorphism
between $\End_\UH(\UH\tensor \UH)$ and 
${\rm{Mat}}_{n,n}(\UH^{\rm op})$ for any $n$-dimensional Hopf algebra $\UH$
 allows us to reduce the
right partial trace property to the defining relation
for the symmetrised right integral.

It is worth to mention the following 
consequence of Theorem \ref{thm:main}.

\begin{proposition}\label{prop:mod-tr-ss}
Let $H$ be a finite-dimensional unimodular pivotal Hopf algebra over a field~$\ok$. 
The right categorical trace $\tr^{\cat}_H$ and its left version $^{\cat}\!\tr_H$
are non zero if and only if $\Hmod$ is semisimple
and in this case coincide  up to a scalar with the trace maps
\be\label{eq:cat-tr-rintg}
f\mapsto \rintg\bigl(f(\one)\bigr)\qquad\text{and}\qquad
f\mapsto \lintg\bigl(f(\one)\bigr)\gc
\ee
respectively, where $f\in\End_H(H)$.
\end{proposition}

In Section~\ref{sec:Hopf} 
we give a Hopf-theoretic proof of Proposition~\ref{prop:mod-tr-ss} 
without using  Theorem \ref{thm:main}.

Proposition~\ref{prop:mod-tr-ss} shows that the symmetrised integral $\rintg$ provides a non-trivial generalisation of the categorical trace for non-semisimple  categories $\Hpmod$.
In this case, the categorical trace  is identically zero, however the symmetrised integral (or rather the corresponding modified trace) is not.
In particular, for a finite group $G$ and its group algebra over $\ok=\mathbb F_p$, the symmetrised integral, which is in this case just the integral, defines a non-degenerate  trace compatible with duality on the category of projective $\mathbb F_p[G]$-modules, even in the case
when the characteristic $p$ divides the order of the group.  This is a 
surprising application of our theorem to the classical modular representation theory, that
will be discussed in more details in  Section~\ref{sec:Hopf}.  
 
In Section \ref{qgroups}, we consider   finite-dimensional Lusztig quantum groups at roots of unity in the simply laced cases and give explicit formulas for
their  integral, cointegral,   symmetrised integral, and hence an explicit expression for the modified trace $\t_H$.
We expect similar formulas to hold in general type.

In type $A_1$, using Theorem \ref{thm:main}
together with formulas for minimal idempotents
given in~\cite{Gainutdinov:2007tc}, we obtain an alternative derivation
of \cite{BBG} formulas for the modified trace 
for all endomorphisms of indecomposable projectives.
This illustrates how the
 modified trace
can be effectively  computed from the symmetrised integral.

The paper is organised as follows.
In Section~\ref{sec:traces}, we collect results on
traces in finite abelian categories. In Section~\ref{sec:pivot-str},
we study a relationhsip between the modified trace and
Calabi-Yau structures in finite pivotal categories and prove
Reduction Lemma. In Section~\ref{sec:Hopf}, after recalling
standard facts from the theory of Hopf algebras, we 
study properties of symmetrised integrals, in particular we show that they provide a non-degenerate symmetric pairing between the center $Z(H)$ and ${\rm HH}_0(H)$, and then prove Proposition \ref{prop:mod-tr-ss}. 
Section~\ref{sec:HH-decomp} contains a detailed analysis of the centralizer algebras
$\End_{\UH}(H\otimes W)$.
Section~\ref{sec:proof} contains our proof of Theorem \ref{thm:main}. Section~\ref{qgroups}  provides an application of our main theorem to restricted quantum groups of types $ADE$:
we compute the modified trace via a calculation of $\rint_\pivot$. Then in Section~\ref{sec:A1case}  we provide more detailed analysis for $\mathfrak{sl}_2$ case. Finally, Appendices contain proofs of several lemmas.

\subsection*{Acknowledgements} The authors are grateful 
to NCCR SwissMap for generous support
and to   Nathan Geer, Bertrand Patureau,
Marco de Renzi and Ingo Runkel
for helpful discussions. 
The authors are also thankful to the organizers of conference ``Invariants in low-dimensional Geometry \& Topology" in Toulouse in May, 2017, where a substantial part of this work was done. CB and AMG also thank Institute of Mathematics in Zurich University for kind hospitality during 2017.
AMG is supported by CNRS and ANR project JCJC ANR-18-CE40-0001, and also thanks the Humboldt Foundation for a partial financial support.
 
\section{Traces on finite categories}\label{sec:traces}
Throughout this section $A$ is a finite-dimensional
 $\ok$-algebra.
Our aim  is to show that 
any symmetric linear form $t$ on $A$ determines
a family of trace functions on $\Apmod$
\be\label{Atraces}
\{t_P\colon \End_A(P)\rightarrow \ok\}_{P\in \Apmod}\ ,
\ee
i.e.\ linear maps 
satisfying cyclicity \eqref{E:Tracefggf}.
We will also show that if $t\in \UA^*$ is non-degenerate,  then the traces \eqref{Atraces}  are non-degenerate  
in the
sense of~\eqref{eq:nondegeneracy}. 

\subsection*{General Setting}
We assume that $\ok$ is a field and $\cat$ is an
additive category.
We call $\cat$ \textit{$\ok$-linear} if $\Hom_\cat(X,Y)$ 
is a vector space over $\ok$ for all $X,Y\in \cat$ and the composition of morphisms is $\ok$-bilinear. All
categories used in this paper are assumed to be $\ok$-linear.

An abelian category $\cat$ is called \textit{finite}
if it is equivalent to the category $\Amod$
of finite-dimensional left $A$-modules 
for some  finite-dimensional  $\ok$-algebra~$A$. In other words, $\cat$ is abelian and has finitely many isomorphism classes of simples, length of any object is finite, it has  enough projectives and Hom spaces are finite-dimensional.
An algebra $A$ can be constructed as $\End_\cat(G)$ for a projective generator $G\in\cat$, see e.g.~\cite{Droz}.
Then, the equivalence functor  $\Hom_\cat(-,G)\colon \cat \to \Amod$ sends $G$ to the regular representation $A$. 
Therefore, without loss of generality   in this section we will assume that $\cat=\Amod$.
  We will also use the notation $\Apmod$ for
the full subcategory of projective $A$-modules.

We first show that 
a family of traces \eqref{Atraces} on $\Apmod$ defines
 a symmetric linear form on~$A$.
 Let us denote by $A^\op$ the  algebra with the opposite multiplication. 

\begin{lemma}\label{lem:Uop-End}
 We have the isomorphism of algebras
\be \label{eq:hop=end}
r\colon A^{\op} \xrightarrow{\;\sim\;} \End_A(A)\ 
\ee
given by
\be
 r(x) = r_x, \qquad r^{-1}(f) = f(\one)
\ee
where by $r_x$ we denote the right multiplication with $x$.
\end{lemma}
\begin{proof}
It is straightforward to check 
that the maps $r$ and $r^{-1}$ defined in
\eqref{eq:hop=end}
are inverse to each other.
Moreover, for any $x,y \in A$ we have
$r(xy)=r_y r_x$ and for any $f,g \in \End_A(A)$,
$r^{-1}(gf)=(gf)(\one)=f(\one) g(\one)$,  where in the last equality we used the intertwining property of $g$.
\end{proof}

Suppose we are given trace functions \eqref{Atraces}.
Then, in particular, for the regular $A$-module~$A$, we have
the trace function
$t_A\colon\End_A(A)\rightarrow \ok$.
Lemma \ref{lem:Uop-End} shows that
 $t_A$ defines a symmetric linear form $t$ on $A^{\op}$.
 Since the flip of multiplication is irrelevant
 in the argument  of a
 symmetric form, we have
$t\in A^*$.

To argue that the converse is also true: given a symmetric form 
$t\in A^*$ we can  extend it uniquely to 
a family of traces on $\Apmod$, we will 
 need a categorical notion of the 
$0^{\rm th}$-Hochschild homology.

\subsection*{Traces of categories}
The $0^{\text{th}}$-Hochschild homology or trace 
 of a $\ok$-linear category $\mathcal C$
 is defined by 
\begin{equation} \label{eq:HH0-C-def}
\HH(\mathcal C):=\frac{\bigoplus_{X\in \mathcal C} \;\End_{\mathcal C}(X)}
{[\mathcal{C},\mathcal{C}]}
\end{equation}
where
$$
[\mathcal{C},\mathcal{C}]:={\rm Span}\{f\circ g-g\circ f\;|\;
f\in \Hom_{\mathcal{C}}(X,Y),\;
g\in \Hom_{\mathcal{C}}(Y,X), \; X,Y\in \mathcal{C}
\}\ .
$$
The image of $f\in \End_{\mathcal{C}}(X)$ in $\HH(\mathcal C)$ will be called its trace class and denoted by $[X,f]$ or simply by $[f]$.

In particular,  $0^{\text{th}}$-Hochschild homology of an algebra
$A$ (viewed as a category with one object) is
\be\label{eq:HH0-alg-def}
 \HH(A):=\frac{A}{[A,A]}\quad\text{with}\quad
[A,A]={\rm{Span}}\{xy-yx\;|\;  x,y\in A\}.
\ee
Again the image of $x\in A$ in $\HH(A)$ will be called its trace class and denoted by $[x]$.

Actually, 
$\HH(A)$ and $\HH(\Apmod)$ are isomorphic.
To show this we will need some preparation.

\begin{lemma}\label{lem:decomp-id}
For any projective $A$-module $P$ there exists a decomposition of the identity:
\be \label{E:identities}
\id_P=
\sum_{i\in I}
 a_i\circ \id_A
\circ b_i
\ 
\ee 
for some finite set $I$ and morphisms $a_i\colon A\to P$ and $b_i\colon P\to A$.
\end{lemma}
\begin{proof}
Recall that
any finitely generated projective $A$-module $P$ splits as a direct sum of indecomposables ones:
\be\label{eq:P-decomp}
P\simeq\bigoplus_{i\in I} P_i \ ,
\ee
for a finite indexing set $I$. Here several direct summands can be isomorphic.
We  further observe that each indecomposable  $P_i$ 
can be realised as a direct summand in $A$, since 
the regular module $A$ is a projective generator of $\Apmod$.
We have therefore an injective map $x_i\colon P_i\hookrightarrow A$ and a surjective map $y_i\colon A \twoheadrightarrow P_i$.
Fixing  these maps such that $y_i\circ x_i = \id_{P_i}$ we can define
$a_i\colon A \rightarrow P$ and  $b_i\colon P\rightarrow  A$
as the compositions 
\be
 a_i\colon A \xrightarrow{\;y_i\;}   P_i\hookrightarrow P \quad \text{and} \quad
 b_i\colon P\twoheadrightarrow P_i \xrightarrow{\;x_i\;} A\gp
 \ee
They clearly satisfy~\eqref{E:identities}.
\end{proof}

We note that the decomposition of $\id_P$ in~\eqref{E:identities} is not unique, and we provide several examples.

\noindent{\bf Example.}
For $P=A$, we can use the trivial decomposition  $a_1=b_1=\id_A$. However we can make another choice, the one corresponding to $a_i$ and $b_i$ in the proof of Lemma~\ref{lem:decomp-id}: 
$a_i=b_i\colon x\mapsto x\pi_i$, for $x\in A$ and here $\pi_i$ is the primitive idempotent  corresponding to the direct summand $P_i$ in the decomposition $A=\oplus_{i=1}^l P_i$. 
The identity~\eqref{E:identities} is clearly satisfied because $\sum_{i=1}^l \pi_i =  \one$.\\
In a more general case of  $P=A^{\oplus m}$, we also have two natural decompositions. For the first one,   we set $a_j\colon A\hookrightarrow A^{\oplus m}$ and  $b_j\colon  A^{\oplus m}\twoheadrightarrow A$ such that $b_i\circ a_j = \delta_{i,j}\, \id_A$, then~\eqref{E:identities} holds. For the other choice, let $V$ be the $m$-dimensional multiplicity space with a basis $e_j$, $1\leq j\leq m$, and we can then define for each pair $i=(k,j)$
the  maps $a_i$, $b_i$ as
\begin{align}\label{eq:Am-ab}
a_{(k,j)}\colon &A\to A^{\oplus m}\ ,
&b_{(k,j)}\colon &A^{\oplus m} \to A\ ,\\
&x\mapsto x\pi_k\tensor e_j\ ,  & &x\tensor e_n \mapsto \delta_{n,j}\, x\pi_k\ , \notag
\end{align}
for $x\in A$, and $1\leq k\leq l$ and  $1\leq n,j\leq m$.
It is then straightforward to check the identity~\eqref{E:identities} on $x\tensor e_n$ for any $x\in A$ and $1\leq n\leq m$.
 
\begin{proposition} \label{Meq}
 For a finite-dimensional algebra $A$, 
there is an isomorphism 
\begin{align} \label{eq:Phi-map-def}
\Phi: \HH(A)&\xrightarrow{\sim}\HH ( \Apmod)\ ,\\
[x]&\mapsto [r_x]
\nonumber
\end{align} 
with the inverse map
\begin{align}\label{eq:Psi-map-def}
\Psi: \HH(\Apmod)&\xrightarrow{\sim}\HH (A)\ ,\\
[P,f]&\mapsto \sum_{i\in I}\left[ (b_i\circ f\circ a_i)(\one)\right] \ ,\nonumber
\end{align} 
for any sets $\{a_i\colon A\to P\}_{i\in I}$ and $\{b_i\colon P\to A\}_{i\in I}$ satisfying~\eqref{E:identities}. 
\end{proposition}

We provide the proof in Appendix~\ref{appA} for completeness.

\begin{proposition}\label{cor:ext}
 A symmetric linear form $t$
on a finite-dimensional algebra $A$ 
extends uniquely to a  family of trace maps $\{t_P\colon \End_A (P)\to \ok\}_{P\in \Apmod}$ where
\be \label{E:extension}
t_P(f)= \sum_{i=1}^k t\bigl((b_i\circ f\circ a_i)(\one)\bigr)\ ,
\qquad f\in\End_A(P)\ ,
\ee
for a given decomposition 
of $\id_P$ 
as in~\eqref{E:identities}.
In particular, we have
\be\label{eq:tAt}
 t_A(r_x)=t(x)\ ,\qquad  
x\in A\ .
\ee
\end{proposition}
\begin{proof}
We first note that there is a bijection between linear forms on $\HH(\Apmod)$ and  families of trace maps $\{t_P\colon \End_A (P)\to \ok\}_{P\in \Apmod}$ 
such that $t_P(f)= l([P,f])$ for a linear form $l$.
A symmetric linear form $t\colon A \rightarrow \ok$
provides
 a linear form on $\HH(A)$ which we also denote by $t$.
By Proposition \ref{Meq},
this defines  a linear form on $\HH(\Apmod)$
by the formula
\be
t_P(f)=t\circ\Psi\bigl([f]\bigr)
\ee
 for any $f\in\End_A(P)$ and $\Psi$ given in~\eqref{eq:Psi-map-def}. Since $\Psi$ is an isomorphism and it
 does not depend on the choice of the
 decomposition of $\id_P$,
we have the existence and uniqueness of the extension. 
Finally, the equality~\eqref{eq:tAt} is straightforward after using~\eqref{eq:Phi-map-def}.
\end{proof}

We remark that a result similar to Proposition~\ref{cor:ext} was also proven in~\cite[proof of Prop.\ 5.8 (1)]{Azat-Runkel}  (however in the case of non-degenerate traces).

\noindent{\bf Example.}
 We assume here that $P=A^{\oplus m}$ and demonstrate the use of the formula~\eqref{E:extension}.
The algebra of $A$-linear
endomorphisms of $A^{\oplus m}$ can be rewritten as a matrix algebra:
\be\label{eq:EndHV-Mat}
\End_A(A^{\oplus m})
\cong \Mat_{m,m}(A^{\op})
\ee
where $\Mat_{m,m}$ is the $m\times m$ matrix algebra and  we used Lemma~\ref{lem:Uop-End}.
With notation as in~\eqref{eq:Am-ab}, the isomorphism~\eqref{eq:EndHV-Mat}   sends a matrix $(h_{ij})$ to the endomorphism $x\tensor e_j \mapsto \sum_{r=1}^m x h_{rj}\tensor e_r$.
Let us choose $a_i$ and $b_i$ as in~\eqref{eq:Am-ab}.
From~\eqref{E:extension}, we then obtain 
the unique extension $t_{A^{\oplus m}}$ of the symmetric form $t$
\be\label{eq:t-oplus-m}
t^{\oplus m}(h) :=  t_{A^{\oplus m}}(h) =\sum_{i=1}^m t(h_{ii})
\ , \qquad h\in \End_A(A^{\oplus m})\ ,
\ee
where we used cyclicity of $t$ and on RHS we  identified $h$ with the corresponding element in 
$\Mat_{m,m}(A^{\op})$ 
under the  isomorphism in~\eqref{eq:EndHV-Mat}.

\begin{remark}\label{rem:indP-ext}
For  an indecomposable projective $A$-module $P$, we   can reformulate Proposition~\ref{cor:ext} in the following way. 
 Let us fix an injection $j\colon P\hookrightarrow A$ and projection $p\colon A\twoheadrightarrow P$ such that $p\circ j=\id_P$ --
 this identity provides 
a decomposition as in~\eqref{E:identities}. Then $j\circ p\in \End_{A}(A)$ is right multiplication by a 
primitive
idempotent $\pi$, and so $t_P(\id_P)=t(\pi)$.\\
If $P\in\Apmod$  is not necessarily indecomposable, then it can be realised as a direct summand of $A^{\oplus m}$ 
for some finite $m\in\mathbb{Z}_{>0}$, i.e.\ 
we have injective and surjective maps:
\be
j_P \colon \; P \hookrightarrow A^{\oplus m}\ , \qquad p_P \colon \; A^{\oplus m} \twoheadrightarrow P 
\ee 
such that the composition $p_P\circ j_P$ is identity on $P$ and $j_P\circ p_P$ is an idempotent in $\End_A( A^{\oplus m})$.
We then  get a decomposition of the form \eqref{E:identities}
with 
\be\label{eq:AmP-ab}
\tilde{a}_{(k,j)}\colon A\xrightarrow{a_{(k,j)}} A^{\oplus m} \xrightarrow{p_P} P\quad \text{and} \quad 
 \tilde{b}_{(k,j)}\colon P\xrightarrow{j_P} A^{\oplus m} \xrightarrow{b_{(k,j)}} A \ ,
 \ee 
 while $a_{(k,j)}$ and $b_{(k,j)}$ are defined as in~\eqref{eq:Am-ab}.
Then~\eqref{E:extension} for the choice~\eqref{eq:AmP-ab} gives the following expression for $t_P$:
\be\label{eq:tP-mug-def}
t_P\colon \; f\mapsto t^{\oplus m} \bigl(j_P \circ f \circ p_P\bigr) \ ,
\ee
with $t^{\oplus m}$ defined in~\eqref{eq:t-oplus-m}.
For certain proofs below it will be more convenient
to use the decomposition $\id_P=p_P\circ j_P$ instead of~\eqref{E:identities}
and this expression of~$t_P$. It is a consequence of Proposition~\ref{cor:ext} that the map~\eqref{eq:tP-mug-def} does not depend on the choices we made in the construction.
\end{remark}

\subsection*{Non-degeneracy}\label{CY}
Let us prove the equivalence of the different notions of
 non-degeneracy.

For  a finite-dimensional algebra $A$ over a field $\ok$,
we call a  linear form $t\in A^*$  \textit{non-degenerate} if the associated   bilinear pairing $(x, y)\mapsto t(xy)$ is non-degenarate, i.e.
$t(xy)=0$ for all $x\in A$ implies $y=0$.

\begin{theorem} \label{Thm:nondegeneracy}
For a finite-dimensional 
algebra $A$ with a
 symmetric linear form $t\in A^*$ the following three
 statements are equivalent:
 \begin{enumerate}
 \item
 $t$ is non-degenerate.
 \item
$\Apmod$ is Calabi-Yau with $t_P$  defined by \eqref{eq:tP-mug-def}.
 \item
 The
 pairings \eqref{eq:nondegeneracy}
$$ 
	\Hom_A(M,P) \times \Hom_A(P,M) \to \,\ok
	\quad , \quad
	(f,g) \mapsto t_P(f \circ g) \ 
$$
 are non-degenerate
for all $P\in \Apmod$ and $M\in \Amod$. 
 \end{enumerate}

\end{theorem}
\begin{proof}
The equivalence of the first two statements was proven
in \cite[Prop. 5.8]{Azat-Runkel}. Since the third statement is the strongest, it is enough  to show that it follows from the first
one.
For that
we need to show that 
for any $f\colon M\to P$
there exists a non-zero map $g\colon P\to M$ such that $t_P(f\circ g)\ne0$. The idea is to use non-degeneracy of the linear form $t^{\oplus m}$. Let us fix a projective cover
$P_M$  of $M$ with the canonical
  surjective map $\pi_M\colon P_M \twoheadrightarrow M$.
  Since any projective module is a direct summand of a projective generator, say  $A^{\oplus m}$ for some $m$, we have 
   surjective and injective maps:
  $$p_M\colon A^{\oplus m}\twoheadrightarrow P_M \quad
  {\text{and}}\quad j_M \colon \; P_M \hookrightarrow A^{\oplus m} .
$$  

Let us consider the surjective map $\tilde p_M= \pi_M \circ p_{M}\colon A^{\oplus m} \twoheadrightarrow M$. 
 By assumption $f$ is non-zero and therefore the composition $j_P \circ f\circ \tilde p_M\in \End_A (A^{\oplus m})$ is non-zero too, because $\tilde p_M$ is surjective and $j_P$ is injective. 
Since $t^{\oplus m}$ is non-degenerate, there should be non-zero $\tilde g\in \End_A (A^{\oplus m})$ such that
\be\label{eq:rintgn-ne-0}
t^{\oplus m}\bigl((j_P \circ f\circ \tilde p_M) \circ \tilde g\bigr) \ne 0\ .
\ee
We set $g= \tilde p_M \circ \tilde g \circ j_P\colon P\to M$ and check using~\eqref{eq:tP-mug-def} the non-degeneracy of $t_P$:
\begin{align}
t_P(f\circ g) 
&=  t^{\oplus m}\bigl(j_P \circ f\circ (\tilde p_M \circ \tilde g \circ j_P) \circ p_P\bigr)\nonumber\\
&=  t^{\oplus m} \bigl( j_P \circ p_P \circ j_P \circ f\circ \tilde p_M \circ \tilde g \bigr)= 
 t^{\oplus m} \bigl( j_P  \circ f\circ \tilde p_M \circ \tilde g \bigr) \ne 0
\end{align}
where in the second equality we used cyclicity of $t^{\oplus m}$ and in the third the identity $ p_P \circ j_P = \id_P$, and finally we used~\eqref{eq:rintgn-ne-0}.
This also shows that the map $g$ is non-zero. This calculation finishes the proof of non-degeneracy of the family  $t_P$.
\end{proof}

\section{Modified trace and Calabi-Yau structure}
\label{sec:pivot-str}

In this section for a  finite pivotal  category $\cat$
 we prove Reduction Lemma
and
 show  that a Calabi-Yau  structure on $\proj(\cat)$ provides a non-degenerate modified trace if and only if a compatibility between the Calabi-Yau structure and  duality holds. Recall that $\proj(\cat)$ denotes  the tensor ideal of projective modules in $\cat$.

\subsection*{Pivotal structure} 
A  category $\cat$ is  pivotal if $\cat$ is a monoidal category with left duality  equipped with a
monoidal natural isomorphism $\delta\colon  \id_{\cat} \to (\arg)^{**}$ between the identity functor and the double duality functor. 
 We note that  the corresponding isomorphisms automatically satisfy $\delta_{V^*}=(\delta^*_V)^{-1}$ for $V\in \cat$, see~\cite[Prop.\,A.1]{Schauenburg}.
 
The pivotal structure allows to define right duality. Right dual objects are identified with the left ones, and the right (co)evaluation maps are defined as
\be\label{eq:right-duality}
\begin{split}
\widetilde\ev_V &:= \ev_{V^*} \circ (\delta_V\tensor \id_{V^*} ) \colon \qquad V\otimes V^* \to \one\ , \\
\widetilde\coev_V &:= (\id_{V^*}\tensor \delta^{-1}_{V} )\circ \coev_{V^*}  \colon \quad \one \to V^*\otimes V\ .
\end{split}
\ee
For the  left and right (co)evaluation maps we will use the following diagrammatical notations:
\begin{align}\label{eq:ev-coev-gr}
\ev_V &~=\hspace*{.7em}   \ipic{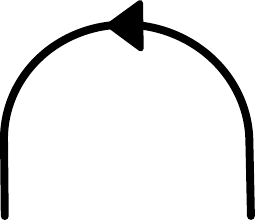}{.25}  
\put(-34,-22){\scriptsize $V^\ast$} \put(-4,-22){\scriptsize $V$} \qc&
\coev_{V} &~=\hspace*{.7em} \ipic{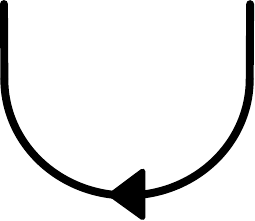}{.25}  
\put(-34,16){\scriptsize $V$} \put(-5,16){\scriptsize $V^\ast$}  \qc&\\ 
\widetilde\ev_V &~=\hspace*{.7em} \ipic{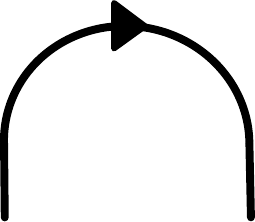}{.25} 
\put(-34,-22){\scriptsize $V$} \put(-6,-22){\scriptsize $V^{*}$} 
\qc
& 
\widetilde\coev_V &~=\hspace*{.7em} \ipic{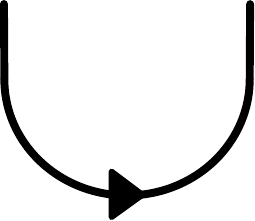}{.25} 
\put(-38,16){\scriptsize $V^{*}$}
 \put(-5,16){\scriptsize $V$} 
\qp\notag
\end{align}
We recall the definition of the right and left partial traces in~\eqref{E:PartialLRtrace}. 
They have the following property.
\begin{lemma}\label{lem:reverse-arrows}
Let $\cat$ be a pivotal category and $Q,P\in\cat$, we have then the equality
\be
 \ipicc{part-tr-right}{.16} 
 \put(-35,-25){{\tiny $Q$}}  \put(-27,1){{\footnotesize $f$}}\put(-27,17){{\tiny $P^*$}} \put(1,5){{\tiny $P^{**}$}}
\quad\; =\; \; 
 \ipicc{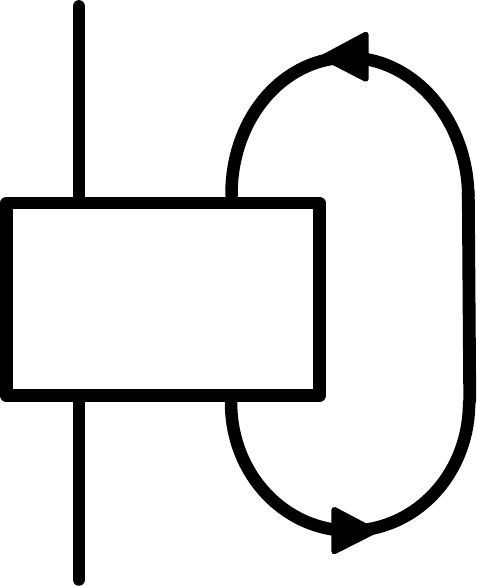}{.16} 
 \put(-35,-25){{\tiny $Q$}}  \put(-27,1){{\footnotesize $f$}}\put(-27,17){{\tiny $P^*$}} \put(1,5){{\tiny $P$}}
\ee
for any $f\in\End_\cat(Q\tensor P^*)$, and similarly for the left partial trace of $f$.
\end{lemma}
\begin{proof}
We factorise  $\id_{P^{**}}= \delta_{P}\circ \delta^{-1}_P$ using pivotal isomorphisms  and use~\eqref{eq:right-duality}
 to reverse  arrows.
\end{proof}

We call an abelian category $\cat$
\textit{finite pivotal} if $\cat$ is a finite tensor category in the sense of~\cite{EGNO-book}, i.e.\ (1) if $\cat$ is finite as an abelian category,  (2) if it is a rigid monoidal category with $\ok$-bilinear and bi-exact tensor product functor, and (3) if its tensor unit is simple; and  if $\cat$ has  a pivotal structure.

\subsection*{Reduction Lemma} Let us  prove  Reduction Lemma mentioned in Introduction, which says
that to verify the right or left partial trace
property, it is enough to check it on a projective generator. Below is the exact statement, recall also Proposition~\ref{cor:ext}.
\begin{lemma}\label{lem:tHH-tH}
Given a  
finite pivotal category $\cat$ and a projective generator $G\in\cat$,
a symmetric linear form $t\in A^*$, where $A:= \End_\cat(G)$, extends uniquely to a right modified trace on $
\proj(\cat)$   if and only if
\be\label{eq:tGG-tG}
 t_{G\otimes G}\left(f \right)=
t_G
 \bigl( \tr^r_G(f)\bigr) ,
  \qquad  \text{for all}\quad f\in \End_\cat(G\otimes G).
\ee
 Analogously, $t$ extends  to a left modified trace on $\proj(\cat)$   if and only if
\be\label{eq:tGG-tG-left}
 t_{G\otimes G}\left(f \right)=
t_G
 \bigl( \tr^l_G(f)\bigr) ,   \qquad  \text{for all}\quad f\in \End_\cat(G\otimes G).
\ee
\end{lemma}

\begin{proof}
Only one direction is not obvious.
By Proposition~\ref{cor:ext},
the symmetric  form $t\in A^*$ extends uniquely to a family of linear maps  $t_P\colon \End_\cat(P)\rightarrow \ok $, for ${P\in \proj(\cat)}$, which satisfies the cyclicity property. We need to check  the right partial trace property.

We first prove~\eqref{rpartial} for a pair of projective objects.
Assume  $P, P'\in \proj(\cat)$ and  $f\in \End_\cat(P\otimes P')$.
We have finite sum decompositions of the identities as in~\eqref{E:identities}:\,\footnote{Here, we use the projective generator $G$ instead of the regular module $A$ as we work in $\cat$, recall that the equivalence functor $\Hom_\cat(-,G)$ between $\cat$ and $\Amod$ sends $G$ to $A$.}
\be\label{eq:idP-idP}
\id_P=\sum_{i\in I}a_i\circ \id_G\circ b_i\ ,\qquad
  \id_{P'}=\sum_{{i'\in I'}}a_{i'}\circ \id_G\circ b_{i'}\ .
 \ee
We can now calculate $t_{P\tensor P'} (f)$ in terms of $t_{G\tensor G}$  by inserting these identities and using the cyclicity. Indeed,
\be\label{eq:t-ab}
t_{P P'} (f)\; =\; 
t_{P P'}\left(
\ipic{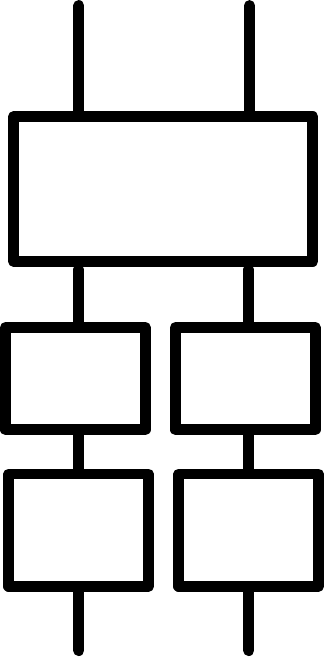}{.17} 
\put(-16,9){\scriptsize $f$}
\put(-24,-6){\scriptsize $a_i$} \put(-11,-6){\scriptsize $a_{i'}$}
\put(-24,-19){\scriptsize $b_i$} \put(-11,-19){\scriptsize $b_{i'}$}
 \right)
  \; \stackrel{\text{cycl.}}{=} \; 
 t_{GG}\left(
\ipic{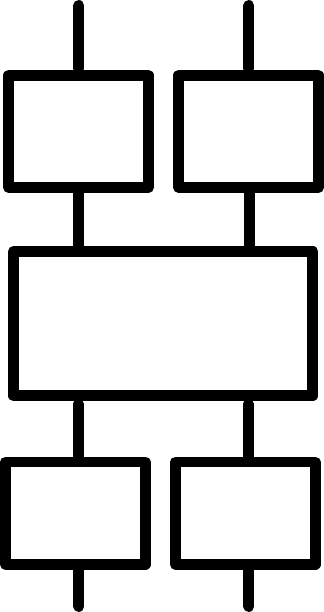}{.17} 
\put(-24,12){\scriptsize $b_i$} \put(-11,12){\scriptsize $b_{i'}$}
\put(-16,-4){\scriptsize $f$}
\put(-24,-19){\scriptsize $a_i$} \put(-11,-19){\scriptsize $a_{i'}$}
 \right)
    \; \stackrel{\eqref{eq:tGG-tG}}{=} \; 
   t_{G}\left(
\ipic{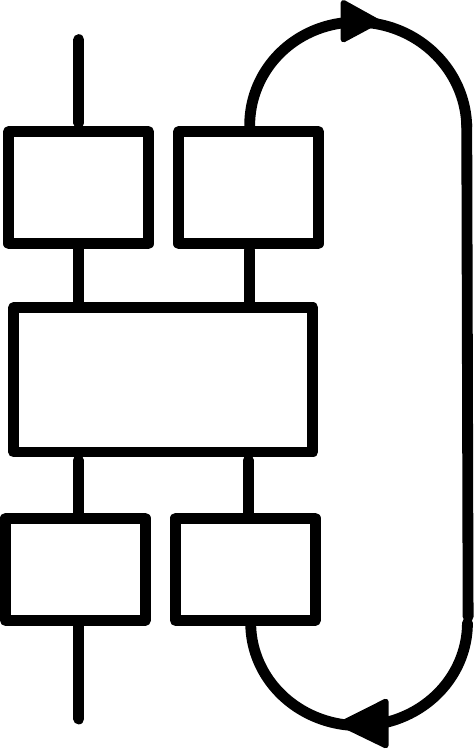}{.17} 
\put(-36,13){\scriptsize $b_i$} \put(-23,13){\scriptsize $b_{i'}$}
\put(-28,-3){\scriptsize $f$}
\put(-36,-18){\scriptsize $a_i$} \put(-23,-18){\scriptsize $a_{i'}$}
 \right)
   \; \stackrel{\text{*}}{=} \; 
   t_P\bigl(\tr^r_{P'}(f)\bigr)\ 
  \ee
where we omit the tensor product symbol in the index of $t$ for brevity, and the summation is assumed over the repeated indices, i.e.\ over $i\in I$ and $i'\in I'$.
In the step $(*)$ we used first the standard manipulations with dual maps to move $b_{i'}$ around the loop and then applied~\eqref{eq:idP-idP}, and finally applied the cyclicity property of $t_G$ using again~\eqref{eq:idP-idP}.

We have thus established the right partial trace property of $t$ in the case where both objects are projective.
Now assume  $P\in \proj(\cat)$ and $V\in \cat$. Then we set $\tP:=P\otimes V$
 which is in $\proj(\cat)$ due to exactness of the tensor product.
  For $f\in \End_\cat(P\otimes V)$, let $\A\in \Hom_\cat(P\otimes P^*,\tP\otimes \tP^*)$ and $\B\in \Hom_\cat(\tP\otimes \tP^*,P\otimes P^*)$ be defined as in Figure \ref{fig:F}. 
\begin{figure}
{
\ipic{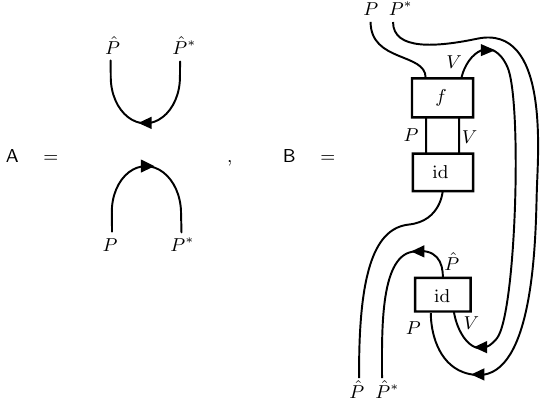}{1.0} 
}
\caption{Morphisms $\A$ and $\B$.}
\label{fig:F}
\end{figure}
Using the  right partial trace property for projective objects established in~\eqref{eq:t-ab},
we get
\be\label{eq:tABBA}
\begin{split}
&t_{P\otimes P^*}(\B\circ \A)=t_P\bigl(\tr^r_{P^*}(\B\circ \A)\bigr)\stackrel{*}{=}t_P\bigl(\tr^r_V(f)\bigr)\ ,\\
& t_{\tP\otimes {\tP}^*}(\A \circ \B)=t_{\tP}\bigl(\tr^r_{{\tP}^*}(\A\circ\B)\bigr)\stackrel{*}{=} t_{P\otimes V}(f)\ ,
\end{split}
\ee
where in steps $(*)$ we used first Lemma~\ref{lem:reverse-arrows} and then simple manipulations with the diagrams, like the zig-zag indentity for the left duality.
Using  the cyclicity equation 
$t_{P\otimes P^*}(\B\circ \A) = t_{\tP\otimes {\tP}^*}(\A \circ \B)$ and comparing both the lines in~\eqref{eq:tABBA} we finally get the equality $t_{P\otimes V}(f) = t_P(\tr^r_V(f))$.
The proof for the left modified trace goes along similar lines 
after reflecting all diagrams on a vertical line. 
\end{proof}

\subsection*{Duality and Calabi-Yau structure}
We now recall that in any  pivotal 
 category~$\catCY$ we have the isomorphisms, for $U,V,W\in \catCY$,
\be
\begin{split} 
 \begin{array}{rcl}
{}^{\cap}d\colon\; &\Hom_\catCY(W,U\tensor V) \xrightarrow{\sim}
 \Hom_\catCY( U^*\tensor W, V)\\
&f\mapsto (\ev_U\otimes \id_V)\circ (\id_{U^*}\otimes f)\ 
\end{array} \
,\qquad
 \ipic{rect-map-1-2-verso}{.15} 
\put(-17,-3){\footnotesize $f$}
 \put(-25,24){\scriptsize $U$} \put(-11,24){\scriptsize $V$} 
\put(-18,-29){\scriptsize $W$} 
  \; \mapsto \; 
\ipic{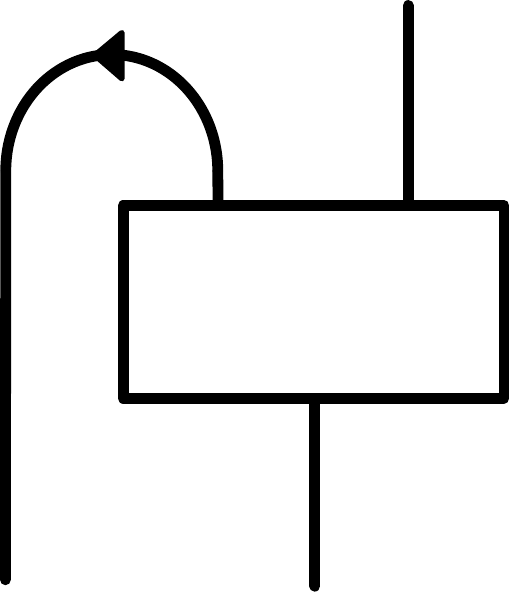}{.15} 
\put(-17,-3){\footnotesize $f$}
 \put(-11,24){\scriptsize $V$} 
\put(-39,-29){\scriptsize $U^*$}  \put(-18,-29){\scriptsize $W$}
\put(-28,-35){\ }
  \\
\begin{array}{rcl}
{}_{\cup}d\colon\;\ &\Hom_\catCY(U\tensor V, W) \xrightarrow{\sim} \Hom_\catCY(V,U^*\tensor W)\\
&f\mapsto (\id_{U^*} \otimes f)\circ (\tcoev_U\otimes \id_V)\ \end{array}\
, \qquad
 \ipic{rect-map-1-2}{.15} 
\put(-17,-2){\footnotesize $f$}
\put(-18,24){\scriptsize $W$} 
 \put(-25,-29){\scriptsize $U$} \put(-11,-29){\scriptsize $V$} 
  \; \mapsto \; 
\ipic{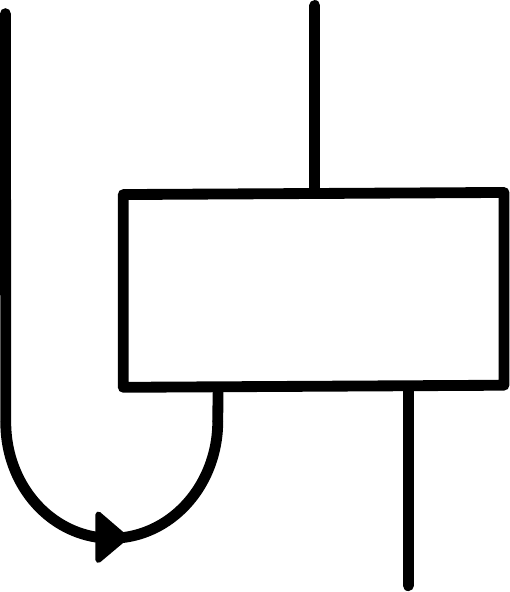}{.15} 
\put(-17,-2){\footnotesize $f$}
\put(-39,24){\scriptsize $U^*$}  \put(-18,24){\scriptsize $W$} 
 \put(-11,-29){\scriptsize $V$}  
 \end{split}
 \label{E:duality_iso_l}
\ee
that are defined analogously to~\eqref{E:duality_iso_r}, with the duality maps on the left side.

 Calabi-Yau (CY) structure  on $\catCY$ compatible with  duality on the  right was introduced before diagram~\eqref{rcomp}. 
Similarly, we say that a CY structure on $\catCY$  is {\it compatible with  duality on the  left}   if the following
diagram commutes  for  all $U,V,W\in \catCY$:
\be \label{lcomp}
\xymatrix@R=20pt@C=40pt@W=10pt@M=10pt{
\Hom_{\catCY}(U\otimes V,W) \times \Hom_{\catCY}(W,U\otimes V) \ar[r]^{\qquad\qquad\quad\circ}\ar@<-45pt>[dd]^{{}_{\cup}d}\ar@<25pt>[dd]^{{}^{\cap}d}&
\End_\catCY(U\otimes V) \ar@<-0pt>[d]_{ t_{U\tensor V}}\\
& \ok\\
 \Hom_\catCY(V, U^*\tensor W) \times  \Hom_\catCY(U^*\tensor W, V) \ar[r]^{\qquad\qquad\qquad\circ} & \End_\catCY(V) \ar[u]^{ t_V}
}
\ee 

\begin{theorem}\label{thm:CYcomp}
Let $\cat$ be a 
$\ok$-linear finite pivotal category. A  Calabi-Yau structure on 
$\proj(\cat)$ is compatible with duality on the right (left) 
if and only if the corresponding trace maps are non-degenerate and have  the 
right (left) partial trace property.
\end{theorem}
\begin{proof} 
We prove the right case only, the left one is similar. 
The one direction is an easy check. Indeed, assume $\t$ is a non-degenerate right modified trace on $\proj(\cat)$, and $a\in \Hom_{\cat}(U\otimes V,W)$ and $b\in \Hom_{\mathcal C}(W,U\otimes V)$, for $U,V,W\in\proj(\cat)$, then the top-right side of  diagram~\eqref{rcomp} gives  $\t_{U\otimes V}(b\circ a)$ while the left-bottom part gives $\t_U \bigl( \tr^r_V(b\circ a)\bigr)$. Then using~\eqref{rpartial} we conclude that diagram~\eqref{rcomp} commutes for $\catCY=\proj(\cat)$.

It remains to show the necessary condition. Let
 $\{t_P\,|\, P\in \proj(\cat)\}$ be CY structure on $\proj(\cat)$
 compatible with duality on the right.
We need to establish the right partial trace property
\eqref{rpartial}.
 By Reduction Lemma~\ref{lem:tHH-tH}, 
it is enough to consider the case where $U=V=G$ for $G$ a projective generator. Let us  also fix $W=G\tensor G$ and choose $b=\id_{G\tensor G}$ and any $a\in\End_\cat(G\tensor G)$.
Then by the assumption and using the previous calculation, commutativity of the diagram~\eqref{rcomp} gives the equality  $t_{G\otimes G}(a) = t_G \bigl( \tr^r_G(a)\bigr)$ which by Reduction Lemma~\ref{lem:tHH-tH} implies that $t$ is a right modified trace.
\end{proof}

\section{Pivotal Hopf algebras}\label{sec:Hopf}
In this section, we first recall standard facts from theory
of finite-dimensional Hopf algebras which will be needed later and then prove Proposition \ref{prop:mod-tr-ss}.
The main reference is the book~\cite{Ra-book}.
In what follows,  $\UH$ will be a finite-dimensional Hopf algebra over a field $\ok$ with the unit $\one$, multiplication  $\mu$, counit~$\epsilon$, coproduct~$\Delta$, and  antipode $S$. In this case, the antipode is invertible \cite{Iovanov}. 
In addition, we show that  if $H$ is a unimodular pivotal Hopf algebra,
then  $\Hpmod$  admits
 a non-degenerate and unique up-to-scalar right modified trace, or equivalently a Calabi-Yau structure compatible with  duality on  the right, and a similar statement for the left property.

\subsection*{Pivot}
We will say that an element $g\in H$ is \textit{group-like} if 
$\Delta(g)=g\tensor g$. It follows \cite[Prop.\,III.3.7]{Kassel} that $g$ is invertible, $S(g)=g^{-1}$ and $\epsilon(g)=1$.

\begin{definition}
A group-like element $\balance\in \UH$  is called a \textit{pivot} if
\be
S^2(x) = \balance x \balance^{-1}, \qquad \text{for all}\quad x\in \UH.
\ee
The pair  $(\UH,\pivot)$ of a Hopf algebra $H$ and a pivot $\pivot$  is called  a \textit{pivotal} Hopf algebra. 
\end{definition}

A  pivot $\pivot$ in  a Hopf algebra, if it exists, is not necessarily unique. For  a group-like  element~$z$ in the center of $H$, the product  $z\balance$ is also a pivot.
We will therefore indicate the choice of a pivot explicitly by the notation
 $(\UH,\pivot)$.

\subsection*{Examples}
Let $G$ be a finite  group. Then its group algebra 
$\ok[G]$ is a finite-dimensional pivotal Hopf algebra with 
$\pivot=\one$. 
\\
Ribbon Hopf algebras defined e.g.\ in \cite{Turaev_book}
are pivotal Hopf algebras.
The canonical  choice of a pivot is given by $\balance = \sqs \ribbon^{-1}$, where  $\sqs = \mu\circ (S\tensor \id) (R_{21})$ is the canonical Drinfeld element, and $\ribbon$ is the ribbon element. 
\\
Many more examples can be constructed as follows.
Any Hopf algebra $H$
can be extended to a pivotal Hopf algebra as follows~\cite[Sec.\,2.1]{AAGTV}. 
Recall that $S$ is invertible and order of $S^2$ is  finite.
Let $G$ be  the cyclic group generated by $S^2$ and set $\pivot=S^2$.
 We can then consider the smash product  of $H$ with $\ok G$. The result is a pivotal Hopf algebra with the pivot $\pivot$.

\subsection*{Symmetrised left and right integrals}
For any pivotal Hopf algebra $(\UH,\pivot)$ with the right integral $\rint$,  the \textit{symmetrised} 
right
integral $\rintg$ is defined by 
$\rintg(x):=\rint(\pivot x)$, for  $x\in \UH$.
Applying~\eqref{eq:rint-def} for $\balance x$ we get the  relation  for $\rintg$: 
\be\label{eq:mug-def-rel}
 (\rintg\tensor\balance)\Delta(x)=\rintg(x)\one\ .
\ee
We note that  relation~\eqref{eq:mug-def-rel} defines $\rintg$ uniquely (up to a scalar) because of up-to-scalar uniqueness of $\rint$ and invertibility of the pivot $\balance$.

Analogously,  the {\em symmetrised left integral} is defined by
$\lintg(x):=\lint(\pivot^{-1} x)$ 
for any $x\in \UH$.
Applying~\eqref{eq:lint-def} for $\pivot^{-1} x$ we get 
the defining relation  for the symmetrised left integral: 
\be\label{eq:mug-def-rel-left}
 (\pivot^{-1}\tensor\lintg)\Delta(x)=\lintg(x)\one
 \quad {\text{for any}} \quad x\in \UH \ .
\ee

 We note that the spaces of left and right integrals 
are not necessarily equal.
 We have a simple lemma.
 \begin{lemma}\label{lem:left-s}
The left integral can be chosen as $\lint(x) = \rint(S(x))$.
 \end{lemma}
 \begin{proof}
From \eqref{eq:rint-def} we have 
 $(\rint\tensor\id)\Delta(S(x))=\rint(S(x))\one\ $ for any $x\in H$.
 Using the identity $(S\otimes S)\Delta^{\op}(x)=\Delta(S(x))$ 
we get
$$ (\rint \circ S \otimes S)\Delta^{\op}(x)=
(S\otimes \rint\circ S)\Delta(x) =\rint( S(x))\one$$
Applying $S^{-1}$ to both sides of the last equality and using
$S^{-1}(\one)=\one$, we obtain that $\rint\circ S$ satisfies  the defining equation for a left integral.
 \end{proof}

\noindent{\bf Example.}
If $\UH$ is semisimple with $S^2=\id$,  
we can choose $\pivot=\one$ and then
$\rint=\rintg=\lint=\lintg$ is the character
of the regular representation \cite[Prop. 10.7.4]{Ra-book}.
\vspace{1mm}

\begin{proposition}[\cite{Radford}]\label{prop:int-non-deg}
Let $H$ be a finite-dimensional Hopf algebra. Then right and left integrals are non-degenerate linear forms. 
\end{proposition}
\begin{proof}
Let us first prove the non-degeneracy of $\rint$.
For any $h \in H$ we set
 $\rint_h(-):=\rint(h\cdot-)$.
By \cite[Theorem 10.2.2(e)]{Ra-book}, 
$H^*$ is a free  $H$-module with
basis $\{\rint\}$,
where the action by $a\in H$ sends $\rint$ 
to $\rint_{S(a)}$.
 This means that
 for any  non-zero $b\in H$, there exist $b'$ such that $\rint(bb')\neq 0$,
 since $S$ is bijective. This proves that left kernel of $\rint$ is trivial. Since $H$ is finite-dimensional,
 $\rint$ is non-degenerate. Non-degeneracy of $\lint$
 follows from Lemma \ref{lem:left-s} and non-degeneracy
 of~$\rint$.
\end{proof}

\subsection*{Unimodular Hopf algebras}
 A \textit{right cointegral} in $\UH$ is an element $\coint\in H$ such that
\be
x\coint=\epsilon(x) \coint\gc \qquad \text{  for all } \; x \in \UH\gp
\ee
Similarly, a \textit{left cointegral} is defined by  the equation $\coint x=\epsilon(x) \coint$. 
Non-zero right and left cointegrals exist in any finite-dimensional Hopf algebra and are unique up to scalar multiple~\cite{LarsonSweedler}.
A Hopf algebra is called {\em unimodular} if 
its right cointegral is also left. In this case, we call the cointegral \textit{two-sided}.

It is shown in \cite[Theorem 2]{Humphreys} that existence of a non-degenerate symmetric linear form on $H$
implies unimodularity. The argument is as follows. Let  $\coint$ and $\coint'$ be respectively right and left cointegrals. 
With respect to a non-degenerate symmetric linear form,
 both $\coint$ and $\coint'$ belong to the orthogonal complement of $\Ker(\epsilon\colon H\to \ok)$, which is 1-dimensional. Let us show the converse.

\begin{proposition}\label{lemma:mug-sym}
For a unimodular  pivotal Hopf algebra $(\UH,\pivot)$, the symmetrised right and left integrals 
 define  non-degenerate symmetric linear forms on $\UH$.
\end{proposition}
\begin{proof}  
 By Proposition
\ref{prop:int-non-deg}, the forms $\rint$ and $\lint$ are non-degenerate.
The shift of the left or right integral by an invertible element
preserves this property. Hence,
  $\rint_\pivot$ and $\lint_{\pivot^{-1}}$ are also non-degenerate.
By \cite[Thm.\ 10.5.4 (e)]{Ra-book}
we have
\be\label{444}
\rint(xy) = \rint\bigl(S^{2}(y) x\bigr)\ 
\ee
since in the unimodular case the distinguished
group-like element of $H^*$ is the counit $\epsilon$. 
Similarly, we have 
\be\label{eq:lintg-S}
\lint\left(S^{-2}(y)x\right)=\rint\left(S\left(S^{-2}(y)x\right)\right)=\rint\left(S(x)S^{-1}(y)\right)=
\rint\left(S(y)S(x)\right)=\lint(xy)
\ee
where we applied Lemma \ref{lem:left-s} for the first 
and last,
and
\eqref{444} for the third equalities.

By an easy computation, we check that $\rint_\pivot$ is symmetric:
$$\rint_\pivot(xy)=\rint(\pivot xy)=\rint(S^2(y)\pivot x)=
\rint(\pivot yx)=\rint_\pivot(yx)$$
where we used~\eqref{444}  and  $S^2(y)=\pivot y\pivot^{-1}$.
Similarly, using~\eqref{eq:lintg-S} we get
\begin{equation*}\lintg(xy)=\lint(\pivot^{-1} xy)=\lint(S^{-2}(y)\pivot^{-1} x)=
\lint(\pivot^{-1} yx)=\lintg(yx)\ .
\end{equation*}
\end{proof}

 By the previous proposition~\ref{lemma:mug-sym} we thus have two non-degenerate
 symmetric forms on a unimodular pivotal $H$, given by the symmetrised 
 left and right integrals. By 
Proposition~\ref{cor:ext} and 
 Theorem~\ref{Thm:nondegeneracy}
 they define two Calabi-Yau structures on $\Hpmod$.
 In other words we have

\begin{corollary}
The symmetric forms $\rintg$  and $\lint_{\pivot^{-1}}$
make a
unimodular  pivotal Hopf
 algebra $(\UH, \pivot)$ a symmetric Frobenius algebra.
 \end{corollary}

 We recall now  definition~\eqref{eq:HH0-alg-def} of $0^{\text{th}}$-Hochschild homology ${\rm HH}_0(H)$ of an algebra $H$.

\newcommand{\ch}{\mathrm{Ch}(H)}
\begin{proposition}
A  right symmetrised integral
on a unimodular pivotal Hopf algebra $H$
gives a non-degenerate symmetric pairing between the center $Z(H)$ and ${\rm HH}_0(H)$: 
\be\label{eq:pair-Z-HH0}
(z,h)\mapsto \rintg(zh) \gc \qquad z\in Z(H)\gc \; h\in {\rm HH}_0(H)\gp
\ee
Similarly, a left  symmetrised integral gives a non-degenerate symmetric pairing.
\end{proposition}
\begin{proof}
We first recall that a linear form $f$ on ${\rm HH}_0(H)$ satisfies $f(ab-ba)=0$, for $a,b\in H$, or defines a symmetric linear form on $H$.  
For a given non-degenerate symmetric form $t$, we have an isomorphism between the center  and  the space $\ch$ of symmetric forms on~$H$, see e.g.\ \cite[Lem.\,2.5]{Broue:2009}: 
\be
Z(H) \xrightarrow{\;\sim\;} \ch\gc \qquad z\mapsto t(z-)\gp
\ee
By Proposition~\ref{lemma:mug-sym}, we can
choose $t=\rintg$,  and therefore any linear form $f$ on ${\rm HH}_0(H)$ can be written as $\rintg(z-)$ for an appropriate $z\in Z(H)$. This is equivalent to non-degeneracy of the pairing~\eqref{eq:pair-Z-HH0}. The proof for a left symmetrised integral is similar.
\end{proof}

\subsection*{Unibalanced Hopf algebras}
We first recall that a right integral generates a one-dimensional right
 ideal of $\UH^*$, which is also a left ideal on
 $(H^*)^{\op}$, by the argument 
 in~\cite[p.\,306]{Ra-book} 
we have
\begin{equation}\label{eq:rint-a}
  (\id\tensor\rint)\Delta(x)=\rint(x)\comod ,
\end{equation}
for a certain $\comod\in \UH$ called \textit{comodulus} 
which  is   group-like. 
Multiplying \eqref{eq:rint-a} with $\comod^{-1}$
and evaluating at $\comod x$, we see that
the left and right integrals are related by the comodulus: 
\be\label{eq:lint=rint}
\lint(x) = \rint(\comod x).
\ee
Recall that in Lemma~\ref{lem:left-s} we had another choice for $\lint(x)$ using the antipode. Let us show that these  two choices agree.

\begin{proposition}\label{pro:comod}
We have the equality $\rint(S(x)) =  \rint(\comod x)$.
\end{proposition}
\begin{proof}
By Lemma~\ref{lem:left-s} and~\eqref{eq:lint=rint}, both  $\rint(S(x))$ and $\rint(\comod x)$ are left integrals.
Then we clearly have $\rint(S(x))=\lambda\rint(\comod x)$, for some $\lambda\in\ok^\times$, because the left integral is unique up to a scalar. To compute the proportionality coefficient it is enough to evaluate both  forms  $\rint(S(-))$ and $\rint(\comod -)$ on one element, we choose it to be the left cointegral $\coint$. Without loss of generality, we will assume  $\rint(\coint)=1$, see~\cite[Thm.\,10.2.2 (b)]{Ra-book}. Then by~\cite[Eq.\,(10.4)]{Ra-book} we also have $\rint(S(\coint))=1$. Therefore,
\be\label{eq:1-eps-a}
1= \rint(S(\coint))  = \lambda \rint(\comod \coint) =\lambda  \epsilon(\comod) \rint(\coint)  = \lambda  \epsilon(\comod)\ .
\ee
Recall that $\comod$ is group-like and so $\epsilon(\comod)=1$, and therefore $\lambda=1$ from the above equality.
\end{proof}

  A pivotal Hopf algebra $(\UH,\pivot)$ is called \textit{unibalanced} if its right symmetrised integral is also left. 
For a given right integral, let us 
choose the left integral as $\lint = \rint \circ S$  
(compare in Lemma~\ref{lem:left-s}). Then in the unibalanced case we have the equality
\be\label{eq:rintg-lintg}
\rintg=\lintg\gp
\ee 
Indeed, we have $\lintg = \lambda\rintg$ for some $\lambda\in\ok^{\times}$ and to compute $\lambda$ we evaluate the symmetrised integrals on left cointegral~$\coint$.
We note that by~\cite[Eq.\,(10.4)]{Ra-book}   $\rint$
and $\lint$ take same non zero value on $\coint$, say $\rint(\coint)=\lint(\coint)=a\in\ok^\times$.
Then, we have  $a=\lintg(\coint) =\lambda \rintg(\coint)=a\lambda$, and so $\lambda=1$.

We have the following characterisation of the unibalanced case in terms of the comodulus~$\comod$.

\begin{lemma}\label{agsquare} 
A pivotal Hopf algebra $(\UH,\pivot)$
is unibalanced if and only if 
$\comod=\pivot^2$. 
\end{lemma}
\begin{proof}
Assume first that $\comod=\pivot^2$. Then
evaluating~\eqref{eq:rint-a} on $\pivot x$ we get
\be \label{left}
  (\pivot^{-1}\tensor\rintg)\Delta(x)=\rintg(x)\one \gp
\ee
which is the  defining relation for the symmetrised left integral, and therefore $\rintg = \lintg$.

For the other direction, assume now $(\UH, \pivot)$ is unibalanced, 
then applying
 \eqref{eq:lint=rint} to $\pivot^{-1}x$ and using~\eqref{eq:rintg-lintg} we get the equality 
 \be\label{eq:rint-comod}
 \rint\bigl((\comod \pivot^{-1}-\pivot)x\bigr)=0\gc \qquad \text{ for any }\; x\in\UH\gp
 \ee
By Proposition~\ref{prop:int-non-deg}, $\rint$ is non-degenerate. 
Therefore, the equality~\eqref{eq:rint-comod} holds if and only if $\comod \pivot^{-1}=\pivot$.
\end{proof}

Quantum groups at roots of unity  provide many examples of unimodular and unibalanced pivotal Hopf algebras, see details in Section~\ref{qgroups}.

\subsection*{Pivotal structure on $\Hmod$}\label{sec:pivot-Hopf}
For a pivotal Hopf algebra $(\UH,\pivot)$,
each object $V$ in $\Hmod$ has a left dual
$V^*=\Hom_\kk(V,\kk)$ with the $\UH$ action defined by $(hf)(x)=f(S(h)x)$, $f\in V^*$, $h,x\in \UH$,
while the action by $\balance$ corresponds to the 
 natural
isomorphism
$\delta$ between the identity functor on $\Hmod$ and the double duality functor $(-)^{**}$ . More precisely,
we have the family of isomorphisms
\be\label{eq:pivot-Hmod}
\delta_V\colon V\to V^{**}\ , \qquad \delta_V = \balance \circ \delta^{\vect}_V, \quad V\in\Hmod\ ,
\ee
where $\delta^{\vect}$ is the standard pivotal structure in the category $\vect_\ok$:
 $\delta^{\vect}_V(v) = \langle-,v\rangle$,
for the underlying vector space $V$, $v \in V$
	and $\langle-,-\rangle$ is the pairing between $V^*$ and $V$.
The  isomorphisms~\eqref{eq:pivot-Hmod} are obviously natural and monoidal.
We have therefore $\Hmod$ is pivotal.
 
In $\Hmod$, we have the standard left duality morphisms. 
Assume $\{v_j \,|\, j\in J\}$ is a basis of~$V$ and 
$\{v_j^*\,|\, j\in J\}$ is the dual basis of $V^*$, then
\begin{align}\label{E:DualitiesC}
\ev_{V}&: \:\: V^*\otimes V\rightarrow \kk, &\text{ given by } \quad &f\otimes v \mapsto f(v), \\ 
\coev_{V} &: \:\: \kk \rightarrow V\otimes V^{*}, & \text{ given by }\quad &1 \mapsto \sum_{j\in J} v_j\otimes v_j^*. \notag
\end{align}
The pivot $\pivot$ allows to define the right  duality morphisms as follows
\begin{align}\label{E:DualitiesC-right}
 \tev_V&: \:\: V\otimes V^*\rightarrow \kk, &\text{ given by }\quad & v\otimes f \mapsto f
(\pivot v)\\
\tcoev_V&: \:\: \kk \rightarrow V^{*}\otimes V, &\text{ given by }\quad & 
1 \mapsto \sum_i v_i^*\otimes 
\pivot^{-1}v_i \ , \notag
\end{align}
where we used the combination of~\eqref{eq:right-duality} and~\eqref{eq:pivot-Hmod}.

We recall the (right)
categorical trace~\eqref{eq:cat-tr-def} which is in our case
\be
\tr^{\Hmod}_V(f):= \tev_V\circ(f\tensor\id)\circ\coev_V(1)\gc
\ee
 for any $V\in\Hmod$ and $f\in\End_H(V)$.
With the  definitions above we have
\be \label{eq:q-tr}
\tr^{\Hmod}_V(f) =\tr_V(l_\pivot \circ f)
\ee
where $\tr_V(f)$ is the usual trace
 of the endomorphism $f$  of $V$. The trace~\eqref{eq:q-tr} is often called {\it quantum trace}.
Analogously, we can define the 
left
categorical trace
$$
^{\Hmod}\!\tr_V(f):= \ev_V\circ(\id\tensor f)\circ\tcoev_V(1)
$$
 for any $V\in\Hmod$ and $f\in\End_H(V)$.
Then we compute
\be\label{eq:lq-tr}
^{\Hmod}\!\tr_V(f) =
\sum_i v_i^* \left(f(\pivot^{-1} v_i)\right)=
\tr_V( f \circ l_{\pivot^{-1}}).
\ee
We note that the left and right traces are related. Indeed, using Lemma \ref{lem:reverse-arrows} for
$Q=\one$, $P=V$,
we have the relation
\be\label{eq:left-right}
^{\Hmod}\!\tr_V(f)=\tr^{\Hmod}_{V^*}(f^*)\ .
\ee

We are now ready  to prove Proposition \ref{prop:mod-tr-ss}.

\begin{proof}[\bf Proof of Proposition~\ref{prop:mod-tr-ss}]
We will assume that the right integral $\rint$ and the cointegral $\coint$  satisfy $\rint(\coint)=1$.
From \cite[Thm.\,10.4.1]{Ra-book}, for any $f\in\End_H(H)$, we  then have
\be\label{eq:trace_r} \tr_H(f)=\rint\bigl(S(\coint'')f(\coint')\bigr)\ee
and
\be\label{eq:trace_l}\tr_H(f)=\rint\bigl(S(f(\coint''))\coint'\bigr)  \ .\ee
We use here  Sweedler's notation with implicit sum: $\Delta(\coint)=\coint'\otimes \coint''$.
From Lemma \ref{eq:hop=end}, any $f\in\End_H(H)$ is right multiplication by $x=f(1)$, i.e.\ $f=r_x$.
The right categorical trace for $f=r_x$ is obtained from \eqref{eq:trace_r} as follows:
\begin{align*}
\tr_H(l_\pivot\circ f)&= \rint\bigl(S(\coint'')\pivot\coint'x\bigr)=\rint\bigl(S(\coint'')S^2(\coint') \pivot x\bigr)\\
&=\rint\bigl(S(S(\coint')\coint'') \pivot x\bigr)=\epsilon(\coint)\rintg(x)\gp
\end{align*}
We similarly get the left categorical trace using \eqref{eq:trace_l}
\begin{align*}
\tr_H(l_{\pivot^{-1}}\circ f)&= \rint\bigl(S(\pivot^{-1}\coint''x)\coint'\bigr)=\rint\bigl(S(x)S(\coint'')\pivot\coint'\bigr)\\
&=\rint\bigl(S(x)\pivot S^{-1}(\coint'')\coint'\bigr)=\rint\bigl(S(x)\pivot S^{-1}(S(\coint')\coint'')\bigr)\\
&=\epsilon(\coint)\rint\bigl(S(\pivot^{-1}x)\bigr)=\epsilon(\coint)\lintg(x)\gc
\end{align*}
where the last equality comes from the formula for a left integral in Lemma \ref{lem:left-s}. 
By \cite[Cor.\,10.3.3]{Ra-book}
 $\epsilon(\coint)$ is non-zero if and only if the algebra $H$ is semisimple. This shows that the categorical traces agree with~\eqref{eq:cat-tr-rintg} up to a non-zero scalar if and only if $\Hmod$ is semisimple.
\end{proof}

From Proposition~\ref{prop:mod-tr-ss}, we conclude
 that in the 
non-semisimple case $\tr_H(l_\pivot r_x)$ is zero 
for all $x\in H$,  while $\rintg(x)$ is not. 
This naturally suggests that $\rintg$ provides  a non-trivial generalisation of the categorical trace for the tensor ideal of projective $\UH$-modules, recall Lemma~\ref{lem:tHH-tH} for the case $G=H$. Such a generalisation indeed exists and is given by the (right) modified trace -- this is the content of our  Theorem~\ref{thm:main}. The proof is rather long and requires more preparation, we delegate it to Section~\ref{sec:proof}.
\vspace{2mm}

{\bf Remark.}
 Proposition~\ref{prop:mod-tr-ss} can also be  deduced directly from Theorem \ref{thm:main}. Indeed, the right symmetrised
 integral $\rintg$ gives a non-zero right modified trace on $H$, which is
 unique up to a scalar.
As we mentioned in Introduction, the right categorical trace is also a right modified trace. However, the right categorical trace is non-zero on $H\in\Hpmod$ if and only if $\Hpmod$ is semisimple, see e.g.~\cite[Rem.\,4.6]{Azat-Runkel}, or equivalently if and only if $\Hmod$ is semisimple. Therefore,  the two traces  agree if and only if $H$ is semisimple as an algebra.
Similar argument applies for the left categorical trace. 
\vspace{2mm}

It is interesting to note an application of  Theorem~\ref{thm:main} in the classical context -- to the modular representation theory of finite groups. 
Let $G$ be a finite group  and consider its group algebra $\mathbb F_p[G]$ over the field $\ok=\mathbb F_p$ when  the characteristic $p$ divides the order of the group. It is a unimodular pivotal Hopf algebra with $\pivot=\one$ and the two sided cointegral is $\coint = \sum_{g\in G} g$. So, the symmetrised integral in this case is just the integral and it  provides a non-degenerate modified trace on the subcategory of projective $\mathbb F_p[G]$-modules. To our knowledge, such modified traces were not observed in this generality.
 However we should also mention that existence and non-degeneracy of the modified trace 
 in the finite characteristic case was proven in~\cite{Azat-Runkel} in the case of Drinfeld doubles of $\mathbb F_p[G]$ and under an extra technical assumption, which did not work e.g.\ in the case of abelian $p$-groups.

As another corollary of   Theorem~\ref{thm:main}  and Theorem~\ref{thm:CYcomp} 
we conclude this section with the following (c.f.\ Corollary~\ref{cor:1.1}). 

\begin{corollary}\label{cor:CYcomp-H}
Let $(H, \pivot)$ be a unimodular pivotal Hopf algebra. Then $\Hpmod$ admits a unique up-to-scalar CY structure  compatible with duality on the right, and a possibly different  CY structure  compatible with duality on the left. The  CY structure on $\Hpmod$ is compatible  with duality on  the right and the left if and only if $H$ is unibalanced.
\end{corollary}

\newcommand{\Heps}{{}_\eps H}
\newcommand{\Weps}{{}_\eps W}

\section{Decomposition of  tensor powers of the regular representation}\label{sec:HH-decomp}
In this section, for any finite-dimensional Hopf algebra $H$
we decompose the tensor product of the regular representation with itself and describe
the centralizer algebra $\End_H (H^{\otimes 2})$ explicitly, and more generally $\End_H(H\otimes W)$ for any $W\in \Hmod$.
We will need these endomorphism algebras to prove our main theorem in  next Section~\ref{sec:proof}.

\subsection*{Diagrammatics for Hopf algebras}\label{sec:conv}
We will use the following diagrams for
the structural maps corresponding to the Hopf algebra data:
\be
\mu \ \; = \  \; \ipic{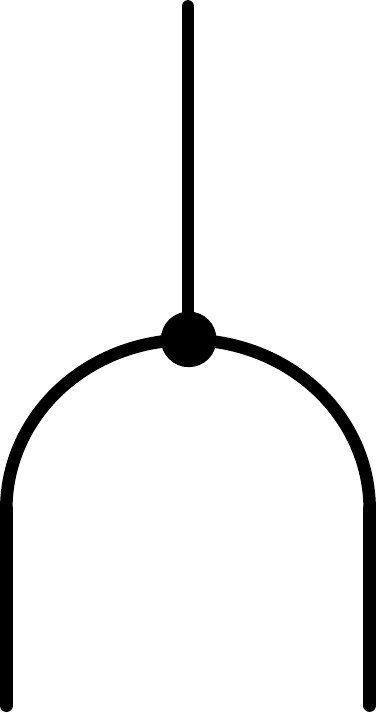}{.16} \quad ,
\quad \Delta \ \;  =\ \;  \ipic{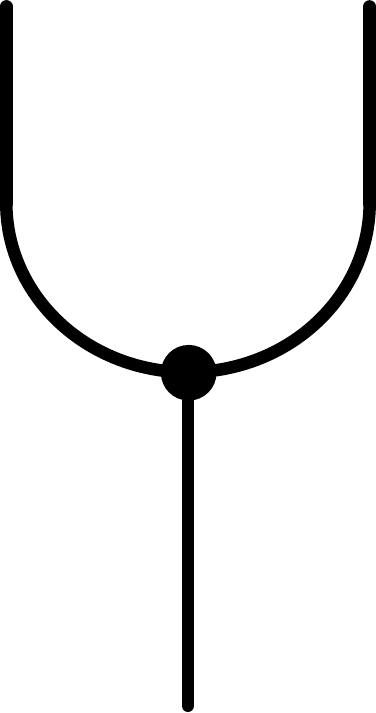}{.16} \quad , \quad
 \eta \ \; = \ \; \ipic{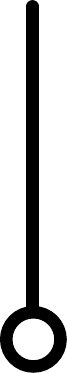}{.17} \quad , \quad
  \eps \ \; = \ \; \ipic{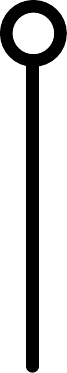}{.17} 
\quad , \quad S \; = \;  \ipic{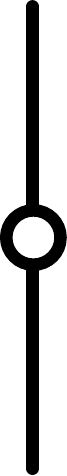}{.17} \quad .
\put(-336,30){\scriptsize $H$}  \put(-351,-36){\scriptsize $H$} \put(-323,-36){\scriptsize $H$}  
\put(-240,-36){\scriptsize $H$}  \put(-254,30){\scriptsize $H$} \put(-225,30){\scriptsize $H$} 
\put(-157,18){\scriptsize $H$} \put(-87,-24){\scriptsize $H$} 
\put(-22,-28){\scriptsize $H$} \put(-22,22){\scriptsize $H$}
\ee
We note that these are maps in the category $\vect_\ok$ of finite-dimensional vector spaces over~$\ok$.
 Here is  a list of  graphical identities corresponding to the Hopf algebra axioms we use extensively below:
\be
 \ipic{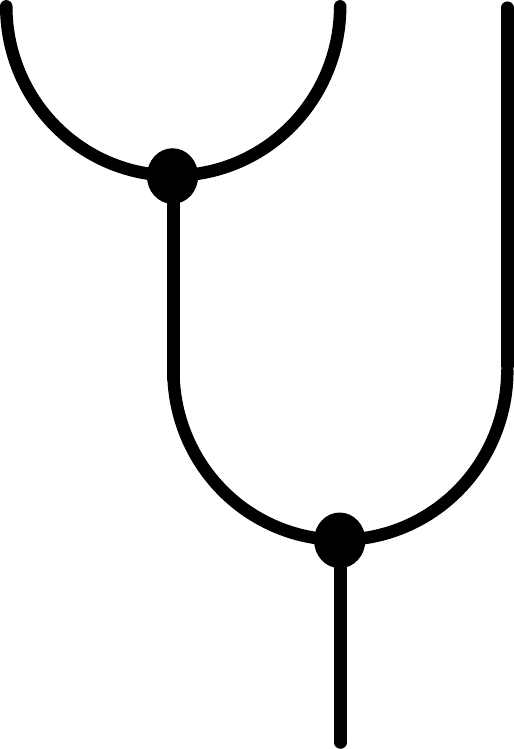}{.16} \  = \ 
 \ipic{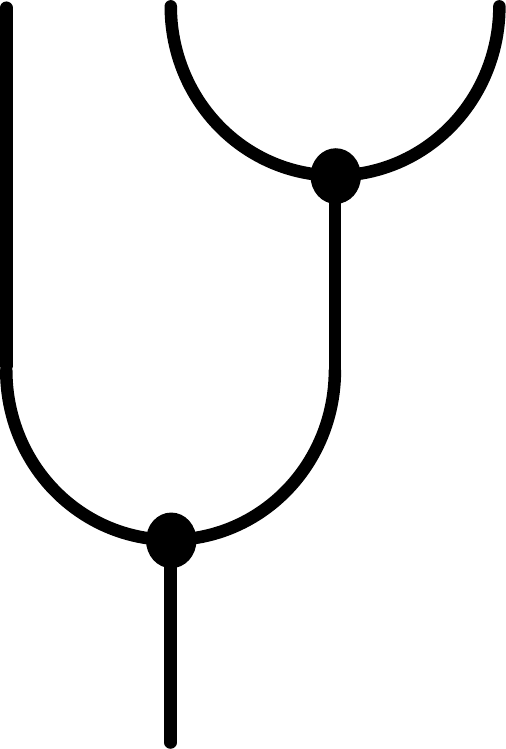}{.16}\qquad, \qquad
  \ipic{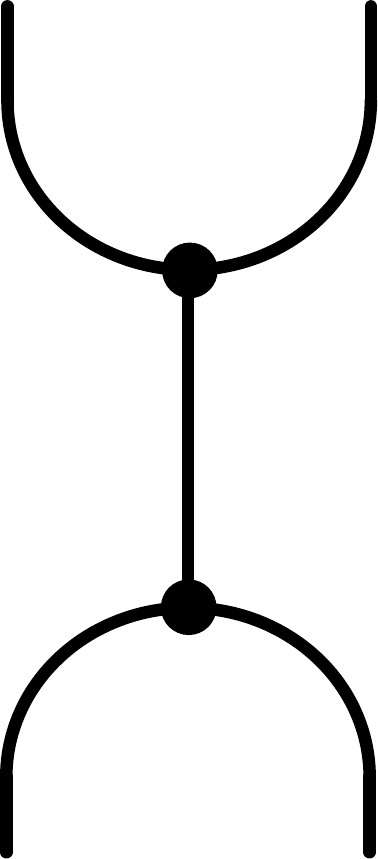}{.16} \  = \ \; \ipic{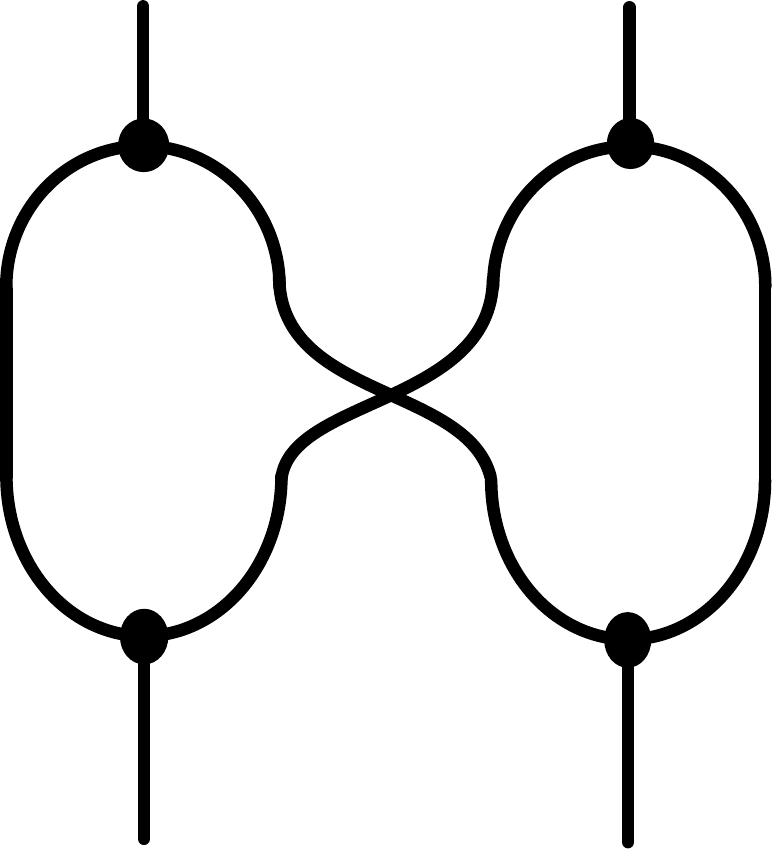}{.16}  
  \put(-272,31){\scriptsize $H$} \put(-248,31){\scriptsize $H$} \put(-234,31){\scriptsize $H$} 
  \put(-249,-37){\scriptsize $H$}  
    \put(-119,34){\scriptsize $H$}    \put(-91,34){\scriptsize $H$}  
  \put(-120,-40){\scriptsize $H$}    \put(-92,-40){\scriptsize $H$}  
\ee
where the first is for coassociativity, the second  says that $\Delta$ is an algebra map, and the antipode axioms
(here, we skip labels $H$ for brevity)
\be\label{eq:S-axiom}
  \ipic{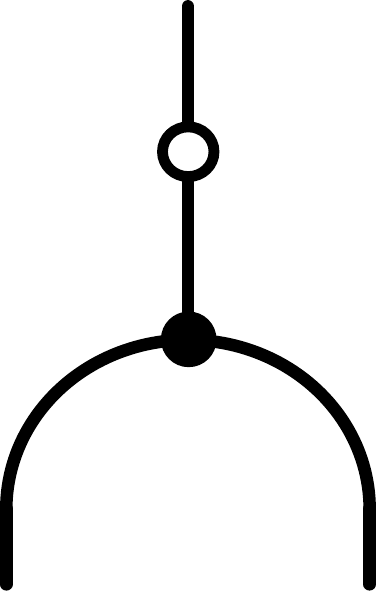}{.18} \  = \ \; \ipic{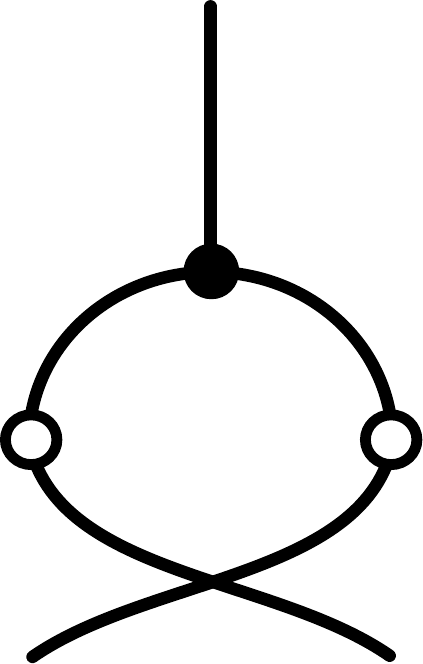}{.18}   \qquad, \qquad
  \ipic{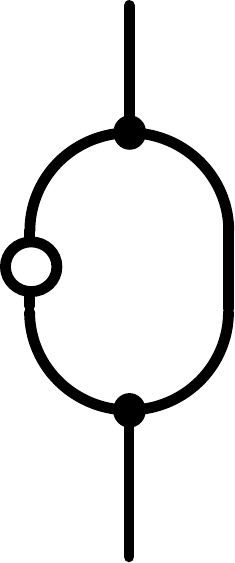}{.22} \  = \ \ipic{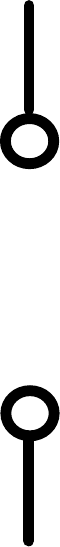}{.22}  \ = \ \ipic{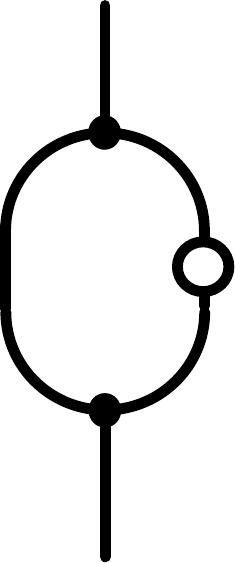}{.22}\qquad,  \qquad
  \ipic{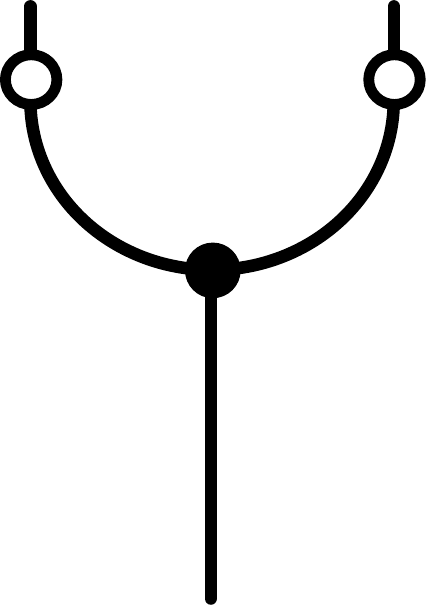}{.18} \  = \ \; \ipic{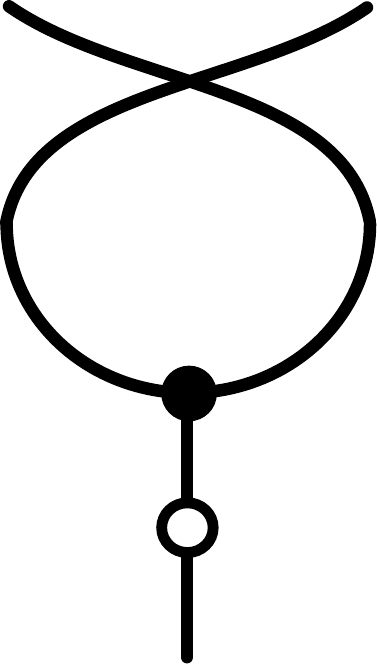}{.18}  
\ee
where the first and third  say that $S$ is an anti-algebra and anti-coalgebra map, respectively.
The axioms involving unit and counit are rather clear and we omit them.

\subsection*{The case of $H^{\otimes 2}$ }
Let us denote by 
$\Heps$ the vector space underlying $\UH$  equipped with the trivial action of $\UH$, i.e.\ for $m\in \Heps$ and $h\in H$ we have $h m = \eps(h)m$.
As a $H$-module, $\Heps$ is isomorphic to $\dim\UH$ copies of the trivial representation. We use  Sweedler's notation with implicit sum: $\Delta(h)=h'\otimes h''$.

\begin{theorem}\label{thm:tensor-powers}
We have for all $h\in H$ and $m\in \Heps$
\begin{enumerate}
\item[$(a)$]
the map
\be\label{eq:phi-def}
\begin{array}{rrcl}
\phi:&\UH \otimes \Heps&\rightarrow &\UH\otimes \UH\\
&h\otimes m&\mapsto&h'\otimes h''m
\end{array}
\ee
 is an isomorphism of $\UH$-modules whose inverse is
\be
\begin{array}{rrcl}
\psi:&\UH \otimes \UH&\rightarrow &\UH\otimes \Heps\\
&x\otimes y&\mapsto&x'\otimes S(x'')y\ ;
\end{array}
\ee

\item[$(b)$]
the map 
\be\label{eq:phil-def}
\begin{array}{rrcl}
\phi^l:&\Heps \otimes \UH&\rightarrow &\UH\otimes \UH\\
&m\otimes h&\mapsto&h'm\otimes h''
\end{array}
\ee
 is an isomorphism of $\UH$-modules whose inverse is
\be\label{eq:psil-def}
\begin{array}{rrcl} 
\psi^l:&\UH \otimes \UH&\rightarrow &\Heps\otimes \UH\\
&x\otimes y&\mapsto&S^{-1}(y')x\otimes y'' \ .
\end{array}
\ee
\end{enumerate}
\end{theorem}
This theorem is a well-known result in Hopf-algebras theory, see\ e.g.~\cite{Schneider}, however we give 
a proof to demonstrate graphical calculations that are often used below.
For the maps $\phi$ and $\psi$ we have the expressions
\be
\ipic{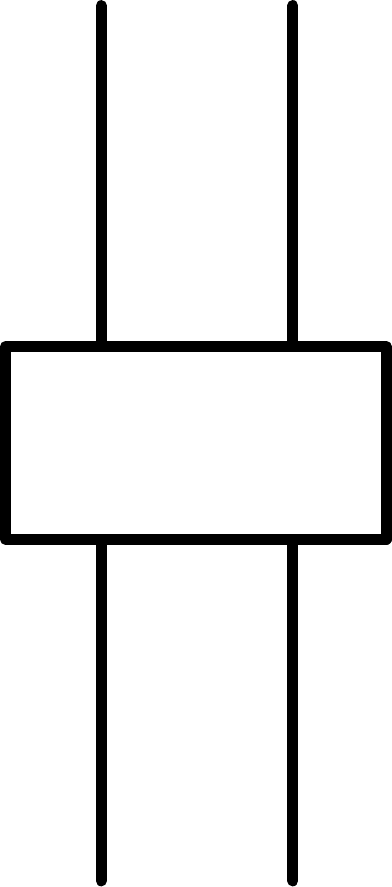}{.15} 
\put(-17,-2){\footnotesize $\phi$}
\put(-25,34){\scriptsize $H$} \put(-10,34){\scriptsize $H$} 
 \put(-25,-40){\scriptsize $H$}   \put(-13,-40){\scriptsize $\Heps$}
  \quad = \quad 
\ipic{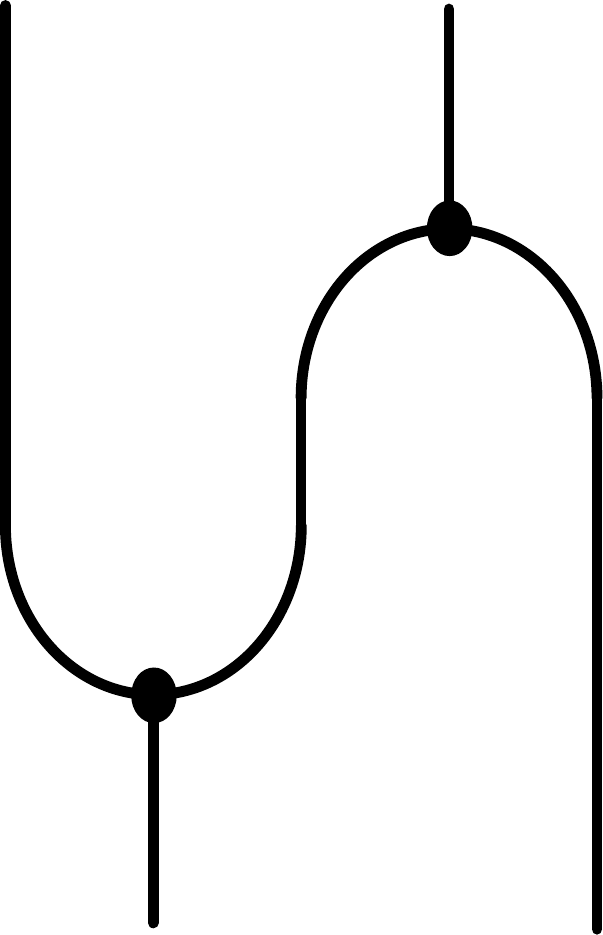}{.15} 
\put(-48,36){\scriptsize $H$} \put(-16,36){\scriptsize $H$} 
 \put(-37,-42){\scriptsize $H$}   \put(-8,-42){\scriptsize $\Heps$}
\qquad \qc \qquad \qquad
\ipic{rect-map}{.15} 
\put(-17,-2){\footnotesize $\psi$} 
\put(-25,34){\scriptsize $H$} \put(-13,34){\scriptsize $\Heps$} 
 \put(-25,-40){\scriptsize $H$}   \put(-10,-40){\scriptsize $H$}
 \quad = \quad 
\ipic{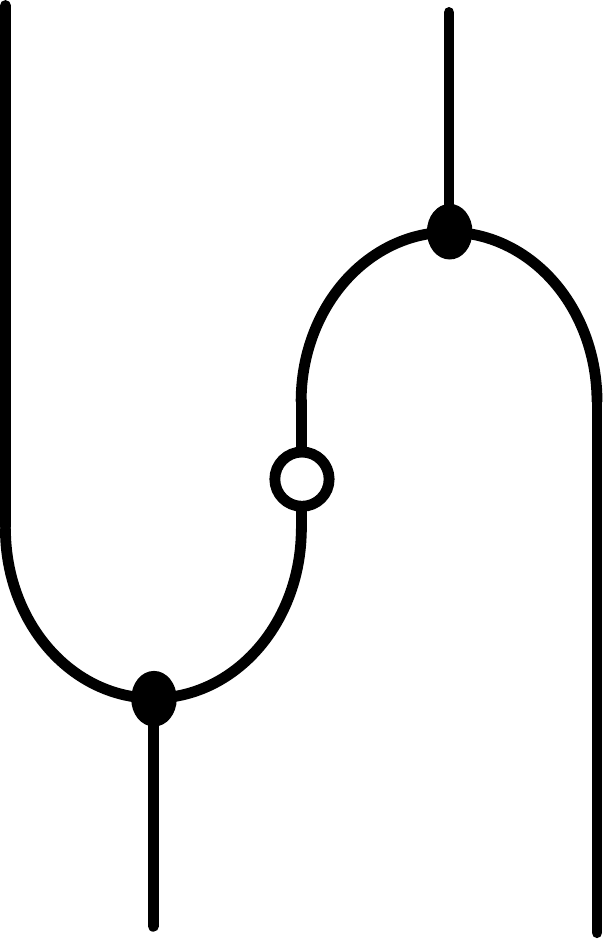}{.15}
\put(-48,36){\scriptsize $H$} \put(-19,36){\scriptsize $\Heps$} 
 \put(-37,-42){\scriptsize $H$}   \put(-5,-42){\scriptsize $H$}
\ee
and similarly for $\phi^l$ and $\psi^l$.

\begin{proof}
We begin with the part (a) and first check that
 $\psi$ is left inverse to $\phi$,  we thus compute the composition
\be
\psi\circ \phi \quad = \quad
\ipic{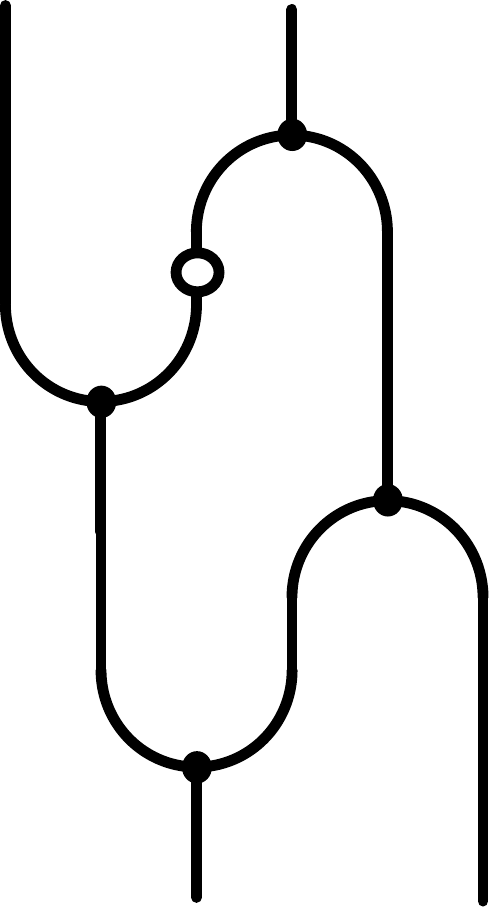}{.2}  
 \quad = \quad 
\ipic{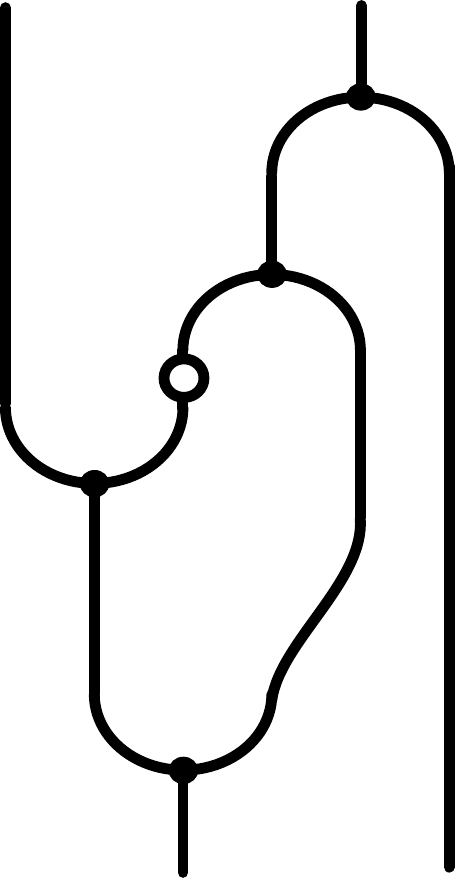}{.2}  
 \quad = \quad 
\ipic{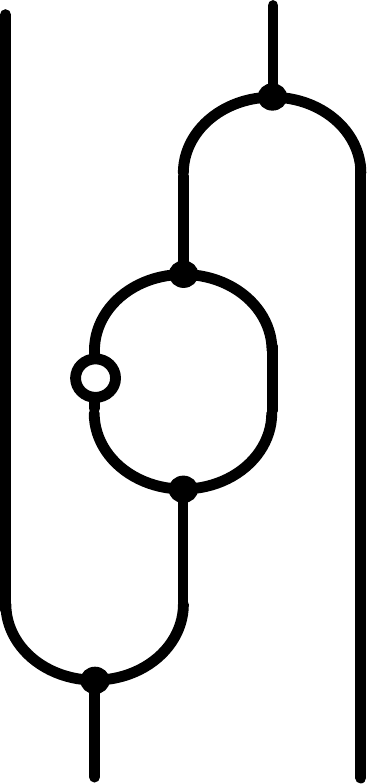}{.21}  
 \quad = \quad 
\ipic{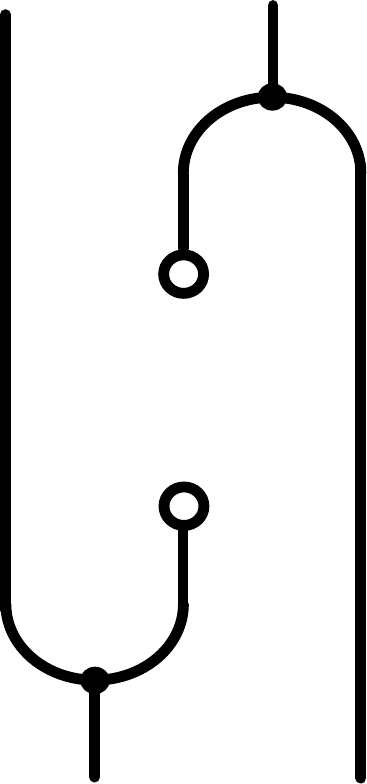}{.21}  
 \quad = \quad 
 \id_{H\tensor \Heps}
 \ee
 where we used  coassociativity of the coproduct in the third equality, and then the antipode axiom.
Since the left and right inverses of a linear endomorphism of a finite-dimensional space are always equal, we also have
 $\phi\circ \psi = \id_{H\tensor H}$.
Then we check that  $\phi$  intertwines the corresponding $H$ actions:\footnote{We show explicitly the source and target labels, $H$ in this case, only on LHS for brevity. }
\be
\ipic{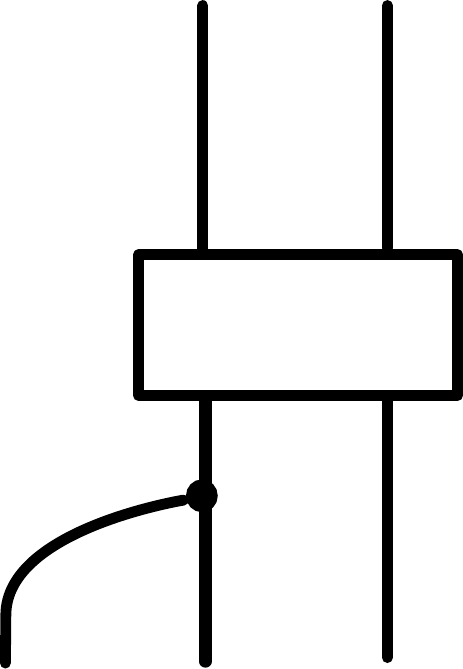}{.2} 
\put(-20,-2){\footnotesize $\phi$}
\put(-29,34){\tiny $H$} \put(-10,34){\tiny $H$} 
 \put(-49,-40){\tiny $H$} \put(-29,-40){\tiny $H$}   \put(-13,-40){\tiny $\Heps$}
  \quad = \quad 
  \ipic{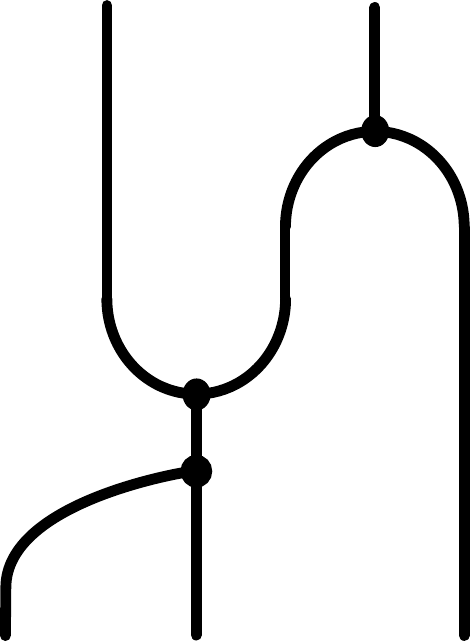}{.2} 
   \quad = \quad
     \ipic{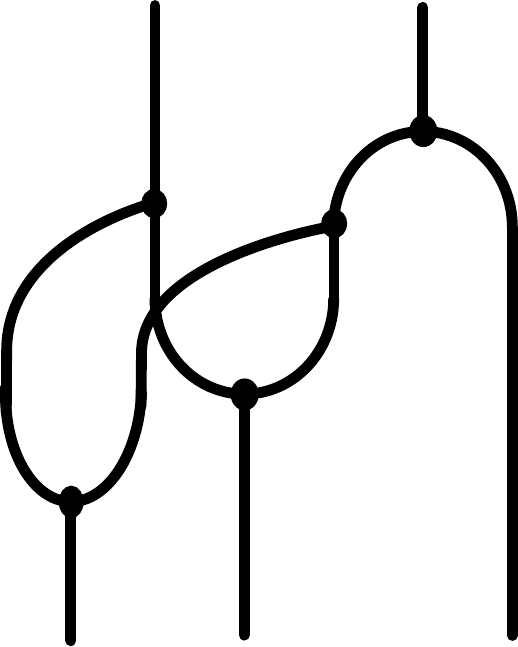}{.2} 
   \quad = \quad
\ipic{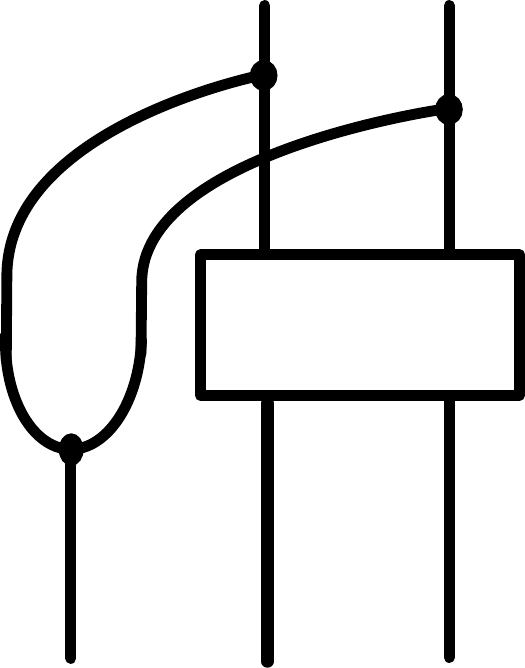}{.2} 
\put(-20,-2){\footnotesize $\phi$}
\ee
where we used the property of coproduct being an algebra map and associativity of multiplication. Clearly, the inverse map of an intertwiner is automatically an intertwiner. Therefore, it proves that $\psi$ is an intertwiner as well. 
The part b) is proven in an analogous way. 
\end{proof}

\newcommand{\basis}{B}
From  Theorem~\ref{thm:tensor-powers} we  obtain  two corollaries: the first is about an explicit decomposition of $\UH\tensor \UH$ while the second contains a description of the centraliser algebra of the $\UH$-action on $\UH\tensor \UH$. First, we need a little preparation. Let us fix a basis $\basis$ of $H$, 
it is a finite set. We introduce then two families of intertwining maps:
\begin{align}\label{eq:g-fy-def}
g_y &\colon \UH \to \UH\tensor \Heps\ ,\qquad h\mapsto h\tensor y\ ,\qquad\;\; y\in \basis\ ,\\
f_y &\colon \UH\tensor \Heps \to \UH\ ,\qquad h\tensor u\mapsto \delta_{u,y} h\ ,\quad u,y\in \basis\ ,\nonumber
\end{align}
where $\delta$ is the Kronecker symbol, and the last map we extend linearly to the whole space $\UH\tensor \Heps$. It is clear that $f_{y'}\circ g_y =\delta_{y',y} \id_H$ and $g_y\circ f_y$ is an idempotent for each $y\in \basis$. The intertwining property of $g_y$ and $f_y$ is very straightforward to see. From this and from the isomorphisms established in Theorem~\ref{thm:tensor-powers} we have the following corollary.

\begin{corollary} \label{cor:HH-decomp}
Let $\UH$ be the regular module of a Hopf algebra $\UH$ and $\basis$ be a 
basis of~$\UH$. We have then the decomposition
\be\label{eq:HH-decomp}
\UH\tensor \UH\; \cong\; \bigoplus_{y\in\basis} \UH_y
\ee
where each direct summand $\UH_y$ is the regular $\UH$-module and the corresponding idempotent $e_y$ is given by the composition
\be\label{eq:idemp-y}
e_y = \iota_y \circ \pi_y
\ee
with the monomorphisms
\be
\iota_y \colon \UH \to \UH\tensor \UH\ , \qquad h\mapsto \phi\circ g_y(h) =
h'\tensor h'' y \ , \quad y\in \basis
\ee
and the epimorphisms
\be
\pi_y \colon \UH\tensor \UH \to \UH\ , \qquad h\tensor u\mapsto f_y\circ\psi(h\tensor u) \ , \quad y\in \basis \ ,\ u\in \UH\ .
\ee
In other words, the image of $\iota_y$ is $\UH_y$ in~\eqref{eq:HH-decomp} and $\pi_y$ is identity on $\UH_y$.  
\end{corollary} 
\begin{proof}
The direct sum decomposition~\eqref{eq:HH-decomp}  clearly follows from Theorem~\ref{thm:tensor-powers} where the corresponding isomorphisms  $\phi$ and $\psi=\phi^{-1}$ are given.
That $\iota_y$ is an intertwiner is clear from the definition $\iota_y:= \phi\circ g_y$ as the composition of two intertwining maps. And the same applies to $\pi_y$. 
The idempotent property of $e_y = \phi \circ g_y \circ f_y \circ \phi^{-1}$ follows from  that of $g_y \circ f_y$.
The image of $e_y$ is $\UH_y\subset \UH\tensor\UH$ and $e_y$ is identity on $\UH_x$ if and only if $x=y$ for $x,y\in \basis$. This finishes the proof.
\end{proof}

From~\eqref{eq:idemp-y}, we also note the equalities
\be
e_y e_{x} = \delta_{y,x} e_y\ ,\quad  x,y\in \basis
\ee
and 
\be\label{eq:ey}
\sum_{y\in\basis} e_y = \id_{\UH\otimes \UH}\ .
\ee

Before formulating the second corollary of Theorem~\ref{thm:tensor-powers}, we recall 
that for any $\ok$-algebra~$A$ there is a natural isomorphism
$\Mat_{n,n}(A)\cong A \otimes \Mat_{n,n}(\ok)$,
where $\Mat_{n,n}$ is the $n\times n$ matrix algebra.
\newcommand{\f}{\boldsymbol{f}}
\begin{corollary}\label{cor:EndHH-Mat}
For  any $n$-dimensional Hopf algebra $H$, there is an algebra isomorphism 
\be\label{eq:EndHH-Mat}
\End_\UH(\UH\tensor \UH) \cong \Mat_{n,n}(\UH^{\op})\ .
\ee 
Hence, the algebra $\End_\UH(\UH\tensor \UH)$ is linearly generated by elements
$\f$ parametrised by  triples $(h,v,\gamma)$, for $h,v\in\UH$ and  $\gamma\in\UH^*$, where
\be\label{eq:tf-hvg}
\f(h,v,\gamma):= \phi \circ f(h,v,\gamma) \circ \psi \, \colon \;\; \UH\tensor \UH \to \UH\tensor \UH
\ee
with
\be\label{eq:f-hvg}
f(h,v,\gamma)\colon\, x\tensor y \mapsto \gamma(y)\cdot (xh)\tensor v \ , \qquad x\in\UH, \ y\in\Heps\gp
\ee
Their product is the composition with
\be\label{eq:ff-mult}
f(h_1,v_1,\gamma_1)\circ f(h_2,v_2,\gamma_2) = \gamma_1(v_2) f(h_2 h_1,v_1, \gamma_2) \gp
\ee
\end{corollary}

Here is  the graphical presentation of the maps 
$f(h,v,\gamma)$ and $\f(h,v,\gamma)$:
\be\label{eq:ffhvg-map}
\ipic{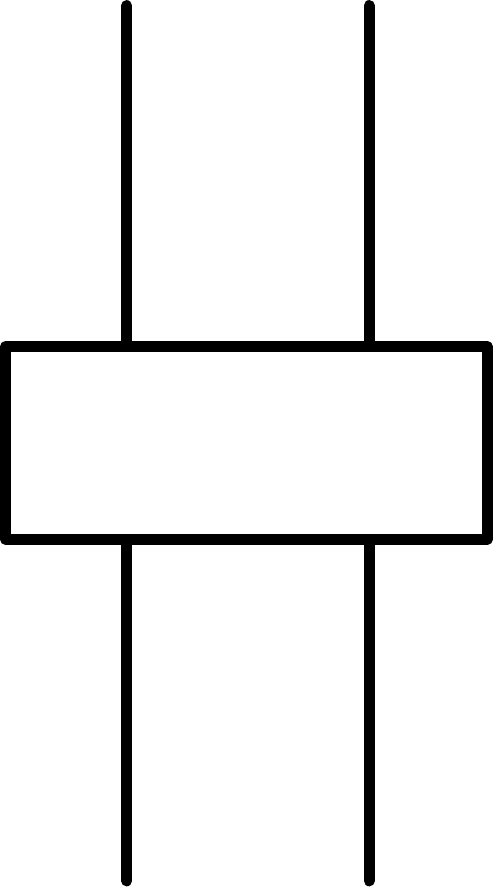}{.17} 
\put(-38,-2){\scriptsize $f(h,v,\gamma)$}
\put(-33,38){\scriptsize $H$} \put(-15,38){\scriptsize $\Heps$} 
 \put(-34,-44){\scriptsize $H$}   \put(-18,-44){\scriptsize $\Heps$}
  \quad = \quad 
\ipic{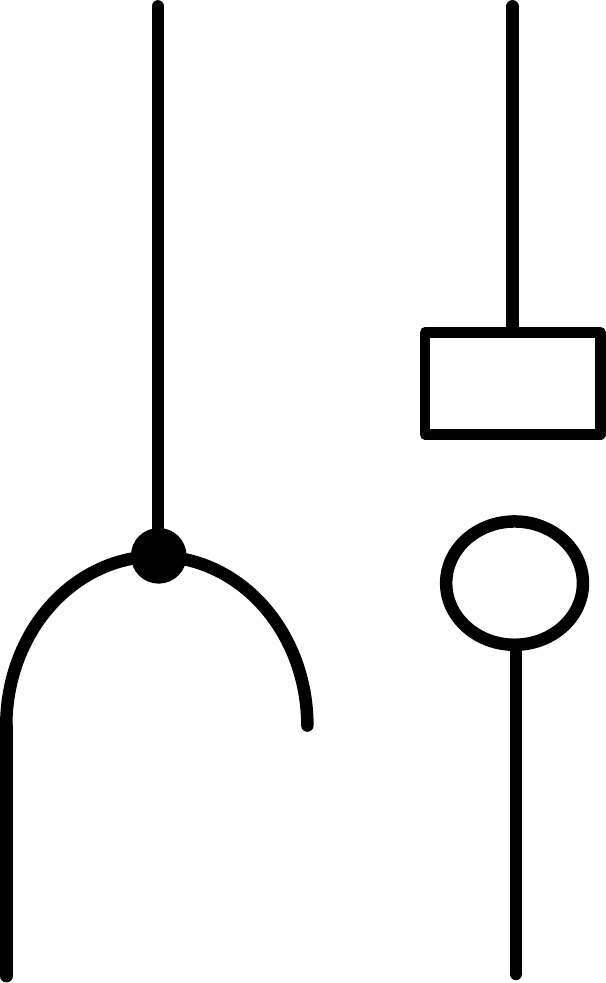}{.15} 
\put(-24,-25){\scriptsize $h$} \put(-9.5,-8){\scriptsize $\gamma$}
 \put(-9.5,6){\scriptsize $v$}
\quad \gc \qquad
\ipic{rect-map-wide}{.17} 
\put(-38,-2){\scriptsize $\f(h,v,\gamma)$}
\put(-33,38){\scriptsize $H$} \put(-13,38){\scriptsize $H$} 
 \put(-34,-44){\scriptsize $H$}   \put(-15,-44){\scriptsize $H$}
  \quad = \quad 
  \ipic{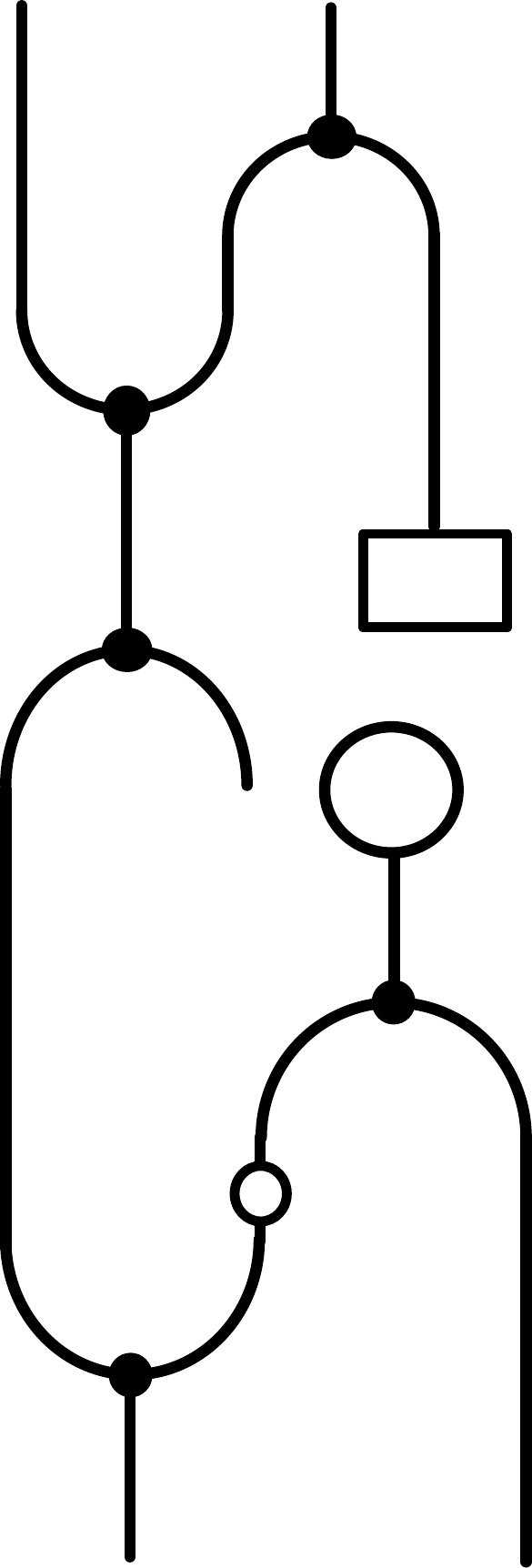}{.15} 
\put(-25,-8){\scriptsize $h$} \put(-13.5,-2){\scriptsize $\gamma$}
 \put(-9.5,14){\scriptsize $v$}
  \quad \gp
\ee

\begin{proof}
We first recall the decomposition~\eqref{eq:HH-decomp} where the multiplicity space is the vector space underlying $H$. We will denote it $M:=H$ in order to distinguish from the regular module $H$. 
We have then  isomorphisms\footnote{Using $\tensor_{\ok}$ we distinguish the tensor product of vector spaces from the one for $H$-modules.}
\be\label{eq:HH-MM}
\End_\UH(\UH\tensor \UH) \cong \End_\UH(\UH\tensor_{\ok} M) \cong \Hom_\UH(\UH\tensor_{\ok} M\tensor_{\ok} M^*, \UH)\ ,
\ee
where in the last isomorphism we used the duality maps $\ev_M$ and $\coev_M$. We note that RHS of~\eqref{eq:HH-MM} is obviously isomorphic to $ \End_\UH(\UH)\tensor_{\ok} \Mat_{n,n}(\ok)$ with $n=\dim \UH$.
Then by  Lemma~\ref{lem:Uop-End}  we get an isomorphism of vector spaces in~\eqref{eq:EndHH-Mat}.
Let us describe this isomorphism explicitly. First, we construct the isomorphism
\begin{align}
\Phi\colon\,&  \UH^{\op} \tensor\big(M\tensor_\ok M^*\big) \xrightarrow{\sim} \End_\UH(\UH\tensor_\ok M)\ ,\\
&h\tensor v\tensor \gamma \mapsto  f(h,v,\gamma)
\end{align}
with $f(h,v,\gamma)$ from~\eqref{eq:f-hvg}. It is straightforward to check that $f(h,v,\gamma)$ is an intertwiner.  
The inverse to the map $\Phi$ is defined as follows. 
Elements in $\End_\UH(\UH\tensor_{\ok} M)$ are of the form
\be
g = r_h\tensor s \colon \,  x\tensor y \mapsto  xh \tensor s(y) \ , 
\ee
where $s\in\End_\ok(M)$ and we used that $g$ has to intertwine the regular $H$-action and that by Lemma~\ref{lem:Uop-End}  such intertwiner is given by  right multiplication $r_h$ with an element $h\in\UH$. Recall the isomorphism $M\tensor_\ok M^*\xrightarrow{\sim} \End_\ok(M)$ that sends $v\tensor \gamma$ to the operator $\gamma(-) v$. Then it is straightforward to check that
$\Phi^{-1}: g \mapsto h\tensor\sum_{v,u\in\basis}s_{vu} v\tensor u^*$, where $(s_{vu})_{v,u\in\basis}$ is the matrix of the linear map~$s$. 
Finally, conjugating the image of $\Phi$ by $\phi$, i.e.\ sending $h\tensor v\tensor \gamma$ to $\f := \phi\circ f(h,v,\gamma)\circ \phi^{-1}$,  gives explicitly the isomorphism~\eqref{eq:EndHH-Mat}.

We show next that the map $\Phi$ is also an algebra map. The multiplication on $M\tensor_\ok M^*$ is 
\be
\big(M\tensor_\ok M^*\big) \tensor \big(M\tensor_\ok M^*\big)  \xrightarrow{\id\tensor \ev_M\tensor \id} M\tensor_\ok M^*
\ee
or explicitly (which is the standard matrix multiplicaition in $\Mat_{n,n}(\ok)$)
\be\label{eq:vgamma-mult}
(v_1\tensor \gamma_1)\cdot (v_2\tensor \gamma_2) =  \gamma_1 (v_2) v_1\tensor \gamma_2\ .
\ee
The source of $\Phi$ is then the product of two algebras $\UH^{\op}$ and $M\tensor_\ok M^*$. In the image space of $\Phi$, the multiplication is given by the composition~\eqref{eq:ff-mult}, as follows from the definition of $f(h,v,\gamma)$. Then using~\eqref{eq:vgamma-mult} it is easy to see that multiplication in $\UH^{\op}\tensor(M\tensor_\ok M^*)$ agrees with the one in  $\End_\UH(\UH\tensor_\ok M)$. By conjugating with $\phi$, the latter algebra is isomorphic to  $\End_\UH(\UH\tensor H)$. This finishes our proof.
\end{proof}

We finally note that the product $\UH\tensor W$ for any $W\in \Hmod$ can be studied similarly to $H\otimes H$. 
 Let us denote by 
$\Weps$ the vector space underlying $W$  equipped with the trivial action of $\UH$.
Then the generalisation of Theorem~\ref{thm:tensor-powers} is straightforward.

\begin{theorem}\label{thm:tensor-powers-W}
Let  $H$ be a finite-dimensional Hopf algebra $H$ and $W\in\Hmod$. We then have 
the isomorphisms of $\UH$-modules
\be\label{eq:phiW}
\phi_W\colon \UH \otimes \Weps\rightarrow \UH\otimes W\ ,\qquad
\phi^{-1}_W\colon \UH \otimes W\rightarrow \UH\otimes \Weps
\ee
which are given graphically as
\be\label{eq:phiW-def}
\ipic{rect-map}{.15} 
\put(-20,-2){\footnotesize $\phi_W$}
\put(-25,34){\scriptsize $H$} \put(-10,34){\scriptsize $W$} 
 \put(-25,-40){\scriptsize $H$}   \put(-13,-40){\scriptsize $\Weps$}
  \quad = \quad 
\ipic{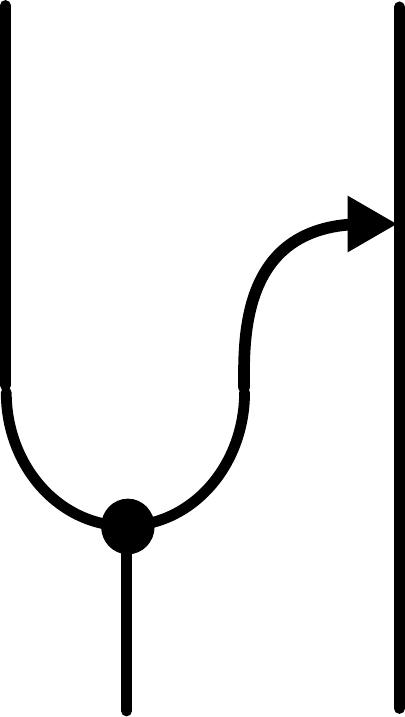}{.17} 
\put(-36,32){\scriptsize $H$} \put(-4,32){\scriptsize $W$} 
 \put(-28,-38){\scriptsize $H$}   \put(-8,-38){\scriptsize $\Weps$}
\qquad \qc \qquad 
\ipic{rect-map}{.15} 
\put(-20,-3){\footnotesize $\phi^{-1}_W$} 
\put(-25,34){\scriptsize $H$} \put(-13,34){\scriptsize $\Weps$} 
 \put(-25,-40){\scriptsize $H$}   \put(-10,-40){\scriptsize $W$}
 \quad = \quad 
\ipic{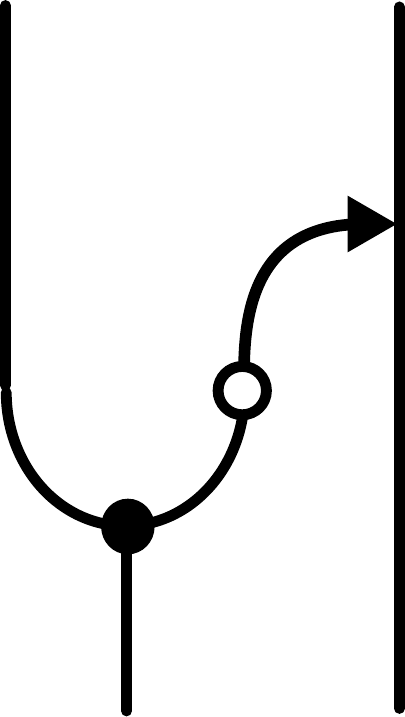}{.17}
\put(-36,32){\scriptsize $H$} \put(-7,32){\scriptsize $\Weps$} 
 \put(-28,-38){\scriptsize $H$}   \put(-5,-38){\scriptsize $W$} \;,
\ee
where the arrow denotes the $\UH$-action on $W$. 
In particular, we have an algebra isomorphism
\be\label{eq:EndHW-Mat}
\UH^{\op}\tensor\Mat_{m,m}(\ok) \xrightarrow{\,\sim \,} \End_\UH(\UH\tensor W)  \ ,\qquad m=\dim(W)\ ,
\ee 
which sends  $h\otimes A$ to the intertwining map
\be\label{eq:EndHW-Mat-map}
x\otimes w\mapsto (x'h)'\otimes (x'h)''m_A\bigl(S(x'')w\bigr)\ , \qquad x\in \UH, \; w\in W\ ,
\ee
and $m_A$ here is the  operator, $m_A\in\End_\ok(\Weps)$, corresponding to the matrix $A$.
\end{theorem}

\section{Proof of Theorem~\ref{thm:main} }\label{sec:proof}
We have now all the necessary ingredients to prove our main theorem. 
We start with a reformulation of 
 Reduction Lemma~\ref{lem:tHH-tH} adapted
to our current setting.
\begin{corollary}\label{cor:tHH-tH}
Given a unimodular pivotal Hopf algebra $(\UH,\pivot)$,
a symmetric linear function $\t\in H^*$ extends uniquely to a right modified trace on $\Hpmod$   if and only if
for all $f\in \End_{\UH}(\UH\otimes \UH)$ 
\be\label{eq:tHH-tH}
\t_{\UH\otimes \UH}\left(f \right)=
\t_\UH
 \bigl( \tr^r_\UH(f)\bigr) \gp
\ee
 Analogously, $\t$ extends  to a left modified trace on $\Hpmod$   if and only if 
\be\label{eq:tHH-tH-left}
 \t_{\UH\otimes \UH}\left(f \right)=
\t_{\UH} \bigl( \tr^l_{\UH}(f)\bigr) \gc   \qquad  \text{for all}\quad f\in \End_{\UH}(\UH\otimes \UH)\gp
\ee
\end{corollary}

Corollary \ref{cor:tHH-tH} allows us to  restrict
the analysis 
to the regular module and its tensor powers,
and therefore we can use the results of the previous section.

The  proof of Theorem~\ref{thm:main} is divided into three steps.

\vspace*{3mm}
\textbf{Step 1: $\rintg$ provides   right modified trace.}\;
We first show that the symmetrised right integral $\rintg$ provides the right modified trace.
By Proposition~\ref{lemma:mug-sym} 
and by the assumption that $H$  is unimodular,
$\rintg$ is a symmetric form on $H$, and it thus defines a family of trace functions that we will denote by $\t$, recall Proposition~\ref{cor:ext}, in particular  $\t_H (f)=\rint_\pivot(f(\one))$. 

In order to see that this family is a right modified trace,
by Corollary~\ref{cor:tHH-tH} it is enough to check the equalty
$\t_{\UH\otimes \UH}\left(f \right)=\t_\UH \left( \tr^r_\UH(f)\right)$ 
for any $f\in \End_H(H)$.
Let us rewrite LHS of the last equation as
\be\label{eq:tHH=tH}
\t_{\UH\otimes \UH}\left(f \right) = \sum_{y\in \basis}\t_{\UH\otimes \UH}\left(f \circ e_y \right) =  \sum_{y\in \basis}\t_{\UH}\left(\pi_y \circ f \circ \iota_y \right) \ ,
\ee
 where $\basis$ is a basis in $H$. Here,
we first  inserted the identity~\eqref{eq:ey}, then used Corollary~\ref{cor:HH-decomp} and  cyclicity
of $\t_H$.
Therefore, the equation we have to check is 
\be\label{eq:mug-ffhvg}
  \sum_{y\in \basis}\rintg\big(\pi_y \circ \f(h,v,\gamma) \circ \iota_y (\one)\big) = \rintg \Bigl( \tr^r_\UH\bigl(\f(h,v,\gamma)\bigr) (\one)\Bigr) \ , \quad h\in \UH,\; v\in\basis,\; \gamma\in\UH^*.
\ee
Recall that by Corollary~\ref{cor:EndHH-Mat} any element $f\in \End_\UH(\UH\tensor \UH)$ is of the form $\f(h,v,\gamma)$ defined in~\eqref{eq:tf-hvg}.
From Corollary~\ref{cor:HH-decomp}, we have that $\iota_y= \phi\circ g_y(h)$, $\pi_y =f_y\circ\psi$, $\psi=\phi^{-1}$ and 
\[
\text{LHS of}\; \eqref{eq:mug-ffhvg}  = 
 \sum_{y\in \basis}\rintg\big(f_y \circ f(h,v,\gamma) \circ g_y (\one)\big) = 
  \sum_{y\in \basis}\gamma(y) \rintg\big(f_y (h\otimes v) \big) = \gamma(v) \rintg(h)\ ,
\]
where we also used~\eqref{eq:g-fy-def}.
It remains to compute the RHS of~\eqref{eq:mug-ffhvg}.
Using the graphical expression for $\f(h,v,\gamma)$ in~\eqref{eq:ffhvg-map}, we get\footnote{We emphasize here by $\vect_\ok$ in the box that the diagrams, as maps from $\ok$ to $\ok$, are morphisms in $\vect_\ok$, so in particular evaluation and coevaluation maps are those from $\vect_\ok$ (the evaluation map in $\rep \UH$ was already resolved by using the pivotal element $\pivot$).} 
\be\label{eq:comp}
\text{RHS of}\; \eqref{eq:mug-ffhvg} \;=\;
\ipic{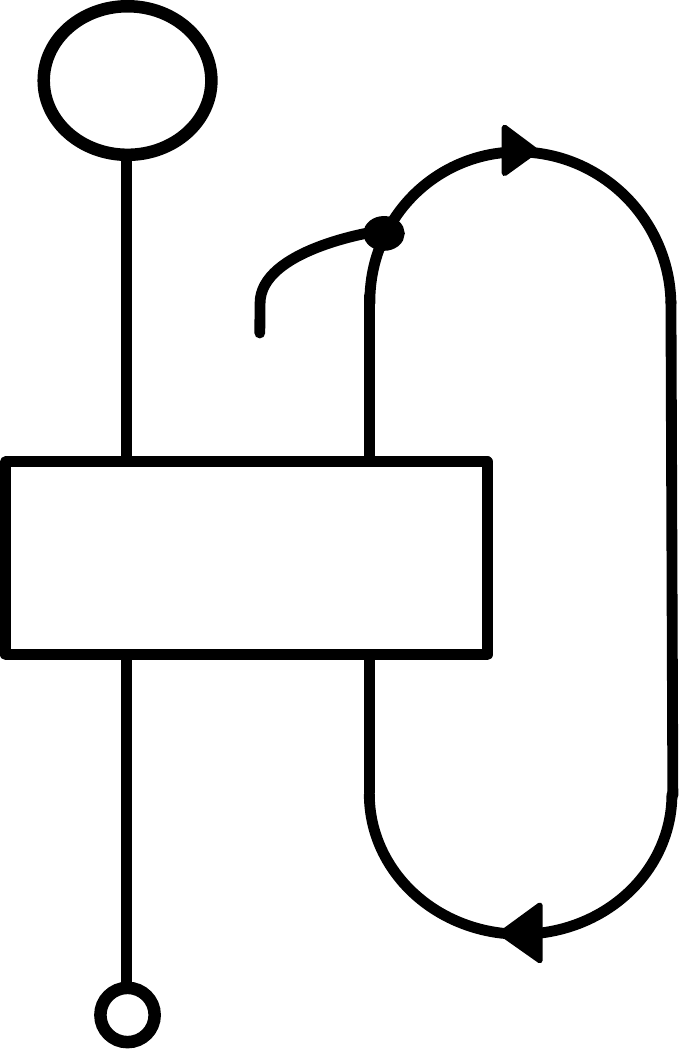}{.17} 
\put(-54,-5){\scriptsize $\f(h,v,\gamma)$}
\put(-50.5,36){\scriptsize $\rintg$} 
\put(-37,10){\scriptsize $\pivot$} 
\raisebox{4.5em}{\framebox{\scriptsize{$\vect_\ok$}}}
  \; = \; 
\ipic{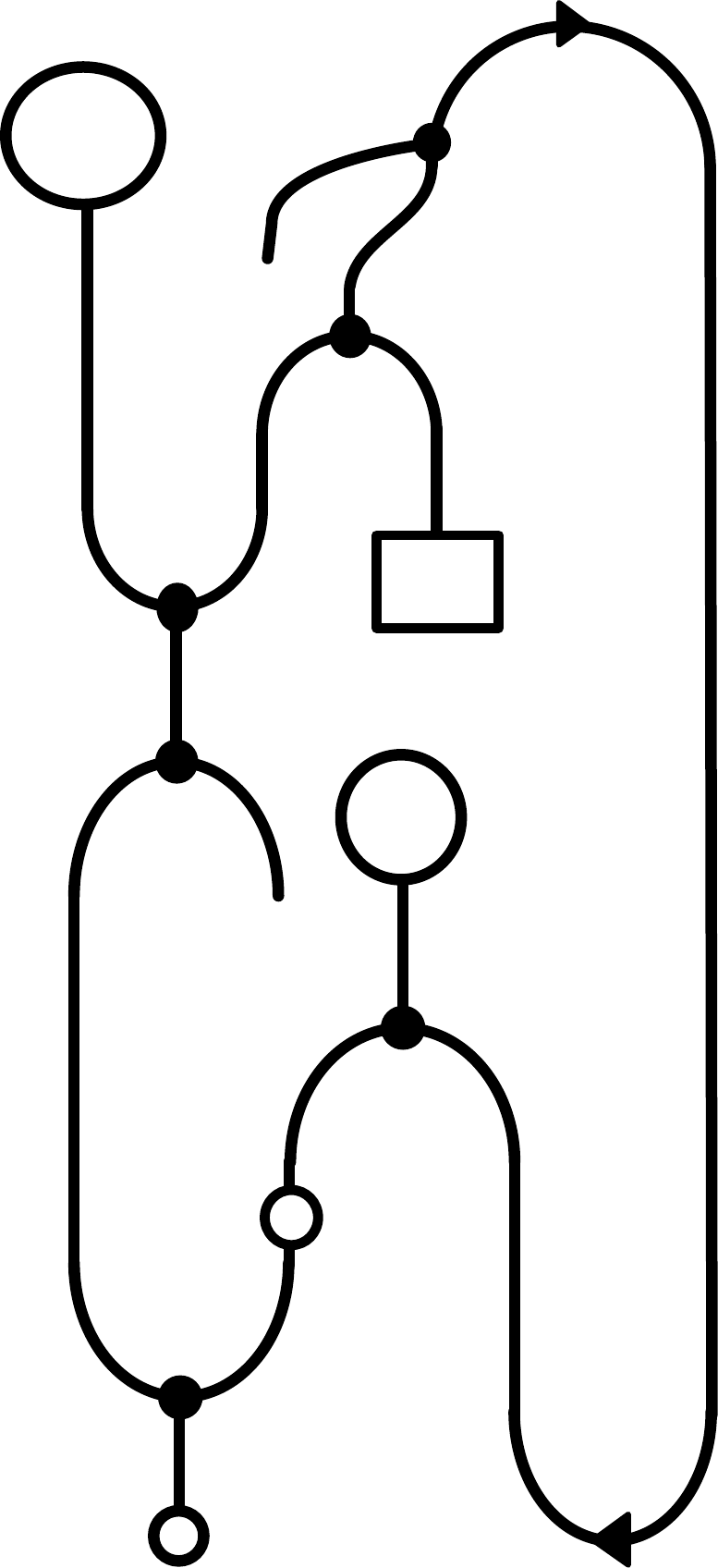}{.17} 
\put(-61.5,57){\scriptsize $\rintg$} 
\put(-42,-18){\scriptsize $h$} \put(-30.5,-5){\scriptsize $\gamma$}
 \put(-27,15.5){\scriptsize $v$}
 \put(-42,41){\scriptsize $\pivot$} 
% \raisebox{7.5em}{\framebox{\scriptsize{$\vect_\ok$}}}
  \; = \; 
  \ipic{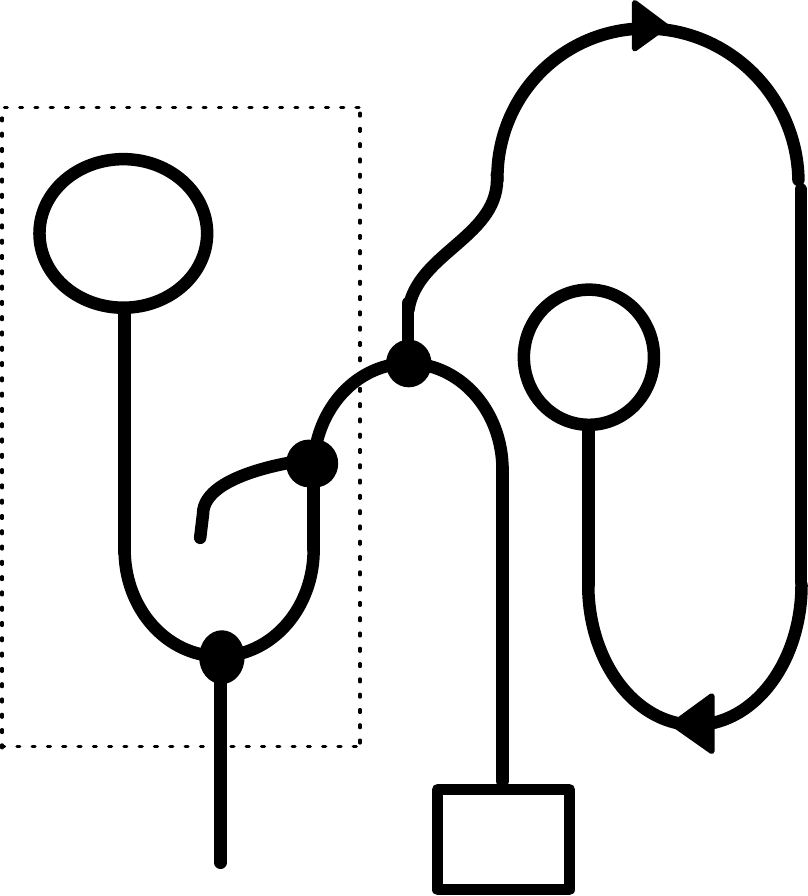}{.17} 
\put(-61.5,17){\scriptsize $\rintg$} 
\put(-52,-42){\scriptsize $h$} \put(-20.5,5.5){\scriptsize $\gamma$}
 \put(-28,-34){\scriptsize $v$}
 \put(-52,-12.5){\scriptsize $\pivot$} 
% \raisebox{7.5em}{\framebox{\scriptsize{$\vect_\ok$}}}
   \; = \; 
  \ipic{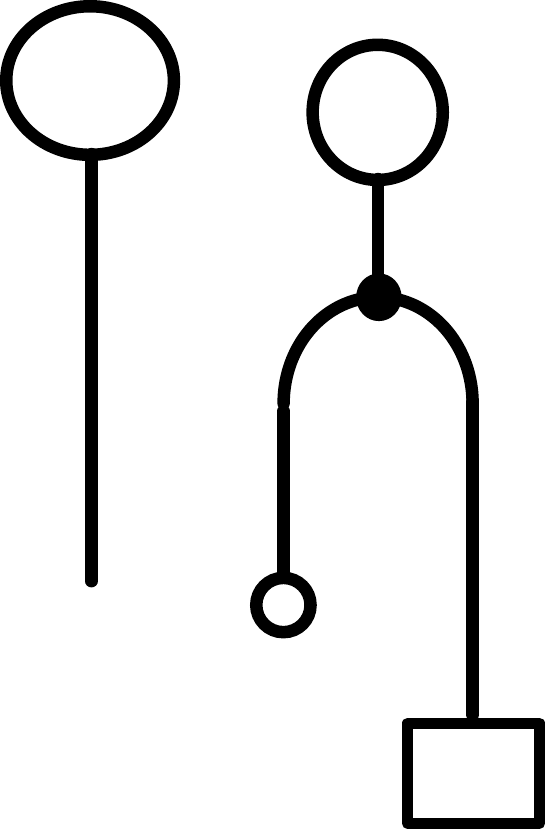}{.17} 
\put(-42.5,27){\scriptsize $\rintg$} 
\put(-41,-21){\scriptsize $h$} \put(-16.5,24){\scriptsize $\gamma$}
 \put(-8,-31){\scriptsize $v$}
% \; = \;   \gamma(v) \rintg(h)\ .
\gc
\ee
where for the first equality we use the definition of the partial trace in~\eqref{E:PartialLRtrace} and formulas~\eqref{E:DualitiesC}-\eqref{E:DualitiesC-right} for the  left  coevaluation $\coev_\UH$  and the right evaluation $\tilde\ev_\UH$  maps; in the second equality we substitute the explicit expression ~\eqref{eq:ffhvg-map} for  $\f(h,v,\gamma)$; the third equality is obvious; then in the fourth equality we  replace the part of the diagram inside the dashed rectangle by the (defining) relation~\eqref{eq:mug-def-rel} for the symmetrised integral $\rintg$ which is diagrammatically written as
\be\label{eq:sym-int}
  \ipic{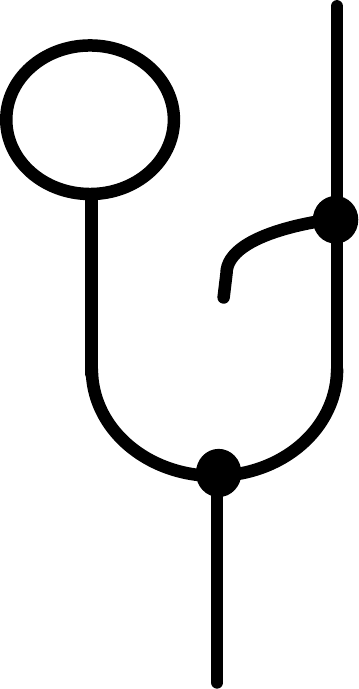}{.19}
   \put(-6,33){\scriptsize $\UH$} 
  \put(-29,20){\scriptsize $\rintg$} 
   \put(-15,-2.5){\scriptsize $\pivot$} 
     \put(-17,-40){\scriptsize $\UH$} 
   \quad = \quad 
  \ipic{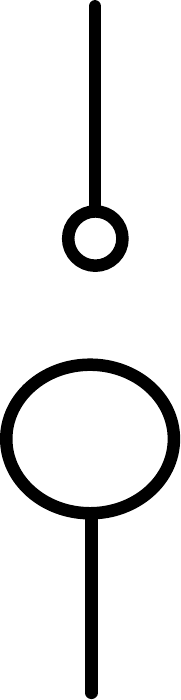}{.19}
   \put(-11,33){\scriptsize $\UH$} 
  \put(-13,-9){\scriptsize $\rintg$} \quad .
     \put(-28,-40){\scriptsize $\UH$} 
\ee
We finally see that RHS of~\eqref{eq:mug-ffhvg} also equals  $\gamma(v) \rintg(h)$, as we got for LHS of~\eqref{eq:mug-ffhvg}. Therefore the equality~\eqref{eq:mug-ffhvg} is true indeed for all $h\in \UH$, $v\in\basis$, and $\gamma\in\UH^*$ and thus for all endomorphisms of $\UH\otimes \UH$. This proves that the symmetric form $\rintg$ satisfies the right partial trace condition,
 and thus provides a right modified  trace for the ideal of projective $H$-modules.

\newcommand{\tf}{\tilde\t}

\vskip-5mm
\mbox{}\\
\textbf{Step 2: Right modified trace is  symmetrised integral.}\;
We now turn to the proof for the opposite direction.
Assume we have a right modified trace, and hence
 the symmetric form $\t_P$ on $\End_H P$ for any projective $P$, in particular the symmetric forms on $\End_H H $ and $\End_H (H\otimes H)$.
  They  satisfy
$\t_{\UH\otimes \UH}\left(f \right)=\t_\UH \left( \tr^r_\UH(f)\right)$, or equivalently 
\be\label{eq:t-ffhvg}
  \sum_{y\in \basis} \t_\UH \big(\pi_y \circ \f(h,v,\gamma) \circ \iota_y \big) = \t_\UH \Bigl( \tr^r_\UH\bigl(\f(h,v,\gamma)\bigr) \Bigr) \ ,
\ee
for  all $h\in \UH$, $v\in\basis$, $\gamma\in\UH^*$.
By the same arguments as in Step 1, we get $ \gamma(v) \t_\UH(r_h)$ for LHS of~\eqref{eq:t-ffhvg}, where $r_h$ is the right multiplication with $h$, which we can rewrite 
 \be\label{eq:LHS-t-ffhvg}
\text{LHS of}\; \eqref{eq:t-ffhvg} \; = \;   \gamma(v) \t(h)\qquad \text{where} \qquad \t(h):= \t_\UH(r_h) \ 
 \ee 
is the image of $\t_\UH$ under the isomorphism in Lemma~\ref{lem:Uop-End}, i.e.\ $\t$ is 
a symmetric form on $\UH$. 
We will further work  with $\t$ only.

Repeating now calculation in~\eqref{eq:comp} for the symmetric form $\t$,  RHS of~\eqref{eq:t-ffhvg} takes the form:
\be\label{eq:RHS-t-ffhvg}
\text{RHS of}\; \eqref{eq:t-ffhvg} \; = \; 
  \ipic{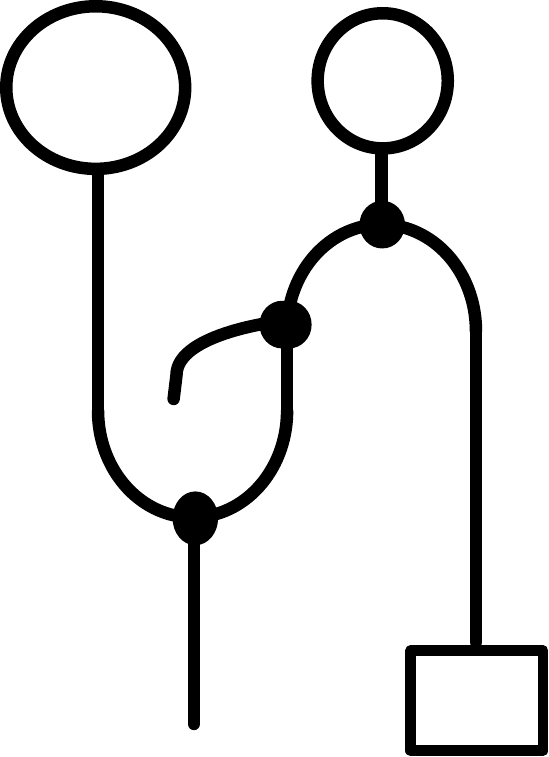}{.18} 
\put(-41,23){\scriptsize $\t$} 
\put(-32,-38){\scriptsize $h$} \put(-17.5,25){\scriptsize $\gamma$}
 \put(-8,-30){\scriptsize $v$}
 \put(-34.5,-7){\scriptsize $\pivot$} \quad ,
\ee
where  we used the relation $\t_\UH(f)= \t\bigl(f(\one)\bigr)$. 
Combining results~\eqref{eq:LHS-t-ffhvg} and~\eqref{eq:RHS-t-ffhvg} for the both sides and setting $v=\one$, we get for any $\gamma\in \UH^*$ and $h\in \UH$ the equality
\be
  \ipic{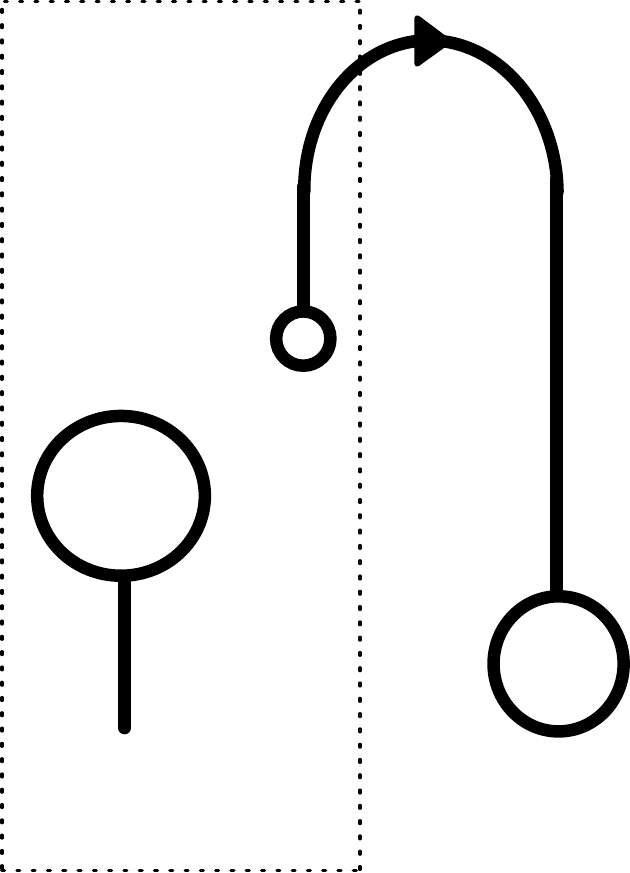}{.18} 
\put(-46,-8){\scriptsize $\t$} 
\put(-46,-34){\scriptsize $h$} 
 \put(-9,-21){\scriptsize $\gamma$}
 \; = \; 
 \ipic{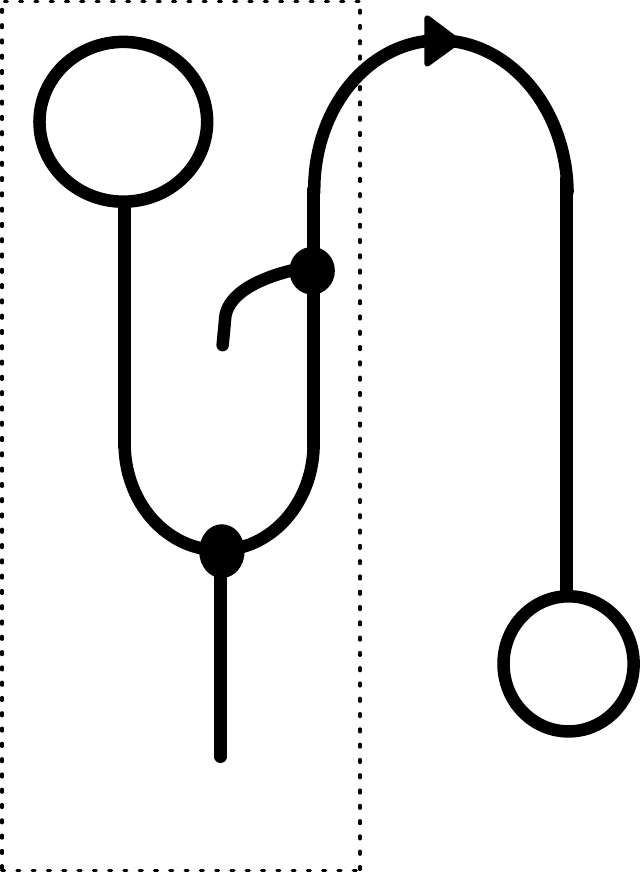}{.18} 
\put(-46,25){\scriptsize $\t$} 
\put(-39,-36){\scriptsize $h$} 
 \put(-9,-21){\scriptsize $\gamma$}
 \put(-39,2){\scriptsize $\pivot$} 
\quad .
\ee
As it is true for all $\gamma\in\UH^*$ we get the corresponding equality for the arguments of $\gamma$ -- the part of the diagram inside the dashed rectangles --
and this agrees with \eqref{eq:sym-int}. In other words, $\t$  satisfies the defining relation for the symmetrised right integral, i.e.
\be
(\t\tensor\pivot)\Delta(h)=\t(h)\one \ , \qquad h\in \UH\ 
.\ee

We thus conclude that $\t$, or equivalently the right modified trace $\t_\UH$, is a symmetrised right integral. As the latter is non-zero and unique up to a scalar,  
and the right modified trace on $\Hmod$ is determined by its value on $\UH$ by  Corollary \ref{cor:tHH-tH},
we  conclude that a non-zero right modified trace on $\Hpmod$ exists (under the assumptions of  Theorem~\ref{thm:main}) and is  \textsl{unique} up to  scalar.

\vskip-5mm
\mbox{}\\
\textbf{Step 3: Non-degeneracy, left and balanced cases.} \;
By Proposition~\ref{lemma:mug-sym} and
Theorem~\ref{Thm:nondegeneracy} the right modified trace
defined by $\rintg$ is non-degenerate. 
This finishes the proof of Theorem~\ref{thm:main} in the right case.

The proof for the left modified trace is completely analogous to the previous one. For example, to
 show that  the left symmetrised integral provides the left modified trace,
  it is enough to check
  the left partial trace property  $\t_{\UH\otimes \UH}\left(f \right)=\t_\UH \left( \tr^l_\UH(f)\right)$ for $\lintg$ which is
\be\label{eq:mug-ffhvg-left}
  \sum_{y\in \basis}\lintg\big(\pi_y \circ \f(h,v,\gamma) \circ \iota_y (\one)\big) = \lintg \Bigl( \tr^l_\UH\bigl(\f(h,v,\gamma)\bigr) (\one)\Bigr) \ ,
\ee
for all $ h\in \UH$, $v\in\basis$ and $\gamma\in\UH^*$.
Computations similar to those in \eqref{eq:comp} reduce this equality to
\eqref{left}, i.e.
$$
(\pivot^{-1}\tensor\lintg)\Delta(x)=\lintg(x)\one\gc 
$$
which is the defining relation for
the symmetrised left integral. 

Clearly, whenever $\UH$ is unibalanced,  left and right symmetrised integrals can be properly normalised such that they agree, e.g.\ by choosing $\lint=\rint\circ S$.
Therefore, the corresponding
left and right modified traces agree too.

This  finishes the proof of Theorem~\ref{thm:main}.
\section{Quantum Groups of types $ADE$}\label{qgroups}
In this section we  study finite-dimensional
quantum groups at roots of unity as defined in \cite[Sec.\,5]{Lusztig}\footnote{We use the opposite coproduct compared to
the one in \cite{Lusztig}.}
in the simply laced case.  We compute their right and left integrals and cointegrals, check that they are unibalanced and give a formula for the modified trace on the regular representation.
Here, the quantum parameter $q\in \kk$ is a root of $1$, whose square has order $\rt\geq 2$.

\subsection*{Definition}
For $n\geq 1$, let $A=\left(a_{ij}\right)$ be an indecomposable positive definite symmetric Cartan matrix 
of type $A_n$, $D_n$ or $E_n$, and $\mathfrak{g}$ denote the corresponding Lie algebra, 
with associated pairing denoted by $(\,\cdot\,|\,\cdot\,)$. 
In particular $a_{ii}=2$ for $1\leq i\leq n$, and $a_{ij}=a_{ji}\in \{0,-1\}$ for $1\leq i<j\leq n$. The $\kk$-algebra $\UU$ 
 is generated by $K_i^{\pm 1}$, $E_i$ and $F_i$, $1\leq i\leq n$,
with relations, for all $i$, $j$:
\begin{align}\label{eq:Uq-rel}
 K_iE_jK_i^{-1}&=q^{a_{ij}}E_j, & K_iF_jK_i^{-1}&=q^{-a_{ij}}F_j,\notag \\
  [E_i,F_j]&=\delta_{ij}\ffrac{K_i-K_i^{-1}}{q-q^{-1}},\quad  & K_iK_j&=K_jK_i,\\ 
    E_i^\rt&=F_i^\rt=0,  \quad &K_i^{2\rt}&=\one, \notag
 \end{align}
 with the Serre relations
\begin{align}
  E_iE_j=E_jE_i,\ \  F_iF_j&=F_jF_i, &\text{ if }a_{ij}=0,&\notag\\
  E_i^2E_j-(q+q^{-1})E_iE_jE_i+E_jE_i^2&=0, &\text{ if }a_{ij}=-1,&\notag\\
  F_i^2F_j-(q+q^{-1})F_iF_jF_i+F_jF_i^2&=0,  &\text{ if }a_{ij}=-1.&\notag
\end{align}
The algebra $\UU$ is a Hopf algebra where the coproduct, counit and
antipode are defined as 
\begin{align}
  \Delta(E_i)&= 1\otimes E_i + E_i\otimes K_i, 
  &\epsilon(E_i)&= 0, 
 &S(E_i)&=-E_iK_i^{-1}, 
 \\
 \Delta(F_i)&=K_i^{-1} \otimes F_i + F_i\otimes 1,  
  &\epsilon(F_i)&=0,& S(F_i)&=-K_iF_i,\notag
   \\
   \Delta(K_i)&=K_i\otimes K_i,
 &\epsilon(K_i)&=1,
 & S(K_i)&=K_i^{-1}.\notag
\end{align}
\newcommand{\weyl}{\mathcal{W}}
Let $L$ be the root lattice, with $\mathbb Z$-basis denoted by $\alpha_i$, $1\leq i\leq n$.
We denote by $\Delta_+$ the set of positive roots, by $N=|\Delta_+|$ its cardinality, and by
$\rho$ half the sum of the positive roots.
The formulas for $N$ and  the sum of positive roots $2\rho$
in different types are given below (compare with~\cite[Ch.\,VI]{Bourbaki}): 
$$
\begin{array}{|c|| c |c|}
\hline
&N&2\rho \\ \hline
A_n , n\geq 1& \frac{n(n+1)}{2}& \sum_{i=1}^n i(n-i+1)\alpha_i  \\ \hline
D_n, n\geq 4& n(n-1)& \sum_{i=1}^n (2in -i(i+1))\alpha_i   \\ \hline
E_6&36 & \text{see \cite[Plate V]{Bourbaki}}\\ \hline
E_7& 63& \text{see \cite[Plate VI]{Bourbaki}}\\ \hline
E_8&120 & \text{see \cite[Plate VII]{Bourbaki}}\\ \hline
\end{array}
$$

\subsection*{PBW basis}
Let $\weyl$ be the Weyl group generated by the simple reflexions $s_i$, $1\leq i\leq n$.
It  is a finite Coxeter group. Its  basic structural properties we  use here can be found in~\cite{Bourbaki}. 
 For $w\in \weyl$ we denote by $l(w)$ the  length of a  reduced expression in the generators $s_i$. 
Let us choose a reduced expression of the longest element
 of~$\weyl$,
\be \label{longest} w_0=s_{i_1}s_{i_2}\dots s_{i_N}\ ,\ee
 in the simple reflexions $s_i$, $1\leq i\leq n$.
To get an ordered list of positive roots \cite[Sec.\,VI.1.6, Cor.\,2]{Bourbaki} we set
\be \label{root_vectors}
\beta_1=\alpha_{i_1}\gc\ \ \beta_2=s_{i_1}(\alpha_{i_2})\gc \ \ \beta_3=s_{i_1}s_{i_2}(\alpha_{i_3})\gc\ \ \dots\gc\ \ \beta_N=s_{i_1}\dots s_{i_{N-1}}(\alpha_{i_N}) \ .\ee

For $1\leq i\leq n$, let $T_i$ be an algebra automorphism  of $\UU$ which acts on generators
$K_j$, $E_j$, and $F_j$ by
\begin{align}\label{eq:T-act}
  T_i(K_j)&= K_i^{-a_{ij}}K_j\gc 
  &T_i(E_i)  &= -F_iK_i\gc 
 &T_i(F_i)&=-K_i^{-1}E_i\gc 
 \\
 T_i(E_j)&=E_j\gc  
  &T_i(F_j)&=F_j\gc& 
  \text{ if }a_{ij}&=0\gc\notag\\
  T_i(E_j)&=-E_iE_j+q^{-1}E_jE_i\gc  
  &T_i(F_j)&=qF_iF_j-F_jF_i\gc& 
  \text{ if }a_{ij}&=-1\gp\notag
\end{align}
The root vectors  are then defined by, see  \cite{Lusztig} or \cite[Ch.\,8]{Jantzen},
\begin{align}
E_{\beta_1}&=E_{i_1}\gc&\ E_{\beta_2}&=T_{i_1}(E_{i_2})\gc& \ E_{\beta_3}&=T_{i_1}T_{i_2}(E_{i_3})\gc&\dots&\ \gc\ &E_{\beta_N}&=T_{i_1}\dots T_{i_{N-1}}(E_{{i_N}}) \ \gc\\
F_{\beta_1}&=F_{i_1}\gc&\ F_{\beta_2}&=T_{i_1}(F_{i_2})\gc& \ F_{\beta_3}&=T_{i_1}T_{i_2}(F_{i_3})\gc&\dots&\ \gc\ &F_{\beta_N}&=T_{i_1}\dots T_{i_{N-1}}(F_{{i_N}}) \ .\notag
\end{align}

\noindent {\bf Example.} For $A_2={\mathfrak{sl}}(3,\mathbb C)$  there are two reduced decompositions
of the longest element $w_0=s_1 s_2 s_1$ and $
w_0=s_2 s_1 s_2$. The corresponding sequences of positive root vectors are
$$
 E_1\gc \quad T_1(E_2)=-E_1E_2+q^{-1}E_2E_1\gc \quad T_1T_2(E_1)=E_2
 $$
and
$$
 E_2\gc\quad  T_2(E_1)=-E_2E_1+q^{-1}E_1E_2\gc \quad T_2T_1(E_2)=E_1\gp
$$

\noindent
The algebra automorphisms $T_i$ satisfy  the braid relations
\begin{align} 
T_i\circ T_j&=T_j\circ T_i & \text{if $a_{ij}=0$\gc\quad}\label{eq:rel-T}\\
T_i\circ T_j\circ T_i &=T_j\circ T_i \circ T_j & \text{if $a_{ij}=-1$\gp\,}\notag
\end{align}
For a given $w\in\weyl$ and a reduced decomposition $w=s_{j_1}\dots s_{j_m}$
there is an algebra automorphism $T_w=T_{j_1}\circ\dots \circ T_{j_m}$.
The relations~\eqref{eq:rel-T} assert that $T_w$ depends only on the element $w$ and not on its decomposition.

The algebra $\UU$ has $L$-grading denoted by $\wt$ and defined 
on generators by $\wt(E_i)=\alpha_i$, $\wt(F_i)=-\alpha_i$
and $\wt(K_i)=0$. We also define $\wt(E_i E_j)=\alpha_i+ \alpha_j$, etc. This makes the algebra graded, because relations are homogeneous.
We will use the following lemma.

\begin{lemma}\label{lem:L-degree}
For any root $\beta$, the root vectors
$E_\beta$ and $F_\beta$ have  $L$-grading $\wt(E_\beta)=\beta$ and $\wt(F_\beta)=-\beta$, respectively.
\end{lemma}
 When $q$ is not a root of unity this known lemma  can be established using the adjoint action of the Cartan elements \cite[Ch.\,6, Prop.\,23]{Klymik_book}. For completeness we give in Appendix \ref{appB}  a proof  of the stronger statement
 in the next lemma for all non-zero values of $q$.
 
\begin{lemma}\label{lem:A.2}
Assume that for a pair $(w,i)$, with $w\in \weyl$ and  $1\leq i\leq n$, we have $l(ws_i)=l(w)+1$. Then  $w(\alpha_i)\in \Delta_+$ and   $T_w(E_i)$ 
 has $L$-grading $\wt\bigl(T_w(E_i)\bigr)=w(\alpha_i)$. And similarly, $\wt\bigl(T_w(F_i)\bigr)=-w(\alpha_i)$.
\end{lemma}

Recall that the root vectors are obtained from a reduced decomposition of the longest word~\eqref{longest}.
Lemma \ref{lem:L-degree} is obtained by applying Lemma \ref{lem:A.2} to 
$(s_{i_1}\dots s_{i_{k-1}},{i_{k}})$, $1\leq k\leq N$, and using \eqref{root_vectors}.

 Introducing 
  $I=\{0,1,\dots,2\rt-1\}$,  $J=\{0,1,\dots,\rt-1\}$, 
 we  can now construct a PBW basis 
 of $\UU$ \cite[Section 5.8]{Lusztig}
\be\label{E:PBW}
B_{m^-,m,m^+}=\prod_{\beta \in \Delta_+}F_\beta^{m^-_\beta}\prod_{i=1}^n K_i^{m_i} \prod_{\beta \in \Delta_+}E_\beta^{m^+_\beta}
 \ee
indexed by $m \in I^n$ and $m^\pm\in  J^{\Delta_+}$,
 or in other words $m^\pm=(m^\pm_\beta)$ is a map from $\Delta_+$ to $J$.
 We will use the notation $m^\pm_k$ for $m^\pm_{\beta_k}$ where $\beta_k$ is the $k$-th root defined in \eqref{root_vectors}.
We denote by $B^*_{m^-,m,m^+}$ the dual basis in $\bigl(\UU\bigr)^*$ defined by 
$$
\langle B^*_{m^-,m,m^+},B_{\tilde m^-,\tilde m,\tilde m^+}\rangle =\delta_{m^-,\tilde m^-}\delta_{m,\tilde m}\delta_{m^+,\tilde m^+}\gp
$$ 

\subsection{Main result} We are now in position to present 
the main result of this section.

\begin{theorem}\label{thm:qg}
a) The Hopf algebra $\UU$ is unimodular with the cointegral 
\be\label{eq:coint}
\coint=\prod_{i=1}^n\left( \sum_{m=1}^{2\rt} K_i^m \right)\prod_{\beta \in \Delta_+}F_\beta^{\rt-1}\prod_{\beta \in \Delta_+}E_\beta^{\rt-1} \gp
\ee
b)
The Hopf algebra $\UU$ is pivotal with pivots 
\be\label{eq:UU-pivots}
\pivot_\varepsilon=K_{2\rho}\prod_{i=1}^n K_i^{p\varepsilon_i}\ ,\qquad \varepsilon\in\{0,1\}^{\times n}
\ee 
and  it is unibalanced for any choice of $\varepsilon$, with
the corresponding symmetrised integral 
\be\label{eq:tr'}
\displaystyle \rintg=\lintg=B^*_{(\rt-1)^{\Delta_+}, p\varepsilon,(\rt-1)^{\Delta_+}}\ .\ee
 Here $(\rt-1)^{\Delta_+}$ is the constant map on $\Delta_+$ with value  $\rt-1$.
\end{theorem}

Before giving a proof, we first note that as a consequence of Theorem~\ref{thm:main} the formula in~\eqref{eq:tr'} computes the modified trace $\t$ for endomorphisms of the regular representation.
We also note that for type $A_n$ and with slightly different version of the quantum group, a cointegral and an integral were computed in \cite{GW}.
Our proof for the cointegral goes along the lines in~\cite[Thm.\,2.1.5]{GW}, however in our case it requires the following lemma on commutation relations
 whose proof is in Appendix \ref{appC}.
 
\begin{lemma}\label{lemma:commut}
For $1\leq j<k\leq N$, we have in $\UU$ the commutation relation for the root vectors, with $\beta_j$ defined in \eqref{root_vectors},
\be \label{lem:commut}
E_{\beta_{j+1}}^{p-1}E_{\beta_{j+2}}^{p-1}\dots E_{\beta_{k}}^{p-1}E_{\beta_j}=
q^{(p-1)(\beta_j|\beta_{j+1}+\dots+\beta_k)} E_{\beta_j}E_{\beta_{j+1}}^{p-1}E_{\beta_{j+2}}^{p-1}\dots E_{\beta_{k}}^{p-1}\gp
\ee
\end{lemma}

\newcommand{\Uq}{\overline{U}_q}
\begin{proof}[Proof of Thm.~\ref{thm:qg}]
We first prove the part $a)$. We begin with computing  
cointegrals for the Borel subalgebras.
For brevity, we will use  the notation $\Uq:=\UU$

Let $\Uq^-$ be the  negative Borel  subalgebra with the basis $B_{m^-,m,0}$\,, $m^-\in  J^{\Delta_+}$ and $m \in I^n$, it is also a Hopf subalgebra. And similarly for the positive $\Uq^+$
 with the basis $B_{0,m,m^+}$, $m \in I^n$ and $m^+\in  J^{\Delta_+}$.
 
We claim that 
\be\label{eq:coint-}
\coint^-= \prod_{i=1}^n\left( \sum_{m=1}^{2\rt} K_i^m \right)\prod_{\beta \in \Delta_+}F_\beta^{\rt-1}
\ee
is a left cointegral for $\Uq^-$.
Indeed,
\be K_i\coint^-=\coint^-=\epsilon(K_i)\coint^-\quad \text{ for }1\leq i\leq n\ .
\ee
From  
Lemma~\ref{lem:L-degree}
 we see that $\prod_{\beta \in \Delta_+}F_\beta^{\rt-1}$ has the minimal possible $L$-degree  $-(p-1)2\rho$.
Therefore we have 
\be\label{eq:FF0}
F_i \cdot \prod_{\beta \in \Delta_+}F_\beta^{\rt-1} = 0\gp
\ee 
We can then check
\be
 F_i\coint^-=0=\epsilon(F_i)\coint^-\gc\quad \text{ for }1\leq i\leq n \gc
\ee 
because moving $F_j$ through the Cartan part of $\coint^-$ just replaces $K_i$ by $q^{a_{ij}}K_i$ and the most non-trivial part is the equality~\eqref{eq:FF0}.
Hence for all $x\in \Uq^-$, we have 
\be x\coint^-=\epsilon(x)\coint^- ,\ee
 and so $\coint^-$ is indeed a left cointegral in $\Uq^-$.
We similarly get that $$\coint^+= \prod_{\beta \in \Delta_+}E_\beta^{\rt-1}\prod_{i=1}^n\left( \sum_{m=1}^{2\rt} K_i^m \right)$$
is a right cointegral in $\Uq^+$. 

We know that $\Uq$ has a non-zero left cointegral $\coint$, unique up to normalisation.
Moreover there exists a group-like element $\alpha\in\Uq^{\,*}$, called the modulus, such that
\be\label{eq:Lambda-modulus}
\coint x =\alpha(x)\coint\quad \text{ for all }x\in \Uq\ ,
\ee
see \cite[Eq. (10.8)]{Ra-book}.
Using the basis~\eqref{E:PBW} in $\Uq$, we see that $\Uq$ is a free left module over $\Uq^-$ with basis $B_{0,0,m^+}$ with $m^+\in J^{\Delta_+}$.
Let us write $\coint$ in this basis 
\be \coint=\sum_{m^+} \coint_{m^+}\,B_{0,0,m^+}\quad \text{with }\quad
\coint_{m^+}\in \Uq^- . \label{E:Lambda}
\ee 

Using~\eqref{eq:Lambda-modulus} we get
 \be \coint E_i=\alpha(E_i)\coint=0 \quad \text{ for }1\leq i\leq n \ .\ee
Here, the vanishing is because   the modulus $\alpha$ is group-like and hence $\alpha(E^p_i)=\alpha(E_i)^p$, but
 $E_i^p=0$ and so $\alpha(E_i)=0$.
 We therefore have that for all root vectors $E_{\beta_j}$ 
 \be \label{vanishing}
 \sum_{m^+} \coint_{m^+}\,B_{0,0,m^+}E_{\beta_j}=0\ .
 \ee
 We show by induction on $\nu=N-j$ that here $\coint_{m^+}=0$ if $m^+_l<p-1$ for some $l\geq j$.
 
 Let us denote by $\tau_j(m^+)$ the result of increasing the $j$-th component of $m^+$ by $1$.
 We have that $B_{0,0,\tau_j(m^+)}$ is zero if $m^+_j=p-1$ and is a PBW basis element otherwise.
 
 We begin with $\nu=0$, 
 the equation \eqref{vanishing} for $j=N$ then gives 
 \be 
 \sum_{m^+} \coint_{m^+}\,B_{0,0,\tau_N(m^+)}=0 \gc
 \ee
 where only terms with  $m^+_N<p-1$ contribute. As the corresponding  elements $B_{0,0,\tau_N(m^+)}$ are linearly  independent over $\Uq^-$, we have 
 $\coint_{m^+}=0$ if $m^+_N<p-1$. This is the first step of induction.

By the induction hypothesis at $\nu=N-j$ we assume $\coint_{m^+}=0$ in~\eqref{vanishing} if $m^+_l<p-1$ for some $l\geq j$.
 Then, equation \eqref{vanishing} for $\nu=N-j+1$  gives
 \be \sum_{m^+\atop
 m^+_j=\dots=m^+_N=p-1} \coint_{m^+}\,B_{0,0,m^+}E_{\beta_{j-1}}=0\ .\notag\ee 
 Using the commutation relation \eqref{lem:commut}, we obtain
 \be \sum_{m^+\atop
 m^+_j=\dots=m^+_N=p-1} \coint_{m^+}\,B_{0,0,\tau_{j-1}(m^+)}=0 \ .\notag\ee
 We deduce as before $\coint_{m^+}=0$ if $m^+_{j-1}<p-1$ and this finishes the proof by induction.
 
 As the equality~\eqref{vanishing} is true for all root vectors,
 we have thus obtained that only the  term with $m^+=(p-1)^{\Delta_+}$ contributes to~\eqref{E:Lambda}.
We obtain that the left cointegral has the form 
\be 
\coint=\coint_{(p-1)^{\Delta_+}}\, B_{0,0,(p-1)^{\Delta_+}} \gc \quad \text{ with }\quad\coint_{(p-1)^{\Delta_+}}\in \Uq^-\ .
\ee

Recall that $\coint$ is a left cointegral by assumption, therefore we  have the  equality 
 \be
  x\coint = \epsilon(x)\coint \quad \text{ for all }x\in \Uq^-\ .
  \ee
 Using that $\Uq$ is a free module over $\Uq^-$, we get 
 \be 
 x\coint_{(p-1)^{\Delta_+}} = \epsilon(x)\coint_{(p-1)^{\Delta_+}}\quad \text{ for all }x\in \Uq^-\ .
 \ee
 We have that $\coint_{(p-1)^{\Delta_+}}$ is a left cointegral
 in $\Uq^-$, i.e.\ it is proportional to $\coint^-$ from~\eqref{eq:coint-}. 
  This shows that $\coint$ is proportional to  $\coint^-B_{0,0,(p-1)^{\Delta_+}}$
   which is the formula in $\eqref{eq:coint}$.
  
 We now show that $\coint$ is two-sided. Indeed, for the right multiplication on $\coint$ we have 
 $$
 \coint K_i=\coint\gc \qquad \coint E_i=\alpha(E_i)\coint= 0\gc \qquad \coint F_i=\alpha(F_i)\coint=0\gc \quad \text{for}\ 1\leq i\leq n\gc
 $$
  where the first equality is due to the relation~\eqref{eq:Uq-rel} and we used explicit expression~\eqref{eq:coint},
for   the second
    and    third equalities we first used~\eqref{eq:Lambda-modulus} and then the fact that the modulus $\alpha$  vanishes on $E_i$ and $F_i$ because $\alpha$  is group-like and $E_i^p= F_i^p=0$.
 We have thus shown that $\alpha=\epsilon$ and therefore $\coint$ is a two-sided cointegral which 
 implies unimodularity of $\Uq$.

Now we prove part $b)$.
 To verify the defining relation  for the right integral $\rint$ we will need a formula for coproduct of PBW basis elements. Let 
$$
K_\beta=\prod_{i=1}^n K_i^{n_i} \qquad \text{for}\quad \beta=\sum n_i\alpha_i\gp
$$
For the  root vectors $E_\beta$, for $\beta\in \Delta_+$, the coproduct can be written as follows \cite[Sec.\,4.12]{Jantzen} 
\begin{align}
\Delta (E_\beta)&=E_\beta\otimes K_\beta + 1\otimes E_\beta+
\sum_\nu x_\nu\otimes y_\nu
\end{align}
where $x_\nu$ and $y_\nu$ are  PBW elements $B_{0,m,m^+}\in \Uq^+$  with 
non-zero $m^+$ and such that $\wt(x_\nu)+\wt(y_\nu)=\beta$.
We similarly have
\begin{align}
\Delta (F_\beta)&=F_\beta\otimes 1+ K_\beta^{-1}\otimes F_\beta+ \sum_\nu x_\nu\otimes y_\nu
\end{align}
where $x_\nu$ and $y_\nu$ are now PBW elements $B_{m^-,m,0}\in \Uq^-$ with non-zero $m^-$ and such that $\wt(x_\nu)+\wt(y_\nu)=-\beta$.
More generally, for the coproduct of a PBW basis element~\eqref{E:PBW}, we have
\begin{align}\label{eq:cop-PBW}
\Delta (B_{m^-,m,m^+})=&B_{m^-,m,m^+}  \otimes K_{\wt(B_{0,0,m^+})}\prod_{i=1}^n K_i^{m_i}\\
&+ K_{\wt(B_{m^-,0,0})}\prod_{i=1}^n K_i^{m_i}\otimes B_{m^-,m,m^+}
+ \sum_\nu x_\nu \otimes y_\nu\notag
\end{align}
where $x_\nu$ and $y_\nu$ are in the span of PBW elements $B_{\tilde m^-,\tilde m,\tilde m^+}$  where  all components of $\tilde m^-$ (resp. $\tilde m^+$) are lower or equal to those of $m^-$
(resp. $m^+$), and at least
one of them is strictly lower.

 Let \mbox{$M:=\bigl(M_i\bigr)_{1\leq i\leq n}$} be the coordinates
of the sum of positive roots in basis of simple roots:
$$
2\rho=\sum_{\beta\in \Delta_+} \beta=\sum_{i=1}^n M_i\alpha_i\gp
$$
 The corresponding Cartan element is
\mbox{$K_{2\rho}=\prod_{i=1}^n K_i^{M_i}$}.

Let us now verify that
\be
\rint=B^*_{(\rt-1)^{\Delta_+},(p+1) M,(\rt-1)^{\Delta_+}}
\ee
satisfies the defining relation for the right integral
 \be\label{eq:rint-def-2}
 (\rint\otimes \id)\Delta(x)=\rint(x) \one\gp
 \ee
For PBW elements $B_{m^-,m,m^+}$ where at least one $m^\pm_\beta$  is lower than $\rt-1$, using~\eqref{eq:cop-PBW} we see that both sides of this equation give $0$.
For $B_{(\rt-1)^{\Delta_+},m, (\rt-1)^{\Delta_+}}$, we get
\begin{align}
\Delta\bigl(B_{(\rt-1)^{\Delta_+},m, (\rt-1)^{\Delta_+}}\bigr)=B_{(\rt-1)^{\Delta_+},m, (\rt-1)^{\Delta_+}}\otimes K_{2\rho}^{\rt-1}\prod_{i=1}^n K_i^{m_i}+ \text{ other terms} .
\end{align}
Here, $\rint\otimes \id$ vanishes on the ``other terms".
If $m\neq (p+1)M$ we again get $0$ on both sides of~\eqref{eq:rint-def-2}.
In the remaining case with  $m=(p+1)M$, we have
 $K_{2\rho}^{\rt-1}\prod_{i=1}^n K_i^{(p+1)M_i}=\one$ 
which shows that the equality~\eqref{eq:rint-def-2} holds indeed.
 
We now compute the comodulus $\comod$ using the defining equation \eqref{eq:rint-a}. Using
$$
\rint\bigl(B_{(\rt-1)^{\Delta_+},(p+1) M,(\rt-1)^{\Delta_+}}\bigr)=1\ ,
$$
 we obtain the formula
\be
\comod=(\id\tensor\rint)\Delta\bigl(B_{(\rt-1)^{\Delta_+},(p+1) M,(\rt-1)^{\Delta_+}}\bigr)\ .
\ee
Taking now into account the second term on RHS of~\eqref{eq:cop-PBW}, we  have 
$$
\Delta\bigl(B_{(\rt-1)^{\Delta_+},(p+1)M, (\rt-1)^{\Delta_+}}\bigr)= K_{2\rho}^{1-\rt}\prod_{i=1}^n K_i^{(p+1)M_i}\otimes B_{(\rt-1)^{\Delta_+},(p+1)M, (1-\rt)^{\Delta_+}}+ \text{ other terms} \ .
$$
From this, we deduce the value of the comodulus 
\be
 \comod=K_{2\rho}^{1-\rt}\prod_{i=1}^n K_i^{(p+1)M_i}=K_{2\rho}^2\ .
 \ee

We study next group-like  square roots of $\comod$, these are $\pivot_\varepsilon=K_{2\rho}\prod_{i=1}^n K_i^{p\varepsilon_i}$, with $\varepsilon\in\{0,1\}^{\times n}$. 
We check on generators that each $\pivot_\varepsilon$ implements $S^2$, and so a pivot.
Indeed, for $1\leq i\leq n$,
\begin{align}
\pivot_\varepsilon K_i\,\pivot_\varepsilon^{-1}&=K_i=S^2(K_i)\ ,\notag \\
\pivot_\varepsilon E_i\,\pivot_\varepsilon^{-1}&=K_iE_iK_i^{-1}=S^2(E_i)\ ,\notag\\
\pivot_\varepsilon F_i\,\pivot_\varepsilon^{-1}&=K_iF_iK^{-1}=S^2(F_i)\ .\notag
\end{align}
Therefore, the Hopf algebra $\Uq$ is pivotal with a pivot $\pivot_\varepsilon=K_{2\rho}\prod_{i=1}^n K_i^{p\varepsilon_i}$ for any $\varepsilon\in\{0,1\}^{\times n}$.
We then get formula~\eqref{eq:tr'} for the right symmetrised integral.
By Lemma \ref{agsquare},  $(\Uq,\pivot_{\varepsilon})$ is unibalanced for any choice of $\varepsilon$ because $\comod=\pivot_{\varepsilon}^2$,
or the right symmetrised  integral is also left. Moreover, we have $\rintg=\lintg$ and so  \eqref{eq:tr'} holds for the left symmetrised integral too.
\end{proof}

 \section{Modified trace for the restricted quantum $\mathfrak{sl}_2$}\label{sec:A1case}
 Here,  we apply results of the previous section to  type $A_1$ and  demonstrate how the modified trace for indecomposable projectives can be explicitly computed from the symmetrised integral.
 For this we will use an explicit basis of  Hom-spaces between indecomposable projectives   constructed in \cite{FGST2}. 
The quantum group in type $A_1$, for the choice $q=e^{i\pi/p}$ and $p\geq 2$, is known as \textit{restricted} quantum $\mathfrak{sl}_2$, and
  will be denoted by $\Urest$. 
 In \cite{BBG}     the
modified trace  on all endomorphisms of indecomposable projectives in $\Upmod$ was  computed and then extended to the regular representation $\Urest$. 
  Here we do the converse: we reprove \cite{BBG} formulas starting with the
  symmetrised integral. In this section, we set 
   $[k]=\ffrac{q^k-q^{-k}}{q-q^{-1}}$  and 
$[m]!=\prod_{k=1}^m [k]$, and $\qbin{m}{k} = \frac{[m]!}{[k]![m-k]!}$, for $k$ and $m$ positive integers.

\newcommand{\coeffN}{\eta}

 \subsection*{Symmetrised integral}
 We will work with the choice of pivot $\pivot:=\pivot_{\varepsilon=1}=K^{p+1}$, recall~\eqref{eq:UU-pivots}.
 In the PBW basis of $\Urest$,  the right integral is given by 
\begin{gather*}
  \rint(F^i E^m K^n)
  =
 \coeffN\,
 \delta_{i, p-1}\delta_{m, p-1}\delta_{n,  p+1}
\end{gather*}
where 
$\coeffN$  is a non-zero normalising coefficient.
 Then our (right) symmetrised integral is
\be\label{eq:mug}
\rintg(F^i E^m K^n ) =
\coeffN \,\delta_{i,p-1}\delta_{m,p-1}\delta_{n,0}
 \ .
\ee

\newcommand{\w}{\boldsymbol{w}}
\newcommand{\e}{\boldsymbol{e}}

\subsection*{Basis for the center  $Z(\Urest)$} 
Recall that the center of $\Urest$ is $3p-1$ dimensional.
 The basis of $Z(\Urest)$  consists of
 the central idempotents $\e_s$ and nilpotent elements $\w^{\pm}_s$.
 The formulas for these elements in the PBW basis 
 were given  in
 \cite{Gainutdinov:2007tc}:\,\footnote{We used here a relation with Radford basis in the center: the formulas are extracted from Section 3.2.7,  Propositions C.4 and C.5.1 in~\cite{Gainutdinov:2007tc}.}
 \begin{align}\label{eq:center-el}
 \w^+_s &=     
 \zeta_s \sum_{n=0}^{s - 1}
    \sum_{i=0}^{n}
    \sum_{j=0}^{2p - 1}\! ([i]!)^2
    q^{j(s - 1 - 2n)}\!\qbin{s - n + i - 1}{i}\!\qbin{n}{i}
    F^{p - 1 - i} E^{p - 1 - i} K^j,\\
  \w^-_s &=   
  \zeta_s\sum_{n=0}^{p-s - 1}
      \sum_{i=0}^{n}
    \sum_{j=0}^{2p - 1}\!(- 1)^{ i + j} ([i]!)^2
    q^{j(p-s - 1 - 2n)}\!\qbin{p-s - n + i - 1}{i}\!\qbin{n}{i}
    F^{p - 1 - i} E^{p - 1 - i} K^j,\nonumber\\
  \e_0 &=
  \zeta_0
  \sum_{n=0}^{p-1}
      \sum_{i=0}^{n}
    \sum_{j=0}^{2p - 1}\!(- 1)^{ i + j} ([i]!)^2
    q^{j(p - 1 - 2n)}\!\qbin{p - n + i - 1}{i}\!\qbin{n}{i}
    F^{p - 1 - i} E^{p - 1 - i} K^j, \nonumber \\
  \e_p &=   
  \zeta_p
  \sum_{n=0}^{p - 1}
    \sum_{i=0}^{n}
    \sum_{j=0}^{2p - 1}\! ([i]!)^2
    q^{j(p - 1 - 2n)}\!\qbin{p - n + i - 1}{i}\!\qbin{n}{i}
    F^{p - 1 - i} E^{p - 1 - i} K^j,\nonumber\\
 \e_s &=  \ffrac{q^s+q^{-s}}{[s]^2}(\w^+_s + \w^-_s)\nonumber \\
 & + 
\zeta_s
   \sum_{m=0}^{p-2} \sum_{j=0}^{2p-1}\Bigl(\sum_{n=0}^{s-1}
 q^{j(s-1-2n)} \Bnm^+_{n,p-1-m}(s)\nonumber
 + \sum_{k=0}^{p-s-1}
 q^{j(-s-1-2k)} \Bnm^-_{k,p-1-m}(p-s)\Bigr)F^m E^m K^{j}, \nonumber
 \end{align}
  where  $\Bnm^{\pm}_{n,m}$ are non-zero numbers and we set
\begin{align}
\zeta_s &
%= \ffrac{\zeta}{\omega_s} 
= \ffrac{(-1)^{p-s-1}}{2p} \ffrac{[s]^2}{([p-1]!)^2}\gc \qquad 1\leq s \leq p-1\gc \\\
\zeta_0 &
%=  \ffrac{ (-1)^{p-1}\zeta}{p\sqrt{2p}} 
= \ffrac{(-1)^{p-1}}{2p} \ffrac{1}{([p-1]!)^2} \gc \qquad
\zeta_p 
%=   \ffrac{\zeta}{p\sqrt{2p}} 
= \ffrac{1}{2p} \ffrac{1}{([p-1]!)^2}\gp \nonumber
\end{align}

The symmetrised integral from~\eqref{eq:mug} has the following values on the central basis elements~\eqref{eq:center-el}:
  \begin{align}\label{eq:rintg-center}
  \rintg( \w^{+}_s) &= s \coeffN \zeta_s\gc \qquad  \rintg( \w^{-}_s) = (p-s) \coeffN \zeta_s\gc
  %\quad  \rintg( \w^{\pm}_s x^{+}_s) = 0,
   \\
       \rintg( \e_s) &= (-1)^{s}p  \coeffN (q^s +q^{-s})
    %\ffrac{(-1)^{p-1}\zeta}{p\sqrt{2p}}
    \zeta_0\gc\notag\\
   \rintg(\e_p) &= p\coeffN 
   %\ffrac{\zeta}{p\sqrt{2p}}
   \zeta_p\gc
   \qquad   \rintg( \e_0) =p\coeffN 
   %(-1)^{p+1}\ffrac{\zeta}{p\sqrt{2p}}
   \zeta_0\gp\notag
  \end{align}
  
\subsection*{Extension of $\rint_\pivot$ to $\Upmod$}
Here, we compute the modified trace\footnote{We recall that by Theorem~\ref{thm:qg} it is both right and left.} on endomorphisms of indecomposable projective $\Urest$-modules.
 We recall now our result in Theorem~\ref{thm:main} on the modified trace $\t$, and also note that for evaluating  $\t_P$ on endomorphisms $f$  of $P$ it is enough to consider only corresponding trace classes $[f]$. For this, we will also recall a basis in
 $$
 \HH:=\HH\bigl(\Upmod\bigr)\gp
 $$

Indecomposable projective $\Urest$-modules are classified up to isomorphism in \cite{FGST2}: they are precisely the
 projective covers $\PP^\pm_s$ of the simple modules of highest weights $\pm q^{s-1}$ 
 where  $1\leq s\leq p$.   In particular, $\PP^\pm_p$ is a simple module with highest weight $\pm q^{\rt-1}$. The module $\PP^+_1$ is the projective cover of the trivial one. 
 The non-trivial morphisms between indecomposable projective modules are listed below: 
  \begin{itemize}
  \item    
  the endomorphism ring $\End_{\Urest}(\PP^\pm_s)$ is one dimensional  for $s=p$ and  two dimensional with basis $\{\Id_{\PP^\pm_s},x_s^{\pm}\}$,
 for $1\leq s \leq p-1$,
 \item the $\Hom$-spaces $ \Hom_{\Urest}(\Proj_s^+,\Proj_{\rt-s}^{-})$ and $ \Hom_{\Urest}(\Proj_{s}^-,\Proj_{\rt-s}^{+})$
 are two dimensional with respective bases $\{a_s^+, b_s^+\}$ and $\{a_s^-, b_s^-\}$, for $1\leq s \leq p-1$.
\end{itemize} 
It is proven in \cite{BBG}, that
the images of $x_s^\epsilon=b_{p-s}^{-\epsilon}a_s^\epsilon$
and 
 $ x_{p-s}^{-\epsilon}=a_s^\epsilon b_{p-s}^{-\epsilon}$
 in $\HH$ coincide, i.e.
 $[x_s^\epsilon]=[x_{p-s}^{-\epsilon}]$ for any $1\leq s\leq p-1$.
 A basis of $\HH$ consists of trace classes of identities of indecomposable projectives 
 $[\id_{\PP^\pm_s}]$, $1\leq s\leq p$, and trace classes of nilpotent elements $[x_s^+]$, $1\leq s\leq p-1$.
 
 \newcommand{\prim}{I}
In order to compute the modified trace $\t$  on the above basis in $\HH$, we need primitive idempotents.
  Let us first define the projectors onto $q^n$-eigenspace of $K$:
  \be
  \pi_n = \ffrac{1}{2p} \sum_{j=0}^{2p-1} q^{-nj} K^j .
  \ee
  The primitive (non-central) idempotents are then
  \be
  \prim_{n,s} = \pi_n \e_s,\qquad n\in\mathbb{Z}_{2p},\quad 1\leq s\leq p-1, \quad n-s=1\,\mathrm{mod}\,2 ,
  \ee
  and the indecomposable projective modules can be constructed as $\PP_s^+ = U I_{s-1,s}$ and $\PP_{p-s}^- = U I_{-s-1,s}$, where $U$ is the regular representation of $\Urest$.
  Finally,   $x^{+}_s$ and $x^{-}_{p-s}$ equal to the actions of  $\w^+_s$ on $\PP_s^+$ and $\w^-_s$ on $\PP_{p-s}^-$, respectively, so that we have
  \be
  \t_{\PP_s^+}(x^+_s) = \rint_\pivot(I_{ s-1,s} \, \w^+_s)\ , \qquad   \t_{\PP_{p-s}^-}(x^-_{p-s}) = \rint_\pivot(I_{ -s-1,s} \, \w^-_s)\ .
  \ee
Recall Remark~\ref{rem:indP-ext} explaining how to express a modified trace on an indecomposable projective via the modified trace on the regular representation given by the symmetric form $\rintg$.
 Inserting the primitive idempotents $\prim_{s-1,s}$ into the arguments of $\rintg$ in~\eqref{eq:rintg-center}, we get
 \begin{align}
  \rintg( I_{s-1,s}\w^{+}_s) &= \coeffN 
  %\ffrac{\zeta}{\omega_s} 
  \zeta_s\gc
  \qquad  \rintg(I_{-s-1,s} \w^{-}_s) = \coeffN 
  %\ffrac{\zeta}{\omega_s} 
  \zeta_s\gc 
   \\
       \qquad  \rintg(I_{s-1,s} \e_s) &=
    % \coeffN\zeta(-1)^{p-s-1} \ffrac{q^s +q^{-s}}{p\sqrt{2p}}.
     \coeffN(-1)^{s}(q^s +q^{-s}) \zeta_0\gc\notag\\
   \rintg(\pi_{p-1}\e_p) &=\coeffN 
   %\ffrac{\zeta}{p\sqrt{2p}},
   \zeta_p\gc
   \qquad   \rintg(\pi_{2p-1} \e_0) = \coeffN
   %\zeta\ffrac{(-1)^{p+1}}{p\sqrt{2p}}
   \zeta_0\gp\notag
  \end{align}
This gives the following  values for modified trace on our basis in $\HH$: 
$$
\begin{array}{|c||c|c|c|c|c|}
\hline
&[\id_{\PP^+_p}]&[\id_{\PP^-_p}]&[x_s^+]=[x_{p-s}^-]&[\id_{\PP^+_s}]&[\id_{\PP^-_{p-s}}]\\
\hline
\t&\coeffN\zeta_p& \coeffN\zeta_0 &
\coeffN\zeta_s
&\coeffN(-1)^{s} (q^s +q^{-s})\zeta_0&\coeffN(-1)^{s} (q^s +q^{-s})\zeta_0
\\ \hline
%\coeffN\zeta=(-1)^{p-1}p\sqrt{2p}
\t \; \text{for} \; \coeffN =\zeta_0^{-1}
 &(-1)^{p-1}& 1& (-1)^s[s]^2&
(-1)^s (q^s +q^{-s})&(-1)^s (q^s +q^{-s})
\\ \hline
\end{array}
$$
where the second row is normalisation free, while  the third row recovers the results of  \cite{BBG} with the normalisation  choice
%$\coeffN\zeta=(-1)^{p-1}p\sqrt{2p}$.
 $\coeffN= \zeta_0^{-1} = (-1)^{p-1}2p ([p-1]!)^2$.

 \appendix
 \section{Proof of Proposition \ref{Meq}}\label{appA}
 From the definitions of $\HH(A)$ and $\HH(\Apmod)$
  the map $x\mapsto r_x$  induces 
  a linear map $\Phi\colon \HH(A)\rightarrow \HH(\Apmod)$ on the corresponding classes.
We need to construct its inverse.
By Lemma~\ref{lem:decomp-id}, for $P\in \Apmod$  we have a decomposition:
\be\label{eq:dec}
\id_P=\sum_{i=1}^k a_i\circ\id_A\circ b_i\ ,
\quad{\text{with}}\quad b_i\colon P\to A, \;a_i\colon A\to P\gp
\ee
Let us define a map $\psi_P\colon\End_A(P)\rightarrow \HH(A)$
by 
\be
\psi_P(f):= \sum_i \left[(b_i\circ f\circ a_i)(\one)\right]\gp
\ee 
We will check that the map
\begin{align}\label{bubu}
\Psi\colon\; \HH(\Apmod)&\xrightarrow{\sim}\HH (A)\\
[P,f]&\mapsto \psi_P(f) \nonumber
\end{align}
is well-defined, i.e.\ it does not depend on the choice of the decomposition
\eqref{eq:dec} and 
descends  on the class of  $f$ in $\HH(\Apmod)$.

Assume we have another  decomposition $\id_P=\sum_{i'}a'_{i'}\circ\id_A \circ b'_{i'}$, with the associated map 
$$
\psi'_P(f)=\sum_{i'}[(b'_{i'}\circ f\circ a'_{i'})(\one)]\gp
$$
 Inserting the identity~\eqref{eq:dec}, we have 
\begin{eqnarray}
\psi'_P(f)&=&\sum_{i,i'} \left[(b'_{i'}\circ f\circ a_i\circ b_i\circ a'_{i'})(\one)\right] \label{E:antihom1}\\
&=&\sum_{i,i'} \left[( b_i\circ a'_{i'})(\one)\,(b'_{i'}\circ f\circ a_i)(\one)\right]\notag
\end{eqnarray}
where we applied the algebra isomorphism from Lemma \ref{lem:Uop-End} to the composition
of  $A$-en\-do\-mor\-phisms $(b'_{i'}\circ f\circ a_i)$ and  $(b_i\circ a'_{i'})$.
Similarly,
\begin{eqnarray}
\psi_P(f)&=&\sum_{i,i'} \left[(b_{i}\circ  a'_{i'}\circ b'_{i'}\circ f\circ a_{i})(\one)\right]\label{E:antihom}\\ 
&=&\sum_{i,i'} \left[(b'_{i'}\circ f\circ a_i)(\one)\, ( b_i\circ a'_{i'})(\one)\right]\notag
\end{eqnarray}
which is equal to the second line in~\eqref{E:antihom1} because the summands are classes in $\HH(A)$. We thus
get the equality $\psi'_P(f)=\psi_P(f)\in \HH(A)$.

Let us now show that the family 
$$
\{\psi_P\colon \End_A(P)\rightarrow \HH(A)\, | \, P\in \Apmod\}
$$ 
has cyclicity property. 
 Let $f\colon P\rightarrow P'$ and $g\colon P'\rightarrow P$, and 
$\id_P$ as in~\eqref{eq:dec} and let $\id_{P'}=\sum_{i'}a'_{i'}\circ\id_A \circ b'_{i'}$.
 We then have
\begin{eqnarray}
\psi_{P'}(f\circ g)&=&\sum_{i,i'} \left[(b'_{i'}\circ f\circ a_i\circ \id_A\circ b_i\circ g\circ a'_{i'})(\one)\right]\\
&=&\sum_{i,i'} \left[( b_i\circ g\circ a'_{i'})(\one)\, (b'_{i'}\circ f\circ a_i)(\one)\right]\notag\\
&=&\sum_{i,i'} \left[(b'_{i'}\circ f\circ a_i)(\one)\,( b_i\circ g \circ a'_{i'})(\one)\right]\notag\\
&=&\sum_{i,i'} \left[( b_i\circ g\circ a'_{i'}\circ b'_{i'}\circ f\circ a_i)(\one)\right]\notag\\
&=&\psi_P(g\circ f)\notag
\end{eqnarray} 
where we again used the algebra isomorphism in Lemma \ref{lem:Uop-End}.
From this cyclicity property, we see that the map $\psi_P$ does not depend on representatives $f$ in the class $[f]\in\HH(\Apmod)$, for $f\in\End_A(P)$. Therefore, the map $\Psi$ in~\eqref{bubu} is well-defined.

To see that $\Psi \circ \Phi=\id_{\HH(A)}$ we have to check that the composition $[x]\mapsto [r_x]\mapsto \psi_A(r_x)$ is identity. Note that here we use only $P=A$ component in the quotient~\eqref{eq:HH0-C-def}. Using  the trivial decomposition of $\id_A$ from~\eqref{eq:dec}, we  indeed get the expected identity, and so $\Psi$ is a left inverse of $\Phi$. 

To show that $\Psi$ is also a right inverse of $\Phi$, assume $P\in \Apmod$ and  $f\in \End_A(P)$.
Then $\Psi$ maps $[P,f]$ to the class of 
$x=\sum_i (b_i\circ f\circ a_i)(\one)\in A$. We note that the corresponding endomorphism of $A$ by right multiplication with $x$ is
$r_x=\sum_i (b_i\circ f\circ a_i)$. And by cyclicity we have 
$[r_x]=[f]\in \HH(\Apmod)$. We thus get $\Phi \circ \Psi=\id_{\HH(\Apmod)}$, which completes the proof of the proposition.

\section{Proof of Lemma \ref{lem:A.2}}\label{appB}
The fact that $w(\alpha_i)$ is a positive root if  $l(ws_i)=l(w)+1$ follows from \cite[VI.1.6, Cor.\,2]{Bourbaki}.
 We will prove the formula for $\wt(T_w(E_i))$ by induction on the length $l(w)=\nu\geq 0$.
 A proof for  $\wt(T_w(F_i))$ works similarly.
For $\nu=0$, $w$ is the unit element and the statement holds by definition of the $L$-grading. 
 We suppose that the statement holds for $\nu\geq 0$, i.e.\ 
  that $T_w(E_i)$ has $L$-grading $\wt\bigl(T_w(E_i)\bigr)=w(\alpha_i)$ if $l(w)\leq \nu$ and $l(ws_i)=l(w)+1$.
 
Let $w\in \weyl$ be an element with length \mbox{$l(w)=\nu+1$} and $i$ be such that
 $l(ws_i)=\nu+2$.
Recall that $w(\alpha_i)
\in \Delta_+$.
 We claim that  there exists $j\neq i$ such that
$w(\alpha_j)$ is a negative root. This follows from~\cite[Sec.\,V.4.4, Thm\,1]{Bourbaki}, indeed if $w$ permutes the positive roots, then $w$
fixes the positive chamber $C=\{x\in L \;|\;(\alpha_i|x)>0, 1\leq i\leq n\}$ and hence is identity.
Let us choose such $j$. Recall that $l(ws_j)=l(w)+1$ would imply that $w(\alpha_j)$ is a positive root, hence we have that $l(ws_j)<\nu+2$. From the defining relations, multiplication with $s_j$ changes the length by $\pm1$, we then  clearly  have $l(ws_j)\ne l(w)$,  therefore $l(ws_j)=\nu$.
 Denote by $\langle s_i,s_j\rangle\subset \weyl$ the subgroup generated  by $s_i$ and $s_j$.
The idea is to use elements from the 
orbit $w\langle s_i,s_j\rangle$ 
to construct an appropriate pair $(w',k)$ to which the induction hypothesis applies.
For a given choice of $j$ above, we have 3 cases: $a_{ij}=0$ or if $a_{ij}=-1$ then  $w s_j s_i$ might have length $\nu\pm 1$. We analyse all of these cases:

\textbf{Case 1:} $a_{ij}=0$. We can choose $(w',k)=(w s_j,i)$. Indeed,   $l(w')=\nu$ and since $l(w s_i)=\nu+2$ then $w's_i=ws_is_j$ has length $\nu+1$, and so we can apply the induction hypothesis. 
We then get $T_w(E_i)=(T_{w'}\circ T_j)(E_i)=T_{w'}(E_i)$ because $T_j(E_i)=E_i$, see~\eqref{eq:T-act}. Using that $s_j(\alpha_i)=\alpha_i$ we get $\wt\bigl(T_w(E_i)\bigr)=w'(\alpha_i)=w(\alpha_i)$.

\textbf{Case 2a:} $a_{ij}=-1$ and  $l(w s_j s_i)=\nu+1$. We choose
$w'=ws_j$ and  to both $(w',i)$, $(w',j)$  the induction hypothesis applies.
We have  $T_j(E_i)=-E_iE_j+q^{-1}E_jE_i$, $s_j(\alpha_i)=\alpha_i+\alpha_j$, hence
 \begin{align}
 \wt(T_w(E_i))&=\wt(T_{w'}\circ T_j(E_i))= \wt(T_{w'}(E_i))+\wt(T_{w'}(E_j))\\
 &=w'(\alpha_i)+w'(\alpha_j)=(w'\circ s_j)(\alpha_i)= w(\alpha_i)\ ,\notag
 \end{align}
 where we used that $T_{w'}$ is an automorphism of the algebra and that $\wt$ makes the algebra graded.
 
\textbf{Case 2b:} $a_{ij}=-1$ and $l(w s_j s_i)=\nu-1$. We choose $w'=w s_j s_i$ and check that $l(w's_j)= l(w s_i s_j s_i) =\nu$ because on one side it is at most $\nu$ and on the other side it is at least $l(w s_i)-l(s_j s_i)=\nu$. Therefore, we can apply the induction hypothesis to $(w',j)$. We have
$(T_i\circ T_j)(E_i)=E_j$ and $(s_js_i)(\alpha_j)=\alpha_i$, hence
$$\wt(T_w(E_i))=\wt(T_{w'}(E_j))=w'(\alpha_j)=w(\alpha_i)\ .$$
This finishes the proof.

\section{Proof of Lemma \ref{lemma:commut}}\label{appC}
We will use the following result \cite[Thm.\,2.3]{Xi}\footnote{We note that  in \cite[Thm.\,2.3]{Xi} a commutation formula is given for divided powers, and we just rewrite it for our choice of powers of $E_\beta$.}
stated for  $1\leq j\leq k$, and $1\leq a,b\leq p-1$:
\be \label{thm:xi}
E_{\beta_k}^aE_{\beta_j}^b=q^{ab(\beta_j|\beta_k)} E_{\beta_j}^bE_{\beta_k}^a +
\sum_{0\leq a_j,a_{j+1},\dots,a_k\leq p-1\atop
a_j<b\ ,\ a_k<a}
\rho(a_j,\dots,a_k)
 E_{\beta_j}^{a_j}E_{\beta_{j+1}}^{a_{j+1}}\dots E_{\beta_k}^{a_k}
\ee
where the coefficients $\rho(a_j,\dots,a_k)\in\kk$
vanish if  the corresponding monomials do not have the expected $L$-grading:
\be \label{coeff_vanishing}
\rho(a_j,\dots,a_k)=0\quad \text{ if } \quad a_j\beta_j+a_{j+1}\beta_{j+1}+\dots+a_k\beta_k\neq b\beta_j+a\beta_k \gp
\ee
We prove the lemma by induction on $\nu=k-j$.

Let us consider the case $\nu=1$.
 The formula~\eqref{thm:xi} gives
\be\label{eq:EE-com}
 E_{\beta_{j+1}}^{p-1}E_{\beta_j}=q^{(p-1)(\beta_j|\beta_{j+1})} E_{\beta_j}E_{\beta_{j+1}}^{p-1} \gc
\ee
 where we used  that the second term in~\eqref{thm:xi} vanishes because of the condition~\eqref{coeff_vanishing}, which is in our case
\be\label{eq:ineq-ab}
a_{j+1}\beta_{j+1}\ne \beta_j+(p-1)\beta_{j+1}\gc
\ee
holds for all $a_{j+1}< p-1$.
Equality~\eqref{eq:EE-com} shows that~\eqref{lem:commut} is true for $k-j=1$.

Assume the induction hypothesis that for $1\leq \nu<N$ the formula \eqref{lem:commut} is true if $k-j\leq\nu$.
We consider the case where $k-j=\nu+1$.
From \eqref{thm:xi}, we get
\be\label{proof_commut}
E_{\beta_k}^{p-1}E_{\beta_j}=q^{(p-1)(\beta_j|\beta_k)} E_{\beta_j}E_{\beta_{k}}^{p-1} +
\sum_{0\leq a_{j+1},\dots,a_k\leq p-1\atop
a_k<p-1}
\rho(0,a_{j+1},\dots,a_k)
 E_{\beta_{j+1}}^{a_{j+1}}\dots E_{\beta_k}^{a_k}\gp
\ee
We then use the condition~\eqref{coeff_vanishing} on vanishing coefficients $\rho(0,a_{j+1},\dots,a_k)$, which is in our case
$$
a_{j+1}\beta_{j+1}+\dots+a_k\beta_k\ne \beta_j+(p-1)\beta_k\gp
$$
 We see that it certainly  holds if all the integers $a_{j+1}, \dots, a_{k-1}$ are  zero -- in this case we get the inequality $a_k\beta_k\ne \beta_j+(p-1)\beta_k$, similar to~\eqref{eq:ineq-ab}. Therefore, for  {\sl non}-vanishing coefficients $\rho$ in the sum~\eqref{proof_commut} we have to necessarily assume that  at least one of the integers $a_{j+1}, \dots, a_{k-1}$ is non zero.
 Let $l$ be the smallest  index for which $a_l$ is non zero.
 We have $j+1\leq l<k$ hence $|k-l|<\nu$. The induction hypothesis 
  gives us commutation relation for the root vector $E_{\beta_l}$, and we get
  \be E_{\beta_{l}}^{p-1}E_{\beta_{l+1}}^{p-1}\dots E_{\beta_{k-1}}^{p-1}E_{\beta_l}= q^{(p-1)(\beta_l|\beta_{l+1}+\dots+\beta_{k-1})} E_{\beta_l}^{p-1} E_{\beta_l}E_{\beta_{l+1}}^{p-1}\dots E_{\beta_{k-1}}^{p-1}=0\ . \ee
 This gives the following vanishing result for terms in the sum~\eqref{proof_commut} corresponding to non-zero coefficients $\rho(0,a_{j+1},\dots,a_k)$:
\be
E_{\beta_{j+1}}^{p-1}E_{\beta_{j+2}}^{p-1}\dots E_{\beta_{k-1}}^{p-1}E_{\beta_l}=0\ ,
\notag\ee
and therefore these terms do not contribute while moving $E_{\beta_j}$ to the left in LHS of~\eqref{lem:commut}.
We have thus obtained
\be 
E_{\beta_{j+1}}^{p-1}E_{\beta_{j+2}}^{p-1}\dots E_{\beta_{k}}^{p-1}E_{\beta_j}=
q^{(p-1)(\beta_j|\beta_k)} E_{\beta_{j+1}}^{p-1}E_{\beta_{j+2}}^{p-1}\dots E_{\beta_{k-1}}^{p-1} E_{\beta_j}E_{\beta_{k}}^{p-1}\ .
\notag \ee
Using again the induction hypothesis, we move $E_{\beta_j}$ to the left using~\eqref{thm:xi} and get the expected formula \eqref{lem:commut}, which completes the proof.

\newcommand\arxiv[2]      {\href{http://arXiv.org/abs/#1}{#2}}
\newcommand\doi[2]        {\href{http://dx.doi.org/#1}{#2}}
\newcommand\httpurl[2]    {\href{http://#1}{#2}}

\end{document}